\documentclass[11pt, twoside, leqno]{article}

\usepackage{amssymb}
\usepackage{amsfonts}
\usepackage{amsmath}
\usepackage{amsthm}
\usepackage{color}
\usepackage{mathrsfs}
\usepackage{txfonts}
\usepackage{bbm}

\usepackage{indentfirst}

\allowdisplaybreaks

\def\red{\color{red}}

\def\oz{{\omega}}
\def\gfz{\genfrac{}{}{0pt}{}}

\def\rr{{\mathbb R}}
\def\rn{{\mathbb{R}^n}}
\def\zz{{\mathbb Z}}
\def\cc{{\mathbb C}}
\def\BB{{\mathbb B}}
\def\nn{{\mathbb N}}
\def\cf{{\mathcal F}}
\def\cg{{\mathcal G}}
\def\cm{{\mathcal M}}
\def\cp{{\mathcal P}}
\def\RR{{\mathcal R}}
\def\SS{{\mathcal S}}
\def\EE{{\mathcal E}}
\def\uu{{\mathcal U}}
\def\vv{{\mathcal V}}
\def\MM{{\mathscr M}}

\def\dis{\displaystyle}

\def\f{\frac}
\def\lf{\left}

\def\ls{\lesssim}

\def\pa{\partial}

\def\wt{\widetilde}

\def\com{\complement}

\def\gfz{\genfrac{}{}{0pt}{}}

\def\dist{\mathop\mathrm{\,dist\,}}
\def\supp{\mathop\mathrm{\,supp\,}}

\def\mEz{(\dot{E}^{\vec{\alpha}^{(0)},\vec{p}^{(0)}}
_{\vec{q}^{(0)}}(\rr^{n}),l^{\infty})}

\def\mEo{(\dot{E}^{\vec{\alpha}^{(1)},\vec{p}^{(1)}}
_{\vec{q}^{(1)}}(\rr^{n}),l^{\infty})}
\def\mE{(\dot{E}^{\vec{\alpha},\vec{p}}_{\vec{q}}
(\rr^{n}),l^{\infty})}
\def\dmE{(\dot{E}^{-\vec{\alpha},\vec{p}'}_{\vec{q}'}
(\rr^{n}),l^{1})}
\def\nEz{\dot{E}^{\vec{\beta}^{(0)},\vec{s}^{(0)}}
_{\vec{t}^{(0)}}(\rn)}

\def\nEo{\dot{E}^{\vec{\beta}^{(1)},\vec{s}^{(1)}}
_{\vec{t}^{(1)}}(\rn)}
\def\nE{\dot{E}^{\vec{\beta},\vec{s}}_{\vec{t}}(\rn)}
\def\one{\mathbf{1}}
\def\lq{L^{\vec{q}}}
\def\lp{L^{\vec{p}}}
\def\ohz{\dot{K}^{\alpha,p}_{q}(\rn)}

\def\iihz{\dot{E}^{\vec{\alpha},\vec{p}}_{\vec{q}}}
\def\riihz{\dot{E}^{\vec{\alpha}/r,r\vec{p}}_{r\vec{q}}}

\def\diihz{\dot{E}^{-\vec{\alpha},\vec{p}'}
_{\vec{q}'}}

\def\ky{\dot{K}^{\alpha_{1},p_{1}}_{q_{1}}}
\def\rky{\dot{K}^{\alpha_{1}/r,rp_{1}}_{rq_{1}}}

\def\dky{\dot{K}^{-\alpha_{1},p_{1}'}_{q_{1}'}}
\def\ke{\dot{K}^{\alpha_{2},p_{2}}_{q_{2}}}
\def\ki{\dot{K}^{\alpha_{i},p_{i}}_{q_{i}}}

\def\kn{\dot{K}^{\alpha_{n},p_{n}}_{q_{n}}}

\def\nuiihz{\dot{E}^{\vec{\alpha}_{n-1},\vec{p}
_{n-1}}_{\vec{q}_{n-1}}(\rr^{n-1})}

\newtheorem{theorem}{Theorem}[section]
\newtheorem{lemma}[theorem]{Lemma}
\newtheorem{corollary}[theorem]{Corollary}
\newtheorem{proposition}[theorem]{Proposition}

\theoremstyle{definition}
\newtheorem{remark}[theorem]{Remark}
\newtheorem{definition}[theorem]{Definition}

\renewcommand{\appendix}{\par
\setcounter{section}{0}%
\setcounter{subsection}{0}%
\setcounter{subsubsection}{0}%
\gdef\thesection{\@Alph\c@section}%
\gdef\thesubsection{\@Alph\c@section.\@arabic\c@subsection}%
\gdef\theHsection{\@Alph\c@section.}%
\gdef\theHsubsection{\@Alph\c@section.\@arabic\c@subsection}%
\csname appendixmore\endcsname
}

\numberwithin{equation}{section}

\begin{document}

\arraycolsep=1pt

\title{\bf\Large Mixed-Norm Herz Spaces
and Their Applications in Related Hardy Spaces
\footnotetext{\hspace{-0.35cm} 2020 {\it
Mathematics Subject Classification}. Primary 42B35;
Secondary 42B25, 42B30, 46E30.
\endgraf {\it Key words and phrases.}
Mixed-norm Herz spaces, Riesz--Thorin interpolation,
ball quasi-Banach function space, maximal function,
hardy space.
\endgraf This project is supported
by the National Natural Science Foundation of China
(Grant Nos.\ 11971058 and 12071197) and
the National Key Research and Development Program of
China (Grant No.\ 2020YFA0712900).}}
\author{Yirui Zhao, Dachun Yang\footnote{Corresponding author,
E-mail: \texttt{dcyang@bnu.edu.cn}/{\red April 26,
2022}/Final version.}\ \
and Yangyang Zhang}
\date{}
\maketitle

\vspace{-0.8cm}

\begin{center}
\begin{minipage}{13cm}
{\small {\bf Abstract}\quad
In this article, the authors introduce a
class of mixed-norm Herz spaces,
$\dot{E}^{\vec{\alpha},\vec{p}}_{\vec{q}}(\mathbb{R}^{n})$,
which is a natural generalization of mixed Lebesgue spaces and
some special cases of which naturally appear in the study of the
summability of Fourier transforms on mixed-norm
Lebesgue spaces.
The authors also give their dual spaces and
obtain the Riesz--Thorin interpolation theorem
on $\dot{E}^{\vec{\alpha},\vec{p}}_{\vec{q}}(\mathbb{R}^{n})$.
Applying these Riesz--Thorin interpolation
theorem and using some ideas from the extrapolation theorem, the
authors establish both the boundedness of
the Hardy--Littlewood maximal operator and
the Fefferman--Stein vector-valued maximal
inequality on $\dot{E}^{\vec{\alpha},\vec{p}}_{\vec{q}}(\mathbb{R}^{n})$.
As applications, the authors develop various real-variable theory of Hardy spaces
associated with $\dot{E}^{\vec{\alpha},\vec{p}}_{\vec{q}}(\mathbb{R}^{n})$ by
using the existing results of Hardy spaces associated with ball
quasi-Banach function spaces. These results strongly depend on the duality
of $\dot{E}^{\vec{\alpha},\vec{p}}_{\vec{q}}(\mathbb{R}^{n})$
and the non-trivial constructions
of auxiliary functions in the Riesz--Thorin
interpolation theorem.
}
\end{minipage}
\end{center}

\vspace{0.2cm}

\tableofcontents

\vspace{0.2cm}

\section{Introduction\label{s0}}

In 1968, the classical Herz space was originally
introduced by Herz
\cite{Hz} to study the Bernstein theorem on
absolutely convergent Fourier transforms, while the
study on Herz spaces can be traced back to the work of Beurling
\cite{casca} in 1964. Indeed, a special Herz space $A^{p}(\rn)$,
with $p\in(1,\infty)$, which is also called Beurling algebra, was
originally introduced by Beurling \cite{casca} to study some convolution
algebras. In 1985, Baernstein and Sawyer \cite{bs} generalized these Herz
spaces and presented lots of applications related to
the classical Hardy spaces. Later, in 1987, to study the Wiener third Tauberian
theorem on $\rn$, Feichtinger \cite{hgf} introduced another
norm of $A^{p}(\rn)$, which is obviously equivalent to the norm defined	
by Beurling \cite{casca}.  From then on, Herz
spaces have been widely explored in the fields
of harmonic analysis and partial differential equations. Indeed, Herz-type spaces
prove useful in the study related to partial differential equations.
For instance, the variable Herz spaces proved to be the key tools in
the study of the regularity of solutions to elliptic equations by
Scapellato \cite{s19}, while Drihem \cite{d22,d22-2} investigated semilinear
parabolic equations with the initial data in Herz spaces
or Herz-type Triebel--Lizorkin spaces. Also, the Fourier--Herz space
is one of the most suitable spaces to study the global stability
for fractional Navier--Stokes equations;
see, for instance, \cite{cs18,c18,lyhh21,nz18}.
Recently, Herz spaces prove crucial in the study of Hardy spaces
associated with ball quasi-Banach function spaces in Sawano et al. \cite{SHYY}.
Indeed, Sawano et al.
\cite{SHYY} introduced the ball quasi-Banach function space
$X$ and the associated Hardy space $H_{X}(\rn)$ via the grand
maximal function. Using a certain inhomogeneous Herz space,
Sawano et al. \cite{SHYY} ingeniously overcame the difficulty
appearing in the proof of the convergence of the atomic
decomposition of $H_{X}(\rn)$. For more progress on Herz spaces,
we refer the reader to \cite{ly96,ly96-2,HY,lyy,yk,yh} for classical Herz spaces,
to \cite{ss,wx,yx} for variable Herz spaces, and also
to \cite{Ho1,Ho2,mrz2,Is,ly95,ly95-2,ly95-3,ly97-2} for
more applications of Herz spaces.

As a good substitute of Herz spaces, Herz-type Hardy spaces
have been found plenty of applications in many branches of mathematics,
such as harmonic analysis and partial differential equations,
which have been systematically studied and developed so far; see, for instance,
\cite{gly,mrz,xy1,xy2,ds16,dh19}. In 1989, the Hardy spaces associated with the
Beurling algebras on the real line were first introduced by Chen and Lau \cite{cl89}.
Later, Grac\'ia-Cuerva \cite{g89} in 1989 generalized the results of Chen and Lau \cite{cl89}
to the higher-dimensional case,  and  Grac\'ia-Cuerva and Herrero \cite{gh94} in 1994
further established various maximal function, atomic, and Littlewood--Paley function
characterizations of these Herz--Hardy spaces. On the other hand,
since 1990, Lu and Yang made a series of studies on the Hardy spaces associated with
the Beurling algebras or with Herz spaces (see, for instance, the monograph \cite{HZ}
and its references). Particularly, in 1992, both the Littlewood--Paley
function and the $\phi$-transform characterizations of the Herz--Hardy
space $HK_{2}(\rn)$ were investigated by Lu and Yang \cite{ly92} and then,
in 1995, they obtained the atomic and the molecular characterizations of
the Herz--Hardy space with more general indices in \cite{ly95-5}. Furthermore,
in 1997, Lu and Yang \cite{ly97} established various maximal function
characterizations of weighted Herz--Hardy spaces.
Meanwhile, Herz--Hardy spaces play an important role in the study of
oscillatory singular integrals. In particular, Lu and Yang \cite{ly95-4}
established the boundedness of some oscillatory singular integrals from
Herz--Hardy spaces to Herz spaces. Since Pan \cite{p91}
pointed out that oscillatory singular integrals may not be
bounded from the classical Hardy space $H^{1}(\rn)$ to the
Lebesgue space $L^{1}(\rn)$, the results of \cite{ly95-4}
showed that the Herz--Hardy space is indeed a proper substitution
of $H^{1}(\rn)$ in the study of oscillatory singular integrals.
We refer the reader to \cite{ly95,hwy96,hy98,dh19,lyh} for various Herz--Hardy spaces
and to \cite{ds16,mrz} for the Herz--Hardy spaces of variable
smoothness and integrability.
For more progress on Herz-type spaces,
we refer the reader to \cite{S18,x1,x2,x3,dh17,d13,d19} for
Herz-type Besov spaces and to \cite{ly97scm,X1,X2,X3,dd19,d18} for Herz-type
Triebel--Lizorkin spaces.

Recall that the mixed Lebesgue space $\lp(\rn)$ with
$\vec{p}\in(0,\infty]^{n}$, as a natural generalization
of the classical Lebesgue space $L^{p}(\rn)$ with $p\in (0,\infty]$,  was investigated
by Benedek and Panzone \cite{BP} in 1961. Actually, the
origin of these mixed Lebesgue spaces can be traced back to
H\"{o}rmander \cite{i-1} in 1960 on the
estimates for translation invariant operators.
In \cite{BP}, Benedek and Panzone also established
the Riesz--Thorin interpolation theorem on mixed
Lebesgue spaces. Later, Bagby \cite{bag}
obtained the boundedness of the partial Hardy--Littlewood maximal
operators on mixed Lebesgue spaces via the related Riesz--Thorin
interpolation theorem.
Inspired by the work
of Benedek and Panzone \cite{BP} on mixed Lebesgue
spaces,
numerous other function spaces with mixed norms sprang
up. For instance, Banach function spaces with mixed norms
was considered by Blozinski \cite{Blo}; anisotropic mixed-norm
Hardy spaces and other function spaces were studied by Cleanthous et al.
\cite{CG-2,cgn19,cgn19-2}; mixed norm $\alpha$-modulation spaces
were investigated by Cleanthous and Georgiadis \cite{CG-1}.
We refer the reader to \cite{cgn17,gjn17,jn16,Ho,mxM,nt,sw, Ho21,noss21}
for more study about mixed-norm spaces. Very recently, Weisz
\cite{w8} proved that
the $\theta$-means of some functions converge to themselves at
all their Lebesgue points if and only if the Fourier transform
of $\theta$ belongs to a suitable Herz space, where $\theta\in
L^{1}(\rn)$. Later, Huang et al. \cite{mhz} generalized this result to mixed-norm
Lebesgue spaces via introducing some suitable mixed-norm Herz
spaces $\dot{E}_{\vec{q}}^{*}(\rn)$ with $\vec{q}\in [1,\infty]^{n}$ (see Remark \ref{E*} below).
This means that those mixed-norm Herz spaces $\dot{E}_{\vec{q}}^{*}(\rn)$ are the best choice in
the study of the summability of Fourier transforms on mixed-norm
Lebesgue spaces. We refer the reader to \cite{fw06,fw08,w08,w08-2,w4,w14}
for applications of Herz spaces in the summability of both
Fourier transforms and series, to \cite{sw17, w16,w17,w09,sdh}
for applications of Herz spaces in other convergence problems related to
integral operators, Lebesgue points, and Gabor expansions
and, particularly, to the monograph \cite{w8} of Weisz for a detailed study on
these subjects. However, to the best of our knowledge, \emph{no other
properties of these mixed-norm Herz
spaces and their applications in related
Hardy spaces} are known so far. The main target of this article
is to explore important properties of these mixed-norm Herz
spaces and their applications in related
Hardy spaces.

Let $\vec{\alpha},$ $\vec{\beta}\in\rn$, $\vec{p},$ $\vec{q},$
$\vec{s},$ $\vec{t}\in(0,\infty]^{n}$, and $r\in (0,\infty]$.
In this article, we introduce the mixed-norm Herz spaces
$\iihz(\rn)$ (see Definition \ref{mhz} below) and
to explore their applications in related Hardy spaces.
These mixed-norm Herz spaces are naturally generalizations
of both the mixed Lebesgue spaces and the aforementioned mixed-norm Herz spaces
recently studied by Huang et al. \cite{mhz}.
Particularly, both the dual and the associate spaces of the mixed-norm Herz spaces
are obtained (see Theorem \ref{dii} and Corollary \ref{Eass}
below). One of our main results is the boundedness of
the Hardy--Littlewood maximal operators  on $\iihz(\rn)$.
To achieve it, we need first to investigate the Riesz--Thorin interpolation
theorem associated with the linear operators mapping $\nE$ continuously
into $(\iihz(\rn),l^{\infty})$ (see Theorem \ref{threeL}
below). Observe that the proof of the Riesz--Thorin
interpolation theorem in \cite{BP} strongly depends on
the density of simple functions in mixed Lebesgue
spaces and the construction of the related auxiliary functions.
However, the approach used in \cite{BP} for mixed Lebesgue
spaces is no longer feasible for $\iihz(\rn)$ because
simple functions may even not belong to mixed-norm Herz spaces
under consideration
and hence the auxiliary functions in \cite{BP} are also inapplicable
to mixed-norm Herz spaces. To overcome these obstacles, we use some special simple
functions supported in $A_{m}$ for some $m\in\nn$ (see Lemma
\ref{Am} below for the definition of the subset $A_{m}$ of $\rn$) and,
based on these special simple functions and
the norm structure of mixed-norm Herz spaces, we construct
a suitable auxiliary function satisfying the well-known three
lines theorem [see \eqref{aux} below]. The construction
of this auxiliary function is quite non-trivial and technical.

Moreover, we realize that Hardy spaces associated with
mixed-norm Herz spaces fall into the framework of Hardy spaces
associated with ball quasi-Banach function spaces which were first
introduced by Sawano et al. \cite{SHYY}.
Recall that Sawano et al. \cite{SHYY} established
various real-variable characterizations of the Hardy space
$H_{X}(\rn)$ associated with the ball quasi-Banach function
space $X$, respectively, in terms of atoms,
molecules, and Lusin area functions.
Recently, the real-variable theory of function
spaces associated with ball quasi-Banach function spaces is well
developed; see, for instance, \cite{is17,ins19,zywc,YYY,ft-bqb,Abs,rm-wyy,SHYY,wk-1}
for Hardy spaces associated with ball quasi-Banach function spaces,
\cite{Ho4} for Erd\'elyi--Kober fractional integral operators on ball
Banach function spaces, and \cite{Ho3,sdh,S18} for other function spaces
associated with ball quasi-Banach function spaces. Applying
the existing results of $H_{X}(\rn)$, we establish a
complete real-variable
theory of Hardy spaces associated with mixed-norm Herz spaces.

To the best of our knowledge, since mixed-norm Herz spaces
own finer structure than mixed Lebesgue spaces, more
applications of mixed-norm Herz spaces in harmonic analysis and
partial differential equations are expectable.

The remainder of this article is organized as follows.

Section \ref{s2} is devoted to studying fundamental properties of
mixed-norm Herz spaces. We begin with introducing mixed-norm Herz spaces
$\iihz(\rn)$. In Subsection \ref{s2.1}, we establish both the dual and
the associate spaces of $\iihz(\rn)$.  In Subsection \ref{s2.2}, we
show that mixed-norm Herz spaces are ball
quasi-Banach function spaces (see Proposition \ref{bqb-2} below)
and, via using some ideas from the extrapolation
theorem, we investigate the Fefferman--Stein vector-valued
maximal inequality associated with the maximal operator
$M_{n}$ for the $n$-th variable
(see Definition \ref{itmax} below for its definition)
on the ball quasi-Banach function space $X$ under the
assumption that the Hardy--Littlewood maximal $M_{n}$ is
bounded on the associate space $(X^{1/p})'$ of $X^{1/p}$, where $p\in(1,\infty)$
(see Lemma \ref{ax-Fs} below).

The purpose of Section \ref{s3} is to build the Riesz--Thorin
interpolation theorem on mixed-norm Herz spaces. To achieve this, we first introduce
a special mixed-norm space $(\iihz(\rn),l^{r})$ which is called the
$l^{r}$-Herz type mixed-norm space (see Definition \ref{EL} below),
and study its associate space. Then we construct
two special functions based on simple
functions supported in $A_{m}$ [see \eqref{Fzpsi} and
\eqref{Gzphi} below] and investigate their basic
properties. Using these properties and the well-known
three lines theorem, we then obtain the Riesz--Thorin interpolation
theorem associated with the linear operators mapping $\nE$ continuously into
$(\iihz(\rn),l^{\infty})$.

In Section \ref{s4}, we establish the boundedness of the classical
Hardy--Littlewood maximal operator on the mixed-norm Herz space $\iihz(\rn)$
(see Corollary \ref{HL-2} below). To this end, using both the associate spaces of
$l^{\infty}$-Herz type mixed-norm spaces and the
Riesz--Thorin interpolation theorem built in Section \ref{s3},
we first obtain the boundedness of the maximal
operator $M_{n}$ on $H(\rn)$ [the class of some special simple functions;
see \eqref{Hsim} below]. Via the density of
$H(\rn)$ in $\iihz(\rn)$ (see Proposition \ref{E-dense}
below), we then extend this boundedness to the whole $\iihz(\rn)$.

In Section \ref{s5}, we apply all the results obtained in
Sections \ref{s2}, \ref{s3}, and \ref{s4} to mixed-norm Herz--Hardy spaces $H\iihz(\rn)$
(see Definition \ref{hd-2} below) and obtain their
Littlewood--Paley
characterizations, respectively, in terms of the Lusin
area function,
the Littlewood--Paley $g$-function, and the
Littlewood--Paley
$g_{\lambda}^{*}$-function in Subsection \ref{s5.1}, and the
boundedness
of Calder\'{o}n--Zygmund operators on mixed-norm Herz--Hardy spaces
in Subsection \ref{s5.2}. Moreover, some other real-variable
characterizations are mentioned in Subsection \ref{s5.3}
without to giving the details to limit the length of this article.

Finally, we make some conventions on notation. Let
$\nn:=\{1,2,\ldots\},$ $\zz_{+}:=\nn \cup
\{0\}$, $\zz_{+}^n:=(\zz_{+})^n$, and $\zz$ denote the
set of all integers. We denote
by $C$ a
\emph{positive constant}
which is
independent of the main parameters, but may vary
from line to
line.  We also use $C_{(\alpha,\beta,\ldots)}$ to denote a
positive constant depending on the indicated parameters $\alpha$, $\beta,$\,$\ldots$\,.
The notation $f\ls g$ means $f\leq  Cg$
and, if $f\ls g\ls f$, we then write $f\sim g$. For
any $q\in
[1,\infty]$, we denote by $q'$ its \emph{conjugate
index}, namely,
$1/q+1/q'=1$. Moreover, for any
$\vec{q}:=(q_{1},\ldots,q_{n})\in
[1,\infty]^{n}$, we denote by $\vec{q}':=(q_{1}',\ldots,q_{n}')$
its conjugate vector, namely,
$1/\vec{q}+1/\vec{q}'=1$ which means that $1/q_{i}+1/q_{i}'=1$ for any
$i\in\{1,\ldots,n \}$. For any $x\in\rn$,
we denote by $|x|$ the $n$-dimensional
\emph{Euclidean metric} of $x$.
For any $s\in\rr$, we denote by $\lceil s \rceil$
the \emph{smallest integer not less than} $s$.
For any set
$E\subset \rn$, we denote by $E^{\com}$ the set
$\rn\setminus E$,
by $\one_{E}$ its \emph{characteristic function},
by $\emptyset$
the empty set. In addition, the \emph{space}
$\MM(\rn)$ denotes the set of all measurable
functions on $\rn$. A \emph{simple function} on $\rn$ is a
finite summation $f=\sum_{k=1}^{N}a_{k}\one_{E_{k}}$,
where $N\in\nn$ and, for any $k\in\{1,\ldots,N\}$, $E_{k}\subset \rn$ is a measurable set
of finite measure, and  $a_{k}$ is a complex constant.
We denote by $\supp (f)$ the set
$\{x\in\rn:\ f(x)\neq 0\}$
and by $Q(x,r)$ the \emph{cube} with
center $x\in\rn$, edge length $r$, and all edges parallel
to the coordinate axes.
A \emph{ball} $B(x,r)$ in $\rn$ centered at
$x\in\rn$ with the
radius $r$ is defined by setting
\begin{equation*}
B(x,r):=\{y\in\rn:\ |x-y|<r\}.
\end{equation*}
Moreover, we let
\begin{equation}\label{Eqball}
\BB:=\lf\{B(x,r): \ x\in\rn~\text{and}~r
\in(0,\infty)\right\}.
\end{equation}	
We also use $\mathbf{0}$ to denote the origin of $\rn$ and
$\epsilon\to 0^{+}$ that $\epsilon\in(0,\infty)$ and
$\epsilon \to 0$.

\section{Fundamental Properties of Mixed-Norm Herz Space\label{s2}}

In this section, we first introduce a new class of
mixed-norm Herz spaces and
investigate their fundamental
properties. We obtain both the dual
spaces and the associate spaces of mixed-norm Herz spaces
in Subsection \ref{s2.1}, and then we show that mixed-norm Herz
spaces are ball quasi-Banach function spaces in
Subsection \ref{s2.2}. Moreover, we investigate the
Fefferman--Stein vector-valued maximal inequality
of the maximal operator $M_{n}$ on the ball quasi-
Banach function space.  Let us begin with the
concept of
mixed Lebesgue spaces on $\rn$
from \cite{BP}. Recall that $\MM(\rn)$ denotes the set of all
measurable functions on $\rn$.
\begin{definition}\label{mix-L}
Let $\vec{q}=(q_{1},\ldots,q_{n})\in (0,
\infty]^{n}$.
The \emph{mixed Lebesgue space} $\lq(\rn)$ is
defined to be the set of all the functions
$f\in \MM(\rn)$
such that
\begin{equation*}
\|f\|_{\lq(\rn)}:=\left\{\int_{\rr}
\cdots\left[\int_{\rr}|f(x_{1},\ldots,x_{n})|
^{q_{1}}
dx_{1}\right]^{\frac{q_{2}}{q_{1}}}\cdots dx_{n}
\right\}^
{\frac{1}{q_{n}}}<\infty
\end{equation*}
with the usual modifications made when $q_{i}=
\infty$
for some $i\in \{1,\ldots,n\}$.
\end{definition}

Let us also recall the concept of classical
Herz spaces on $\rn$;
see \cite[Definition 1.1.1]{HZ}.

\begin{definition}	\label{chz}
Let $\ p,\ q\in(0,\infty]$ and $\alpha\in\rr$.
The \emph{Herz space}
$\ohz$ is defined to be the set of all the
functions $f\in \MM(\rn)$
such that
\begin{equation*}
\|f\|_{\ohz}:=\left[\sum_{k\in \zz}2^{kp\alpha}
\|f\one_{B(\mathbf{0},2^{k})\setminus B(\mathbf{0},
2^{k-1})}
\|^{p}_{L^{q}(\rn)}\right]^{\frac{1}{p}}<\infty
\end{equation*}
with the usual modifications made when
$p=\infty$ or $q=\infty$.
\end{definition}
Now, we introduce a new class of mixed-norm Herz
spaces on $\rn$.
\begin{definition}\label{mhz}
Let $\vec{p}:=(p_{1},\ldots,p_{n}),$ $\vec{q}:=(q_{1},\ldots,q_{n})
\in(0,\infty]^{n}$ and $\vec{\alpha}:=
(\alpha_{1},\ldots,
\alpha_{n})\in\rn$.
The \emph{mixed-norm Herz space} $\iihz(\rn)$ is
defined to be the
set of all the functions
$f\in \MM(\rn)$ such
that
\begin{align*}
\|f\|_{\iihz(\rn)}:\,=&\,\left\{\sum_{k_{n} \in
\zz}2^{k_{n}
p_{n}\alpha_{n}}
\left[\int_{R_{k_{n}}}\cdots\left\{\sum_{k_{1}
\in \zz}
2^{k_{1}p_{1}\alpha_{1}}\right.\right.\right.\\
&\,\left.\left.\left.\times\left[\int_{R_{k_{1}}}|f(x_{1},
\ldots,x_{n})|
^{q_{1}}\,dx_{1} \right]^{\f{p_{1}}{q_{1}}}
\right\}^{\f{q_{2}}
{p_{1}}}\cdots
\,dx_{n}\right]^{\f{p_{n}}{q_{n}}}\right\}
^{\f{1}{pn}}\\
=&:\,\left\|\cdots\|f\|_{\ky(\rr)}\cdots\right\|_{\kn(\rr)}
<\infty
\end{align*}
with the usual modifications made when $p_{i}
=\infty$ or
$q_{j}=\infty$ for some $i,$ $j\in\{1,\ldots,n \}$,
where
we denote by $\|\cdots\|f\|_{\ky(\rr)}
\cdots\|_{\kn(\rr)}$ the
norm obtained after taking successively
the $\ky(\rr)$-norm
to $x_{1}$, the $\ke(\rr)$-norm to $x_{2}$,
$\ldots$, and
the $\kn(\rr)$-norm to $x_{n}$.
Here and thereafter, for any $i\in \{1,
\ldots,n\}$ and
$k_{i}\in\zz$, let
\begin{equation*}
R_{k_{i}}:=(-2^{k_{i}},2^{k_{i}})\setminus
(-2^{k_{i}-1},2^
{k_{i}-1}).
\end{equation*}	
\end{definition}
\begin{remark}\label{E*}
Note that, in Definition \ref{mhz}, when
$\vec{p}=\vec{q}:=
(p,\ldots,p)\in (0,\infty]^{n}$ and
$\vec{\alpha}:=\textbf{0}$, the mixed-norm Herz spaces $\iihz(\rn)$ coincide with
Lebesgue spaces
$L^{p}(\rn)$. Moreover, when
$p_{i}=q_{i}$ for any
$i\in\{1,\ldots,n\}$, and $\vec{\alpha}:=\textbf{0}$,
the mixed-norm Herz
spaces $\iihz(\rn)$ coincide with the mixed
Lebesgue
spaces $L^{\vec{q}}(\rn)$ defined in Definition
\ref{mix-L}.
Besides, when $\vec{p}:=(1,\ldots,1),$ $ \vec{q}
\in[1,\infty]^{n}$,
and $\vec{\alpha}:=\f{1}{\vec{q}'}$, the mixed-norm Herz space
$\iihz(\rn)$ is
just the  Herz space $\dot{E}_{\vec{q}}^{*}(\rn)$
in \cite[Definition 3.8]{mhz}.
\end{remark}

The concept of the $r$-convexification of $\iihz(\rn)$
is defined as follows.
\begin{definition}\label{cvx-2}
Let $\vec{p}$, $\vec{q}\in(0,
\infty]^{n},$ $\vec{\alpha}\in\rn$, and $r\in(0,\infty)$. The
\emph{r-convexification}
$[\iihz(\rn)]^{r}$ of $\iihz(\rn)$ is defined by
setting
$$[\iihz(\rn)]^{r}:=\left\{f\in  \MM(\rn): \
|f|^{r} \in \iihz(\rn)\right\},$$
equipped
with the quasi-norm
$$\|f\|_{[\iihz(\rn)]^{r}}:=
\left\|\,|f|^{r}\right\|_{\iihz(\rn)}^{1/r}.$$
\end{definition}
The mixed-norm Herz spaces $\iihz(\rn)$ enjoy the following
property
on convexification.
\begin{lemma}\label{hzcox-2}
Let $\vec{p}$, $\vec{q}\in(0,
\infty]^{n},\
\vec{\alpha}\in\rn$, and $r\in (0,\infty)$. Then
$f\in
[\iihz(\rn)]^{r}$ if and only
if $f\in\riihz(\rn)$. Moreover, for any $f\in \MM(\rn)$,
\begin{equation*}
\|f\|_{[\iihz(\rn)]^{r}}=\|f\|_{\riihz(\rn)}.
\end{equation*}
\end{lemma}
\begin{proof}
Let $\vec{p}:=(p_{1},\ldots,p_{n})$, $\vec{q}:=(q_{1},\ldots,q_{n})\in(0,
\infty]^{n},$ $
\vec{\alpha}:=(\alpha_{1},\ldots,\alpha_{n})\in\rn$, and $r\in (0,\infty)$.
For any given $r\in(0,\infty)$,
we first
consider the following two cases on $q_{1}$. If
$q_{1}\in(0,\infty)$,
then, we easily conclude that, for any $f\in\MM(\rn)$
and $x_{2},\ldots,x_{n}\in\rr$,
\begin{equation}\label{ncq}
\left\||f(\cdot,x_{2},\ldots,x_{n})|^{r}\one_{R_{k_{1}}}
(\cdot)\right\|_{L^{q_{1}}(\rr)}=\|f(\cdot,x_{2},\ldots,x_{n})
\one_{R_{k_{1}}}
(\cdot)\|_{L^{rq_{1}}(\rr)}^{r}.
\end{equation}
If $q_{1}=\infty$, then,
for any $f\in\MM(\rn)$ and $x_{2},\ldots,x_{n}\in\rr$,
$$\left\||f(\cdot,x_{2},\ldots,x_{n})|^{r}\one_{R_{k_{1}}}
(\cdot)\right\|_{L^{\infty}(\rr)}= \|f(\cdot,x_{2},\ldots,x_{n})
\one_{R_{k_{1}}}
(\cdot)\|_{L^{\infty}(\rr)}^{r}.$$
From this, we further
deduce that,
for any given $q_{1}\in(0,\infty]$, \eqref{ncq}
holds true. Then, for any given $r\in(0,\infty)$ and $q_{1}\in (0,\infty]$,
we consider the following two cases on $p_{1}$.
If $p_{1}\in (0,\infty)$, then, by Definition \ref{chz},
we conclude
that, for any $f\in\MM(\rn)$ and $x_{2},\ldots,x_{n}\in\rr$,
\begin{align*}
\left\||f(\cdot,x_{2},\ldots,x_{n})|^{r}\right\|_{\ky(\rr)}&=
\left\{\sum_{k_{1}
\in\zz}2^{k_{1}p_{1}\alpha_{1}}\left\||f(\cdot,x_{2},
\ldots,x_{n})|^{r}
\one_{R_{k_{1}}}(\cdot)\right\|_{L^{q_{1}}(\rr)}^{p_{1}}
\right\}^{\f{1}{p_{1}}}\\
&=\left\{\sum_{k_{1}\in\zz}2^{k_{1}rp_{1}
\f{\alpha_{1}}{r}}
\left\|f(\cdot,x_{2},\ldots,x_{n})\one_{R_{k_{1}}}(\cdot)\right\|
_{L^{rq_{1}}(\rr)}
^{rp_{1}} \right\}^{\f{r}{rp_{1}}}\\
&=\|f(\cdot,x_{2},\ldots,x_{n})\|^{r}_{\rky(\rr)}.
\end{align*}
If $p_{1}=\infty$, then, for any $x_{2},\ldots,x_{n}
\in\rr$, we have
\begin{align*}
\left\||f(\cdot,x_{2},\ldots,x_{n})|^{r}\right\|_{\ky(\rr)}&=
\sup_{k_{1}\in\zz} \left\{2^{k_{1}\alpha}\|f(\cdot,
x_{2},\ldots,x_{n})\|
^{r}_{L^{rq_{1}}(\rr)}
\right\}\\
&= \left[\sup_{k_{1}\in\zz}\left\{2^{k_{1}\f{\alpha}{r}}
\|f(\cdot,x_{2},\ldots,x_{n})\|_{L^{rq_{1}}(\rr)}
\right\}\right]^{r}\\
&=\|f(\cdot,x_{2},\ldots,x_{n})\|^{r}_{\rky(\rr)}.
\end{align*}
By this, we conclude that, for any given
$p_{1}$, $q_{1}\in (0,\infty]$ and $r\in(0,\infty)$, and
for any $f\in\MM(\rn)$ and $x_{2},\ldots,x_{n}\in\rr$,
$$\left\||f(\cdot,x_{2},\ldots,x_{n})
|^{r}\right\|_{\ky(\rr)}=
\|f(\cdot,x_{2},\ldots,x_{n})\|_{\rky(\rr)}^{r}.$$
Using Definition \ref{mhz},
similarly to the above estimation, we further obtain, for any $f\in\MM(\rn)$,
$$\left\||f|^{r}\right\|_{\iihz(\rn)}=\|f\|^{r}_{\riihz(\rn)}.$$
This finishes the proof of Lemma \ref{hzcox-2}.
\end{proof}
The following lemma is a direct consequence of the dominated
convergence theorem.
\begin{lemma}\label{donn}
Let $\vec{p}$, $\vec{q}\in(0,\infty)^{n}$
and $\vec{\alpha}\in\rn$. For any given $G\in\iihz(\rn)$,
any $\{f_{m} \}
_{m\in\nn}\subset \MM(\rn)$, and $f\in \MM(\rn)$, if
$|f_{m}|\leq |G|$
and $f_{m}\rightarrow f$ almost everywhere on $\rn$ as $m\to\infty$, then
\begin{equation}
\lim_{m\to \infty}\|f_{m}\|_{\iihz(\rn)}=\|f\|_{\iihz(\rn)}.
\end{equation}
\end{lemma}
\begin{proof}
Let $\vec{p}:=(p_{1},\ldots,p_{n})$, $\vec{q}:=(q_{1},\ldots,q_{n})\in(0,
\infty)^{n}$ and $
\vec{\alpha}:=(\alpha_{1},\ldots,\alpha_{n})\in\rn$. Since $f_{m}\to f$ almost
everywhere on $\rn$, we deduce that
\begin{equation*}
\left|\left\{x\in\rn:\ f_{m}(x)\nrightarrow f(x)\ \text{as}\ m\to \infty
\right\} \right|=0.
\end{equation*}
By this and \cite[p.\,69,\ Exercise 49]{folland}, we conclude that,
for almost every given
$x_{2},\ldots,x_{n}\in\rr$,
\begin{equation*}
\left|\left\{x_{1}\in\rr:\ f_{m}(x)\nrightarrow f(x)\ \text{as}\ m\to \infty
\right\} \right|=0
\end{equation*}
and hence
$f_{m}(\cdot,x_{2},\ldots,x_{n})\rightarrow
f(\cdot,x_{2},\ldots,x_{n})$ almost everywhere on $\rr$ as $m\to \infty$.
Moreover, by $G\in\iihz(\rn)$ and Definition \ref{mhz}, we conclude that,
for any $k_{n}\in\zz$,
\begin{equation*}
	\left\|\left\|G \right\|_{\dot{E}_{\vec{q}_{n-1}}^{\vec{\alpha}_{n-1},
	\vec{p}_{n-1}}(\rr^{n-1})} \one_{R_{k_{n}}}\right\|_{L^{q_{n}}(\rr)}<\infty
\end{equation*}
and hence, for almost every $x_{n}\in\rr$, $\|G(\cdot,x_{n})\|_{\dot{E}_{\vec{q}_{n-1}}^{\vec{\alpha}_{n-1},
\vec{p}_{n-1}}(\rr^{n-1})}<\infty$. Similarly to this estimation, we obtain, for
any $k_{1}\in\zz$ and almost every $x_{2},\ldots,x_{n}\in\rr$,
$$\|G(\cdot,x_{2},\ldots,x_{n})\one_{R_{k_{1}}}(\cdot)\|_{L^{q_{1}}(\rr)}<\infty.$$
From this, $|f_{m}
\one_{R_{k_{1}}}|\leq |G\one_{R_{k_{1}}}|$ for
any $m\in\nn$ and $k_{1}\in\zz$, and the dominated convergence theorem,
it follows that, for any $k_{1}\in\zz$ and
almost every $x_{2},\ldots,x_{n}\in\rr$,
\begin{equation}\label{cll2}
\lim_{m\to \infty}\left\|f_{m}(\cdot,x_{2},\ldots,x_{n})
\one_{R_{k_{1}}}(\cdot)
\right\|_{L^{q_{1}}(\rr)}
=\left\|f(\cdot,x_{2},\ldots,x_{n})\one_{R_{k_{1}}}(\cdot)
\right\|_{L^{q_{1}}(\rr)}.
\end{equation}
Note that, for any $k_{1}\in\zz$ and $x_{2},\ldots,x_{n}\in\rr$,
$$2^{k_{1}p_{1}\alpha_{1}}
\left\|f_{m}(\cdot,x_{2},\ldots,x_{n})
\one_{R_{k_{1}}}(\cdot)\right\|^{p_{1}}_{L^{q_{1}}(\rr)}\leq 2^{k_{1}p_{1}\alpha_{1}}\left\|G(\cdot,x_{2},\ldots,x_{n})
\one_{R_{k_{1}}}(\cdot)\right\|
^{p_{1}}_{L^{q_{1}}(\rr)}.$$
By this, \eqref{cll2}, and
the dominated
convergence theorem, we find that, for almost every
$x_{2},\ldots,x_{n}\in\rr$,
\begin{equation*}
\lim_{m\to \infty}\|f_{m}(\cdot,x_{2},\ldots,x_{n})
\|_{\ky(\rr)}
=\|f(\cdot,x_{2},\ldots,x_{n})\|_{\ky(\rr)}.
\end{equation*}
Similarly to this, from Definition \ref{mhz}, it
follows that
\begin{equation*}
\lim_{m\to \infty}\|f_{m}\|_{\iihz(\rn)}=\|f\|
_{\iihz(\rn)},
\end{equation*}
which completes the proof of Lemma \ref{donn}.
\end{proof}
As an immediate consequence of Lemma \ref{donn},
we have the following conclusion.
\begin{proposition}\label{ab-2}
Let $\vec{p}$, $\vec{q}\in(0,
\infty)^{n}$
and $\vec{\alpha}\in\rn$. Then, for any
$f\in\iihz(\rn)$ and any sequence
$\{E_{j} \}_{j\in\nn}$ of measurable sets of
$\rn$ satisfying
that $\one_{E_{j}}\rightarrow 0$
almost everywhere as $j\rightarrow \infty$,
$\|f\one_{E_{j}}\|
_{\iihz(\rn)}\rightarrow 0$ as $j\rightarrow
\infty$.
\end{proposition}
\begin{proof}
Let all the symbols be as in the present
proposition.
For any $f\in\iihz(\rn)$ and any sequence
$\{ E_{j}\}
_{j\in\nn}$ of
measurable sets of $\rn$ satisfying that
$\one_{E_{j}}
\rightarrow 0$
almost everywhere as $j\rightarrow \infty$,
let $g_{j}(x):
=f(x)\one_{E_{j}}(x)$
for any $x\in\rn$ and for any given $j\in\nn$. Then, for any $j\in\nn$,
$|g_{j}|
\leq |f|\in\iihz(\rn)$
and $g_{j}\rightarrow 0$ almost everywhere as $j\to \infty$.
By this
and Lemma \ref{donn}, we conclude that, as $j\to \infty$,
\begin{equation*}
\|f\one_{E_{j}}\|_{\iihz(\rn)}=\|g_{j}\|_
{\iihz(\rn)}
\rightarrow 0,
\end{equation*}
which completes the proof of Proposition
\ref{ab-2}.
\end{proof}

\subsection{Dual Spaces of Mixed-Norm Herz Spaces}\label{s2.1}

We begin with the definition of dual spaces of Banach spaces.

\begin{definition}
Let $X$ be a Banach space equipped with the norm
$\|\cdot\|$. The vector
space of all continuous linear functionals
on $X$ is called the
\emph{dual space} of $X$, which is denoted
by $X^{*}$.
\end{definition}

Now, we establish the dual space of $\iihz(\rn)$.
\begin{theorem}\label{dii}
Let $\vec{p}$, $\vec{q}\in
[1,\infty)^{n}$ and
$\vec{\alpha}\in\rn$. Then the dual space
of $\iihz(\rn)$,
denoted by
$[\iihz(\rn)]^{*},$ is $\diihz(\rn)$ in the
following sense:
\begin{enumerate}
\item[$\mathrm{(i)}$] Let $g\in \diihz(\rn)$.
Then the
linear functional
\begin{equation}\label{fff}
J_{g}:f\rightarrow J_{g}(f):=\int_{\rn}
f(x)g(x)\,dx
\end{equation}
is bounded on $\iihz(\rn)$.
\item[$\mathrm{(ii)}$] Conversely,
any continuous linear functional on
$\iihz(\rn)$ arises as
in \eqref{fff}
with a unique $g\in\diihz(\rn)$.
\end{enumerate}
Moreover, $\|g\|_{\diihz(\rn)}=\|J_{g}\|_
{[\iihz(\rn)]^{*}}$.
\end{theorem}
To show Theorem \ref{dii}, we need some preliminary
lemmas and propositions.
The following lemma is a direct consequence of
the H\"{o}lder inequality.
\begin{lemma}\label{mhr}
Let $\vec{p}$,
$\vec{q}\in[1,\infty]^{n}$ and
$\vec{\alpha}\in\rn$.
If $f,\ g\in\MM(\rn)$, then
\begin{equation}
\int_{\rn}|f(x)g(x)|\,dx\leq \|f\|_{\iihz(\rn)}
\|g\|
_{\diihz(\rn)}.
\end{equation}
\end{lemma}
\begin{proof}
Let $\vec{p}:=(p_{1},\ldots,p_{n})$,
$\vec{q}:=
(q_{1},\ldots,q_{n})\in[1,\infty]^{n}$ and
$\vec{\alpha}:=(\alpha_{1},\ldots,\alpha_{n})\in\rn$.
From the Tonelli theorem, the H\"{o}lder inequality,
and Definition
\ref{mhz}, it follows that
\begin{align*}
&\int_{\rn}|f(x)g(x)|\,dx\\
&\quad =
\int_{\rr}\cdots\int_{\rr}\sum_{k_{1}\in\zz}\int_{R_{k_{1}}}
|f(x_{1},\ldots,x_{n})
g(x_{1},\ldots,x_{n})|\,dx_{1}dx_{2}\cdots dx_{n}\\
&\quad \leq
\int_{\rr}\cdots\int_{\rr}\sum_{k_{1}\in\zz}2^{k_{1}\alpha_{1}}
\|f(\cdot,x_{2},\ldots,x_{n})
\one_{R_{k_{1}}}(\cdot)\|_{L^{q_{1}}(\rr)}\\
&\qquad\times 2^{-k_{1}
\alpha_{1}} \|g(\cdot,x_{2},\ldots,x_{n})\one_{R_{k_{1}}}(\cdot)
\|_{L^{q_{1}'}(\rr)}\,dx_{2}\cdots dx_{n}\\
&\quad\leq \int_{\rr}\cdots\int_{\rr}\|f (\cdot,x_{2},
\ldots,x_{n})\|
_{\ky(\rr)}\|g(\cdot,x_{2},\ldots,x_{n})\|
_{\dky(\rr)}\,dx_{2}\cdots dx_{n}.
\end{align*}
Repeating this  process, we further obtain
\begin{equation*}
\int_{\rn}|f(x)g(x)|\,dx\leq \|f\|_{\iihz(\rn)}\|g\|_{\diihz(\rn)},
\end{equation*}
which then completes the proof of Lemma \ref{mhr}.
\end{proof}
\begin{proposition}\label{E-lat}
Let $\vec{p}$, $\vec{q}\in(0,
\infty]^{n}$,
$\vec{\alpha}\in\rn$, and $f$, $g\in \MM(\rn)$. If
$|g|\leq |f|$
almost everywhere on $\rn$, then
$$\|g\|_{\iihz(\rn)}\leq
\|f\|_{\iihz(\rn)}.$$
\end{proposition}

\begin{proof}
Let all the symbols be as in the present proposition.
Without loss of generality, we may only consider the case $n:=2$.
Let $\vec{p}:=(p_{1},p_{2}),$ $\vec{q}:=(q_{1},q_{2})$, $\vec{\alpha}
:=(\alpha_{1},\alpha_{2})$, and
$$E:=\left\{(x_{1},x_{2})\in\rr^{2}:\ |f(x_{1},x_{2})|
<|g(x_{1},x_{2})|\right\}.$$
If $|g|\leq |f|$ almost everywhere on $\rr^{2}$,
then $|E|=0$.
By this and \cite[p.\,69, Exercise 49]{folland},
we find that,
for almost every $x_{2}\in\rr$, $|E^{x_{2}}|=0$,
where
$$E^{x_{2}}:
=\left\{x_{1}\in\rr:\ |f(x_{1},x_{2})|<|g(x_{1},x_{2})|\right\}.$$
Thus,
for almost every $x_{2}\in\rr$, $|g(\cdot,x_{2})|
\leq |f(\cdot,x_{2})|$
almost everywhere on $\rr$. Using Definition \ref{chz},
we easily obtain, for almost every $x_{2}\in\rr$,
$$\|g(\cdot,x_{2})\|_{\ky(\rr)}\leq \|f(\cdot,x_{2})\|_
{\ky(\rr)}.$$
Similarly to this estimation, we further have
$$\|g\|_{\iihz(\rr^{2})}
\leq \|f\|_{\iihz(\rr^{2})}.$$
This finishes the proof of Proposition \ref{E-lat}.
\end{proof}
The following lemma shows that all bounded functions supported in the set $A_{m}$ belong to $\iihz(\rn)$, which plays
an important role in the proof of the duality and the Riesz--Thorin
interpolation theorem below.
\begin{lemma}\label{Am}
Let $\vec{p},$ $\vec{q}\in (0,\infty]^{n}$ and
$\vec{\alpha}\in\rn$.
For any $m\in\nn$, let $I_{m}:=(-2^{m},-2^{-m}]\cup[2^{-m},2^{m})$
and $A_{m}:=I_{m}^{n}$. Then, for any $m\in\nn$, $\one_{A_{m}}
\in \iihz(\rn)$.
\end{lemma}
\begin{proof}
Let $\vec{p}:=(p_{1},\ldots,p_{n})$,
$\vec{q}:=
(q_{1},\ldots,q_{n})\in(0,\infty]^{n}$ and
$\vec{\alpha}:=(\alpha_{1},\ldots,\alpha_{n})\in\rn$.
For any given $i\in\{1,\ldots,n \}$, we first consider the
following two cases
on $q_{i}$. If $q_{i}\in(0,\infty)$, then
$\|\one_{R_{k_{i}}}\|_{L^{q_{i}}(\rr)}=2^{k_{i}\f{1}{q_{i}}}$.
If $q_{i}=\infty$, then $\|\one_{R_{k_{i}}}\|_{L^{\infty}(\rr)}=1$.
Thus, for any $q_{i}\in(0,\infty]$,
$\|\one_{R_{k_{i}}}\|_{L^{q_{i}}(\rr)}=2^{k_{i}\f{1}{q_{i}}}$.
Then we consider the following two cases on $p_{i}$. If
$p_{i}\in (0,\infty)$,
then we conclude that
\begin{equation*}
\|\one_{I_{m}}\|_{\ki(\rr)}=\left[\sum_{k_{i}=-m+1}^{m}
2^{k_{i}p_{i}\alpha_{i}}
\|\one_{R_{k_{i}}}\|_{L^{q_{i}}(\rr)}^{p_{i}} \right]^
{\f{1}{p_{i}}}
\sim\left[\sum_{k_{i}=-m+1}^{m}2^{k_{i}p_{i}(\alpha_{i}+
\f{1}{q_{i}})}
\right]^{\f{1}{p_{i}}}<\infty.
\end{equation*}
If $p_{i}=\infty$, then we conclude that
\begin{equation*}
\|\one_{I_{m}}\|_{\ki(\rr)}=\left[\sup_{k_{i}\in\zz\cap[-m+1,m]}
2^{k_{i}\alpha_{i}}\|\one_{R_{k_{i}}}\|_{L^{\infty}(\rr)}
\right]
\sim\left(\sup_{k_{i}\in\zz\cap[-m+1,m]}2^{k_{i}\alpha_{i}}
\right)<\infty.
\end{equation*}
Thus, for any $p_{i},\ q_{i}\in(0,\infty]$, $\|\one_{I_{m}}\|_
{\ki(\rr)}<\infty$.
By this, we obtain $$\|\one_{A_{m}}\|_{\iihz(\rn)}=
\prod_{i=1}^{n}\|\one_{I_{m}}
\|_{\ki(\rr)}<\infty,$$ which implies that $\one_{A_{m}}
\in \iihz(\rn)$.
This finishes the proof of Lemma \ref{Am}.
\end{proof}
\begin{proposition}\label{equ0}
Let $\vec{p}$, $\vec{q}\in (0,\infty]^{n}$ and
$\vec{\alpha}\in \rn$. For any $f\in \iihz(\rn)$, if
$\|f\|_{\iihz(\rn)}=0$,
then $f=0$ almost everywhere.
\end{proposition}	
\begin{proof}
Let $\vec{p}:=(p_{1},\ldots,p_{n}),$ $\vec{q}:=(q_{1},\ldots,q_{n})\in (0,\infty]^{n}$ and
$\vec{\alpha}:=(\alpha_{1},\ldots,\alpha_{n})\in\rn$.
Let $r\in
(0,\min\{p_{-},q_{-} \})$, here and thereafter,
\begin{equation}\label{p-}
p_{-}:=\min\{p_{1},\ldots,p_{n} \}\ \text{and}\ q_{-}:=\min\{q_{1},
\ldots,q_{n} \}.
\end{equation}
From Lemma \ref{hzcox-2} and Proposition \ref{E-lat},
we deduce that, for any $f\in\iihz(\rn)$ and $m\in\nn$,
\begin{equation*}
\left\||f|^{r}\one_{A_{m}}\right\|^{1/r}_{\dot{E}^
{r\vec{\alpha},\vec{p}/r}
_{\vec{q}/r}(\rn)}=\|f\one_{A_{m}}\|_{\dot{E}^
{\vec{\alpha},\vec{p}}_
{\vec{q}}(\rn)}\leq\|f\|_{\iihz(\rn)}=0,
\end{equation*}
where $A_{m}$ is as in Lemma \ref{Am}. This, combined with Lemmas \ref{mhr} and \ref{Am}, further
implies that,
for any $m\in\nn$,
\begin{equation*}
\left\||f|^{r}\one_{A_{m}}\right\|_{L^{1}(\rn)}\leq \left\||f|^{r}
\one_{A_{m}}\right\|
_{\dot{E}^{r\vec{\alpha},\vec{p}/r}_{\vec{q}/r}(\rn)}
\|\one_{A_{m}}\|_
{\dot{E}^{-r\vec{\alpha},(\vec{p}/r)'}_{(\vec{q}/r)'}(\rn)}
\ls 0.
\end{equation*}
By this and the monotone convergence theorem, we conclude
that $\||f|^{r}\|_{L^{1}(\rn)}=0$ and hence $f=0$
almost everywhere on $\rn$. This finishes the proof of
Proposition \ref{equ0}.
\end{proof}
The following lemma is an essential generalization of \cite[p.\,303,\ Theorem 1]{BP}.
\begin{lemma}\label{assnd}
Let $\vec{p}$, $\vec{q}\in[1,\infty]^{n}$ and
$\vec{\alpha}\in\rn$. Then, for any $f\in\iihz(\rn)$,
\begin{equation}\label{fsp}
\|f\|_{\iihz(\rn)}=\sup\left\{\|fg\|_{L^{1}(\rn)}:\ \|g\|_{\diihz(\rn)}
=1 \right\}.
\end{equation}
\end{lemma}
\begin{proof}
Let $\vec{p}:=(p_{1},\ldots,p_{n}),$ $\vec{q}:=(q_{1},\ldots,q_{n})
\in [1,\infty]^{n}$ and
$\vec{\alpha}:=(\alpha_{1},\ldots,\alpha_{n})\in\rn$.
For any $f\in \iihz(\rn)$, we consider
the following
two cases on $\|f\|_{\iihz(\rn)}$.

\emph{Case 1)} $\|f\|_{\iihz(\rn)}=0$. In this case, using Proposition \ref{equ0},
we obtain $f=0$ almost everywhere. By this, we have
\begin{equation*}
\|f\|_{\iihz(\rn)}=0=\sup\left\{\|fg\|_{L^{1}(\rn)}:\
\|g\|_{\diihz(\rn)}=1
\right\}
\end{equation*}
and hence \eqref{fsp} holds true in this case.

\emph{Case 2)}
$\|f\|_{\iihz(\rn)}\neq 0$. In this case,
from Lemma \ref{mhr}, it follows
that, for any $g\in\diihz(\rn)$ satisfying $\|g\|_{\diihz(\rn)}=1$,
\begin{equation*}
\|fg\|_{L^{1}(\rn)}\leq \|f\|_{\iihz(\rn)}\|g\|_{\diihz(\rn)}
=\|f\|_{\iihz(\rn)},
\end{equation*}
which implies that
\begin{equation*}
\sup\left\{\|fg\|_{L^{1}(\rn)}:\ \|g\|_{\diihz(\rn)}=1
\right\}\leq
\|f\|_{\iihz(\rn)}.
\end{equation*}

Conversely, we prove that
\begin{equation}\label{gf}
\sup\left\{\|fg\|_{L^{1}(\rn)}:\ \|g\|_{\diihz(\rn)}=1
\right\} \geq
\|f\|_{\iihz(\rn)}.
\end{equation}
To this end, we only need to find a family of functions,
$\{g_{\epsilon}\}_
{\epsilon\in (0,\infty)}$ with $\|g_{\epsilon}\|_{\diihz(\rn)}=1$
for any $\epsilon\in (0,\infty)$,
such that
$$\sup_{\epsilon\in (0,\infty)}\|fg_{\epsilon}\|_{L^{1}(\rn)}\geq
\|f\|_{\iihz(\rn)}.$$
To show this, let
$\ell\in\iihz(\rn)$, $\|\ell\|_{\iihz(\rn)}\neq 0$, and
$\supp(\ell)\subset A_{k_{0}}$
for some $k_{0}\in\nn$, where $A_{k_{0}}$ is as in
Lemma \ref{Am}.
For any $i\in\{1,\ldots,n-1\}$, let $\vec{p}_{i}:=
(p_{1},\ldots,p_{i})$, $\vec{q}_{i}:=(q_{1},\ldots,q_{i})
\in[1,\infty]^{i}$, and $\vec{\alpha}_{i}:=
(\alpha_{1},\ldots,\alpha_{i})\in\rr^{i}$.

Now, for any given $i\in\{0,1,\ldots,n-1\}$ and $\gamma\in\rr$,  and
for any $x_{i+1},\ldots,x_{n}\in\rr$,
we define $[E_{i}(\ell)(x_{i+1},\ldots,x_{n})]^{\gamma}$ by
setting, for any $i\in\{1,\ldots,n-1\}$,
\begin{align}\label{Eiga}
&[E_{i}(\ell)(x_{i+1},\ldots,x_{n})]^{\gamma}\\
&\quad:=
\begin{cases}
\left\|\ell(\cdot,x_{i+1},\ldots,x_{n})\right\|^{\gamma}
_{\dot{E}^
{\vec{\alpha}_{i},\vec{p}_{i}}_{\vec{q}_{i}}(\rr^{i})}&
\mathrm{if}\ \left\|\ell(\cdot,x_{i+1},\ldots,x_{n})
\right\|_{\dot{E}^
{\vec{\alpha}_{i},\vec{p}_{i}}_{\vec{q}_{i}}
(\rr^{i})}\neq 0,\\
0&\mathrm{otherwise},\nonumber
\end{cases}
\end{align}
$[E_{0}(\ell)(x)]^{\gamma}:=|\ell(x)|^{\gamma}$
if $|\ell(x)|\neq0$, and
$[E_{0}(\ell)(x)]^{\gamma}:=0$ if $|\ell(x)|=0$, and $E_{n}(\ell):=
\|\ell\|_{\iihz(\rn)}\neq 0$.
Then, for any given $j\in\{1,\ldots,n-1\}$, $ q_{j}\in[1,\infty]$, $s_{j}\in\rr$,
and $k_{j}\in\zz$, and for any $x_{j+1},\ldots,x_{n}\in\rr$, let
\begin{align}\label{2.9x}
&P^{(k_{j})}_{s_{j},q_{j}}(\ell)(x_{j+1},\ldots,x_{n})\\
&\quad:=\begin{cases}
\left\|E_{j-1}(\ell)(\cdot,x_{j+1},\ldots,x_{n})\one_
{R_{k_{j}}}(\cdot)\right\|^{s_{j}}_{L^{q_{j}}(\rr)}\\
&\hspace{-3cm}\mathrm{if}\ \left\|E_{j-1}(\ell)(\cdot,x_{j+1},\ldots,x_{n})\one_
{R_{k_{j}}}(\cdot)\right\|_{L^{q_{j}}(\rr)}\neq 0,\\
0 &\hspace{-3cm}\mathrm{otherwise}.
\end{cases}\nonumber
\end{align}
Moreover, let $P^{(k_{n})}_{s_{n},q_{n}}(\ell):=\|E_{n-1}(\ell)\one_
{R_{k_{n}}}\|^{s_{n}}_{L^{q_{n}}(\rr)}$ if $\|E_{n-1}(\ell)
\one_{R_{k_{n}}}\|_{L^{q_{n}}(\rr)}\neq 0$, and $P^{(k_{n})}_
{s_{n},q_{n}}(\ell):=0$ if $\|E_{n-1}(\ell)\one_
{R_{k_{n}}}\|_{L^{q_{n}}(\rr)}= 0$.  Now, for any given $\epsilon\in(0,\infty),$
$j\in\{1,\ldots,n-1 \}$, and
$k_{j}\in\zz$, and for any $x_{j+1},\ldots,x_{n}\in\rr$, we define
$h_{k_{j},\epsilon}(\ell)(x_{j+1},\ldots,x_{n})$ by setting
\begin{align}\label{hke}
&h_{k_{j},\epsilon}(\ell)(x_{j+1},\ldots,x_{n})\\
&\quad:=\begin{cases}
\displaystyle\f{\one_{E_{j,\epsilon}(\ell)(x_{j+1},\ldots,x_{n})}(k_{j})}{\#[E_{j,\epsilon}(\ell)
(x_{j+1},\ldots,x_{n}) ]} &\text{if}\ p_{j}=\infty\ \text{and}\
E_{j,\epsilon}(\ell)(x_{j+1},\ldots,x_{n})\neq \emptyset,\\
0 &\text{if}\ p_{j}=\infty\ \text{and}\ E_{j,\epsilon}(\ell)
(x_{j+1},\ldots,x_{n})
=\emptyset,\\
1 & \text{if}\ p_{j}\in[1,\infty),
\end{cases}	\notag
\end{align}
and $h_{k_{n},\epsilon}(\ell)$ is as in \eqref{hke}
with $E_{j,\epsilon}(\ell)(x_{j+1},\ldots,x_{n})$
replaced by $E_{n,\epsilon}(\ell)$, where $\#$ denotes the
\emph{counting measure}
(see, for instance, \cite[p.\,263]{ra} for the precise definition)
and $E_{j,\epsilon}(\ell)(x_{j+1},\ldots,x_{n})$ is
defined by setting
\begin{align}\label{Esion}
&E_{j,\epsilon}(\ell)(x_{j+1},\ldots,x_{n})\\
&\quad:=
\bigg\{k_{j}\in\zz:\  2^{k_{j}\alpha_{j}}(1+\epsilon)
\|E_{j-1}(\ell)
(\cdot,x_{j+1},\ldots,x_{n})\one_{R_{k_{j}}}(\cdot)
\|_{L^{q_{j}}(\rr)}\nonumber\\
&\qquad\quad > E_{j}(\ell)(x_{j+1},\ldots,x_{n}) \bigg\}\nonumber
\end{align}
and
$$E_{n,\epsilon}(\ell):=\left\{k_{n}\in\zz:\ 2^{k_{n}
\alpha_{n}}(1+\epsilon)
\left\|E_{n-1}(\ell)(\cdot)\one_{R_{k_{n}}}(\cdot)\right\|_
{L^{q_{n}}(\rr)}>E_{n}(\ell)\right\}.$$
For any given $\epsilon\in(0,\infty)$, $j\in\{ 1,\ldots,n\}$,
and $k_{j}\in\zz$, and for any $x_{j+1},\ldots,x_{n}\in\rr$, let
\begin{align}\label{GKj}
G_{k_{j},\epsilon}(\ell)(x_{j+1},\ldots,x_{n}):=P^{(k_{j})}_
{\widetilde{p}_{j}-\widetilde{q}_{j},q_{j}}(\ell)(x_{j+1},
\ldots,x_{n})
h_{k_{j},\epsilon}(\ell)(x_{j+1},\ldots,x_{n})
\end{align}
and
\begin{equation}\label{2.12x}
	G_{k_{n},\epsilon}(\ell):=
	P^{(k_{n})}_{\widetilde{p}_{n}-\widetilde{q}_{n},q_{n}}(\ell)
	h_{k_{n},\epsilon}(\ell),
\end{equation}
where, for any $j\in \{1,\ldots,n \}$,
\begin{equation}\label{pyw}
\widetilde{p}_{j}:=
\begin{cases}
p_{j} & \text{if}\ p_{j}\in [1,\infty),\\
1 & \text{if}\ p_{j}=\infty
\end{cases}
\quad\mathrm{and}\quad
\widetilde{q}_{j}:=
\begin{cases}
q_{j} & \text{if}\ q_{j}\in[1,\infty),\\
1 & \text{if}\ q_{j}=\infty.
\end{cases}
\end{equation}
For any given $\epsilon\in(0,\infty)$, $j\in\{1,\ldots,n-1 \}$,
and $k_{j}\in\zz$, and for any $x_{j},\ldots,x_{n}\in\rr$,  let
\begin{align}\label{fksion}
&f_{k_{j},\epsilon}(\ell)(x_{j},\ldots,x_{n})\\
&\quad:=\begin{cases}
\displaystyle\f{\one _{F_{\epsilon}^{(k_{j})}(\ell)(x_{j+1},\ldots,
x_{n})}(x_{j})}
{|F_{\epsilon}^{(k_{j})}(\ell)(x_{j+1},\ldots,x_{n})|}
& \text{if}
\ q_{j}=\infty\ \text{and}\ |F_{\epsilon}^{(k_{j})}
(\ell)(x_{j+1},
\ldots,x_{n})|\neq 0,\\
0 & \text{if}\ q_{j}=\infty\ \text{and}\ |
F_{\epsilon}^{(k_{j})}
(\ell)(x_{j+1},\ldots,x_{n})|= 0,\\
1 & \text{if}\ q_{j}\in[1,\infty),
\end{cases}\nonumber
\end{align}
and $f_{k_{n},\epsilon}(\ell)(x_{n})$ is as in \eqref{fksion}
with $F^{(k_{j})}_
{\epsilon}(\ell)(x_{j+1},\ldots,x_{n})$ there
 replaced by $F^{(k_{n})}_{\epsilon}(\ell)$,
where the set $F^{(k_{j})}_{\epsilon}(\ell)(x_{j+1},
\ldots,x_{n})$
is defined by setting
\begin{align}\label{Fsion}
F^{(k_{j})}_{\epsilon}(\ell)(x_{j+1},\ldots,x_{n}):
=&\bigg\{x_{j}
\in\rr:\  (1+\epsilon)E_{j-1}(\ell)(x_{j},\ldots,x_{n})
\one_{R_{k_{j}}}
(x_{j})\\
&>\|E_{j-1}(\ell)(\cdot,x_{j+1},\ldots,x_{n})
\one_{R_{k_{j}}}(\cdot)
\|_{L^{\infty}(\rr)}\bigg\},\nonumber
\end{align}
and
$$F_{\epsilon}^{(k_{n})}(\ell)
:=\left\{x_{n}\in\rr:\ (1+\epsilon)E_{n-1}
(\ell)(x_{n})\one_{R_{k_{n}}}(x_{n}) >\|E_{n-1}(\ell)
\one_{R_{k_{n}}}\|_{L^{\infty}(\rr)}\right\}.$$
For any given $\epsilon\in(0,\infty)$ and $j\in\{1,\ldots,n \}$,
and for any $x_{j},\ldots,x_{n}\in\rr$, let
\begin{equation}\label{Hje}
H_{j,\epsilon}(\ell)(x_{j},\ldots,x_{n}):=\sum_{k_{j}\in\zz}f_{k_{j},
\epsilon}(\ell)(x_{j},\ldots,x_{n})\one_{R_{k_{j}}}(x_{j}).
\end{equation}
Now, we can define the following iterated functions.
For any given $\epsilon\in(0,\infty)$ and $j\in\{1,\ldots, n-1 \}$,
and for any $x_{j},\ldots,x_{n}\in\rr$, let
\begin{equation}\label{2.40x}
g_{j,\epsilon}(\ell)(x_{j},\ldots,x_{n}):=
\sum_{k_{j}\in\zz}2^{k_{j}\widetilde{p}_{j}\alpha_{j}}G_{k_{j},
\epsilon}(\ell)(x_{j+1},\ldots,x_{n})\one_{R_{k_{j}}}(x_{j})
\end{equation}
and
\begin{equation}\label{2.17x}
g_{n,\epsilon}(\ell)(x_{n}):=\sum_{k_{n}\in\zz}2^{k_{n}\widetilde{p}_{n}
\alpha_{n}}G_{k_{n},\epsilon}(\ell)\one_{R_{k_{n}}}(x_{n}),
\end{equation}
and
\begin{align}\label{2.40y}
\widetilde{g}_{j,\epsilon}(\ell)(x_{j},\ldots,x_{n}):=&\,[E_{n}(\ell)]^
{1-\widetilde{p}_{n}}\prod_{i=j}^{n-1}[E_{i}(\ell)(x_{i+1},\ldots,x_{n})]
^{\widetilde{q}_{i+1}-\widetilde{p}_{i}}\\
&\,\times \prod_{i=j}^{n}[g_{i,\epsilon}(\ell)(x_{i},\ldots,x_{n}) H_{i,\epsilon}(\ell)(x_{i},\ldots,x_{n})]\nonumber
\end{align}
and
\begin{equation}\label{2.18xx}
\widetilde{g}_{n,\epsilon}(\ell)(x_{n}):=[E_{n}(\ell)]^{1-\widetilde{p}
_{n}}g_{n,\epsilon}(\ell)(x_{n}) H_{n,\epsilon}(\ell)(x_{n}).
\end{equation}
For any given $\epsilon\in(0,\infty)$ and $j\in\{2,\ldots,n \}$, and for any $x_{j},
\ldots,x_{n}\in\rr$, let
\begin{equation}\label{2.41y}
\widehat{g}_{j,\epsilon}(\ell)(x_{j},\ldots,x_{n}):=[E_{j-1}(\ell)(x_{j},
\ldots,x_{n})]^{\widetilde{q}_{j}-1}\widetilde{g}_{j,\epsilon}
(\ell)(x_{j},\ldots,x_{n})
\end{equation}
and
\begin{equation}\label{2.41x}
\widehat{g}_{1,\epsilon}(\ell)(x):=\left[\mathrm{sgn}\,\overline{\ell(x)}\right][E_{0}
(\ell)(x)]^{\widetilde{q}_{1}-1}\widetilde{g}_{1,\epsilon}(\ell)(x),
\end{equation}
where, for any $z\in\cc$,
\begin{equation*}
\mathrm{sgn}\,z:=\begin{cases}
0 & \mathrm{if}\ z=0,\\
\dis\f{z}{|z|} & \mathrm{if}\ z\neq 0.
\end{cases}
\end{equation*}
Then we claim that, for any $\vec{p},$ $\vec{q}\in [1,\infty]^{n}$,
$\vec{\alpha}\in\rn$, $\epsilon\in (0,\infty)$, and $\ell\in\iihz(\rn)$ satisfying $\|\ell\|
_{\iihz(\rn)}\neq 0$ and $\supp(\ell)\subset A_{k_{0}}$ for some
$k_{0}\in\nn$,
\begin{equation}\label{c1}
\left\|\widehat{g}_{1,\epsilon}(\ell)\right\|_{\diihz(\rn)}=1
\end{equation}
and
\begin{equation}\label{c22}
\int_{\rn}\ell(x)\widehat{g}_{1,\epsilon}(\ell)(x)\,dx\geq (1+\epsilon)^
{-2n}\|\ell\|_{\iihz(\rn)},
\end{equation}
where $1/\vec{p}+1/\vec{p}'=1=1/\vec{q}+1/\vec{q}'$.

We first show \eqref{c1}. Indeed, by \eqref{2.41x}, \eqref{2.40y},
and \eqref{2.40x}, we
conclude that, for any given $\epsilon\in(0,\infty)$ and $k_{1}\in\zz$,
and for any $x_{2},\ldots,x_{n}\in\rr$,
\begin{align*}
&\left\|\widehat{g}_{1,\epsilon}(\ell)(\cdot,x_{2},\ldots,x_{n})\one_
{R_{k_{1}}}(\cdot)\right\|_{L^{q_{1}'}(\rr)}\\
&\quad=\left\|[E_{0}(\ell)(\cdot,x_{2},\ldots,x_{n})]^{\widetilde{q}_{1}-1}g
_{1,\epsilon}(\ell)(\cdot,x_{2},\ldots,x_{n})H_{1,\epsilon}(\ell)(\cdot,
x_{2},\ldots,x_{n})\one_{R_{k_{1}}}(\cdot)\right\|_{L^{q_{1}'}(\rr)}
\nonumber\\
&\quad\quad\times[E_{1}(\ell)(x_{2},\ldots,x_{n})]^{\widetilde{q}
_{2}-\widetilde{p}_{1}}\widetilde{g}_{2,\epsilon}(\ell)(x_{2}
,\ldots,x_{n})\\\nonumber
&\quad=2^{k_{1}\widetilde{p}_{1}\alpha_{1}}G_{k_{1},\epsilon}
(\ell)(x_{2},\ldots,x_{n})\left\|[E_{0}(\ell)(\cdot,x_{2},\ldots,x_{n})]
^{\widetilde{q}_{1}-1}H_{1,\epsilon}(\ell)(\cdot,x_{2},\ldots,x_
{n})\one_{R_{k_{1}}}(\cdot)\right\|_{L^{q_{1}'}(\rr)}\nonumber\\
&\quad\quad\times[E_{1}(\ell)(x_{2},\ldots,x_{n})]^{\widetilde{q}
_{2}-\widetilde{p}_{1}}\widetilde{g}_{2,\epsilon}(\ell)(x_{2}
,\ldots,x_{n}).\nonumber
\end{align*}
Then we consider the following three cases on $q_{1}$.

If $q_{1}\in(1,\infty)$, then, from \eqref{GKj}, \eqref{pyw},
\eqref{Hje}, \eqref{fksion}, and \eqref{2.9x}, we deduce that,
for any given $\epsilon\in(0,\infty)$ and $k_{1}\in\zz$,
and for any $x_{2},\ldots,x_{n}\in\rr$,
\begin{align*}
&\left\|\widehat{g}_{1,\epsilon}(\ell)(\cdot,x_{2},\ldots,x_{n})\one_{R_{
k_{1}}}(\cdot)\right\|_{L^{q_{1}'}(\rr)}\\\nonumber
&\quad=
2^{k_{1}\widetilde{p}_{1}\alpha_{1}}G_{k_{1},\epsilon}(\ell)(x_{2}
,\ldots,x_{n})\left\|[E_{0}(\ell)(\cdot,x_{2},\ldots,x_{n})]^{q_{1}-1}
\one_{R_{k_{1}}}(\cdot)\right\|_{L^{q_{1}'}(\rr)}\\\nonumber
&\quad\quad\times[E_{1}(\ell)(x_{2},\ldots,x_{n})]^{\widetilde{q}_
{2}-\widetilde{p}_{1}}\widetilde{g}_{2,\epsilon}(\ell)(x_{2},
\ldots,x_{n}).\\\nonumber
&\quad
=2^{k_{1}\widetilde{p}_{1}\alpha_{1}}G_{k_{1},\epsilon}(\ell)(x_{
2},\ldots,x_{n})\left\|E_{0}(\ell)(\cdot,x_{2},\ldots,x_{n})
\one_{R_{k_{1}}}(\cdot)\right\|^{\f{q_{1}}{q'_{1}}}_{L^{q_{1}}(
\rr)}\\\nonumber
&\quad\quad\times[E_{1}(\ell)(x_{2},\ldots,x_{n})]^{\widetilde{q}
_{2}-\widetilde{p}_{1}}\widetilde{g}_{2,\epsilon}(\ell)(x_{2},
\ldots,x_{n})\\\nonumber
&\quad=2^{k_{1}\widetilde{p}_{1}\alpha_{1}} P^{(k_{1})}_
{\widetilde{p}_{1}-1,q_{1}}(\ell)(x_{2},\ldots,x_{n})h_{k_{1},
\epsilon}(\ell)(x_{2},\ldots,x_{n})\\\nonumber
&\quad\quad\times[E_{1}(\ell)(x_{2},\ldots,x_{n})]^{\widetilde{q}
_{2}-\widetilde{p}_{1}}\widetilde{g}_{2,\epsilon}(\ell)(x_{2}
,\ldots,x_{n}).\nonumber
\end{align*}

If $q_{1}=1$, then, using \eqref{pyw}, \eqref{GKj}, \eqref{Eiga}, and \eqref{2.9x}, we find that, for any given $\epsilon\in(0,\infty)$ and $k_{1}\in\zz$,
and for any $x_{2},\ldots,x_{n}\in\rr$,
\begin{align*}
&\left\|\widehat{g}_{1,\epsilon}(\ell)(\cdot,x_{2},\ldots,x_{n})\one_{R_
{k_{1}}}(\cdot)\right\|_{L^{\infty}(\rr)}\\
&\quad =2^{k_{1}\widetilde{p}_{1}\alpha_{1}}G_{k_{1},\epsilon}
(\ell)(x_{2},\ldots,x_{n})
\left\|[E_{0}(\ell)(\cdot,x_{2},\ldots,x_{n})]^{0}
\one_{R_{k_{1}}}(\cdot)\right\|_{L^{\infty}(\rr)}
\nonumber\\
&\quad\quad\times[E_{1}(\ell)(x_{2},\ldots,x_{n})]^{\widetilde{q}
_{2}-\widetilde{p}_{1}}\widetilde{g}_{2,\epsilon}(\ell)(x_
{2},\ldots,x_{n})\\\nonumber
&\quad=2^{k_{1}\widetilde{p}_{1}\alpha_{1}} P^{(k_{1})}_
{\widetilde{p}_{1}-1,1}(\ell)(x_{2},\ldots,x_{n})h_{k_{1},\epsilon}
(\ell)(x_{2},\ldots,x_{n})\\\nonumber
&\quad\quad\times[E_{1}(\ell)(x_{2},\ldots,x_{n})]^
{\widetilde{q}_{2}-\widetilde{p}_{1}}\widetilde{g}_{2,\epsilon}
(\ell)(x_{2},\ldots,x_{n}).\nonumber	
\end{align*}

Finally, we consider the case $q_{1}=\infty$. Observe that, for any given
$\epsilon\in (0,\infty)$ and $k_{1}\in\zz$, and for any $x_{2},
\ldots,x_{n}\in\rr$, $F_{\epsilon}^{(k_{1})}(\ell)(x_{2},\ldots
,x_{n})\subset R_{k_{1}}$. We now consider two cases on
$|F_{\epsilon}^{(k_{1})}(\ell)(x_{2},\ldots,x_{n})|$.
If $|F_{\epsilon}^{(k_{1})}
(\ell)(x_{2},\ldots,x_{n})|\neq 0$, then, by \eqref{Hje}, \eqref{fksion},
\eqref{GKj},
\eqref{pyw}, \eqref{Eiga}, and \eqref{2.9x}, we obtain
\begin{align*}
&\left\|\widehat{g}_{1,\epsilon}(\ell)(\cdot,x_{2},\ldots,x_{n})\one_
{R_{k_{1}}}(\cdot)\right\|_{L^{1}(\rr)}\\
&\quad=2^{k_{1}\widetilde{p}_{1}\alpha_{1}}G_
{k_{1},\epsilon}(\ell)(x_{2},\ldots,x_{n})
\left\|[E_{0}(\ell)(\cdot,x_{2},\ldots,x_{n})]^{0}f_{k_{1},
\epsilon}(\cdot,x_{2},\ldots,x_{n})\one_{R_{k_{1}}}
(\cdot)\right\|_{L^{1}(\rr)}\\
&\quad\quad\times[E_{1}(\ell)(x_{2},\ldots,x_{n})]^
{\widetilde{q}_{2}-\widetilde{p}_{1}}\widetilde{g}_{2,
\epsilon}(\ell)(x_{2},\ldots,x_{n})\\
&\quad=2^{k_{1}\widetilde{p}_{1}\alpha_{1}}P^{(k_{1})}
_{\widetilde{p}_{1}-1,\infty}(\ell)(x_{2},\ldots,x_{n})
h_{k_{1},\epsilon}(\ell)(x_{2},\ldots,x_{n})\left\|\f{
\one _{F_{\epsilon}^{(k_{1})}(\ell)(x_{2},\ldots,
x_{n})}(\cdot)}{|F_{\epsilon}^{(k_{1})}(\ell)(x_{2},\ldots,x_
{n})|}\one_{R_{k_{1}}}(\cdot)\right\|_{L^{1}(\rr)}\\
&\quad\quad\times [E_{1}(\ell)(x_{2},\ldots,x_{n})]^{
\widetilde{q}_{2}-\widetilde{p}_{1}}\widetilde{g}
_{2,\epsilon}(\ell)(x_{2},\ldots,x_{n})\\
&\quad=2^{k_{1}\widetilde{p}_{1}\alpha_{1}}P^{(k_{1})
}_{\widetilde{p}_{1}-1,\infty}(\ell)(x_{2},\ldots,x_{n}
)h_{k_{1},\epsilon}(\ell)(x_{2},\ldots,x_{n})\\
&\quad\quad\times[E_{1}(\ell)(x_{2},\ldots,x_{n})]^{
\widetilde{q}_{2}-\widetilde{p}_{1}}\widetilde{g}
_{2,\epsilon}(\ell)(x_{2},\ldots,x_{n}).
\end{align*}
If $|F_{\epsilon}^{(k_{1})}(\ell)(x_{2},\ldots,x_{n})|=0$,
then, by \eqref{Fsion}, we conclude that
\begin{align}\label{FQar}
&\left\|E_{0}(\ell)(\cdot,x_{2},\ldots,x_{n})\one_{R_{k_{1}}}(
\cdot)\right\|_{L^{\infty}(\rr)}\\
&\quad\leq \sup_{x_{1}\in\rr\setminus
F_{\epsilon}^{(k_{1})}(f)(x_{2},\ldots,x_{n})}E_{0}
(\ell)(x_{1},\ldots,x_{n})\one_{R_{k_{1}}}(x_{1})\nonumber\\
&\quad\leq (1+\epsilon)^{-1}\left\|E_{0}(\ell)(\cdot,x_{2},\ldots,
x_{n})\one_{R_{k_{1}}}(\cdot)\right\|_{L^{\infty}(\rr)},\nonumber
\end{align}
which implies that $\|E_{0}(\ell)(\cdot,x_{2},\ldots,x_{n})
\one_{R_{k_{1}}}(\cdot)\|_{L^{\infty}(\rr)}=0$. By this, \eqref{fksion}, \eqref{Hje},
and \eqref{2.9x}, we obtain
\begin{align*}
&\left\|\widehat{g}_{1,\epsilon}(\ell)(\cdot,x_{2},\ldots,x_{n})\one
_{R_{k_{1}}}(\cdot)\right\|_{L^{1}(\rr)}\\
&\quad=0\\
&\quad=2^{k_{1}\widetilde{p}_{1}\alpha_{1}}P^{(k_{1})}_{
\widetilde{p}_{1}-1,\infty}(\ell)(x_{2},\ldots,x_{n})
h_{k_{1},\epsilon}(\ell)(x_{2},\ldots,x_{n})\\
&\qquad\times[E_{1}(\ell)(x_{2},\ldots,x_{n})]^{
\widetilde{q}_{2}-\widetilde{p}_{1}}\widetilde{g}_{2,
\epsilon}(\ell)(x_{2},\ldots,x_{n}).
\end{align*}
Thus, combining the above three cases on $q_{1}$, we have, for any
given $\epsilon\in(0,\infty)$, $k_{1}\in\zz$, and $q_{1}\in [1,\infty]$,
and for any $x_{2},\ldots,x_{n}\in\rr$,
\begin{align*}
&\left\|\widehat{g}_{1,\epsilon}(\ell)(\cdot,x_{2},\ldots,x_{n})\one_
{R_{k_{1}}}(\cdot)\right\|_{L^{q_{1}'}(\rr)}\\
&\quad=2^{k_{1}\widetilde{p}
_{1}\alpha_{1}} P^{(k_{1})}_{\widetilde{p}_{1}-1,
q_{1}}(\ell)(x_{2},\ldots,x_{n})h_{k_{1},\epsilon}(\ell)
(x_{2},\ldots,x_{n})\\
&\qquad\times[E_{1}(\ell)(x_{2},\ldots,x_{n})]^{\widetilde
{q}_{2}-\widetilde{p}_{1}}\widetilde{g}_{2,\epsilon}
(\ell)(x_{2},\ldots,x_{n}).
\end{align*}

Then, for any given $\epsilon\in(0,\infty)$, $q_{1}\in[1,\infty]$,
and $\alpha_{1}\in\rr$,
and for any $x_{2},\ldots,x_{n}\in\rr$, we consider the following
three cases on $p_{1}$.

If $p_{1}\in(1,\infty)$, then,  from Definition \ref{chz}, \eqref{pyw},
\eqref{hke}, \eqref{2.9x}, \eqref{Eiga}, and \eqref{2.41y},
we deduce that
\begin{align*}
&\left\|\widehat{g}_{1,\epsilon}(\ell)(\cdot,x_{2},\ldots,x_{n})
\right\|_{\dky(\rr)}\\
&\quad=\left\{
\sum_{k_{1}\in\zz}2^{k_{1}\alpha_{1}p_{1}'(p_{1}-1)}\left[
P^{(k_{1})}_{p_{1}-1,q_{1}}(\ell)(x_{2},\ldots,x_{n})
h_{k_{1},\epsilon}(\ell)(x_{2},\ldots,x_{n})\right]^{p_{1}'}
\right\}^{\f{1}{p_{1}'}}\nonumber\\
&\quad\quad\times[E_{1}(\ell)(x_{2},\ldots,x_{n})]^{
\widetilde{q}_{2}-p_{1}}\widetilde{g}_{2,\epsilon}(\ell)
(x_{2},\ldots,x_{n})\nonumber\\
&\quad=[E_{1}(\ell)(x_{2},\ldots,x_{n})]^{p_{1}/p_{1}'}
[E_{1}(\ell)(x_{2},\ldots,x_{n})]^{\widetilde{q}_{2}-p_{1}}
\widetilde{g}_{2,\epsilon}(\ell)(x_{2},\ldots,x_{n})\nonumber\\
&\quad=[E_{1}(\ell)(x_{2},\ldots,x_{n})]^{\widetilde{q}_{2}-1}
\widetilde{g}_{2,\epsilon}(\ell)(x_{2},\ldots,x_{n})=\widehat{g}_{2,
\epsilon}(\ell)(x_{2},\ldots,x_{n}).\nonumber
\end{align*}

If $p_{1}=1$, then, by Definition \ref{chz}, \eqref{pyw}, \eqref{2.9x},
and \eqref{hke}, we obtain
\begin{align*}
&\left\|\widehat{g}_{1,\epsilon}(\ell)(\cdot,x_{2},\ldots,x_{n})\right\|
_{\dot{K}_{q'_{1}}^{-\alpha_{1},\infty}(\rr)}\\
&\quad =\sup_{k_{1}\in\zz}\left\{2^{-k_{1}\alpha_{1}}
2^{k_{1}\alpha_{1}}P^{(k_{1})}_{0,q_{1}}(\ell)(x_{2},\ldots,
x_{n})h_{k_{1},\epsilon}(\ell)(x_{2},\ldots,x_{n})\right\}\\
&\quad\quad\times[E_{1}(\ell)(x_{2},\ldots,x_{n})]^
{\widetilde{q}_{2}-1}\widetilde{g}_{2,\epsilon}(\ell)(x_{2},
\ldots,x_{n})\nonumber\\
&\quad
=[E_{1}(\ell)(x_{2},\ldots,x_{n})]^{\widetilde{q}_{2}-1}
\widetilde{g}_{2,\epsilon}(\ell)(x_{2},\ldots,x_{n})=\widehat{
g}_{2,\epsilon}(\ell)(x_{2},\ldots,x_{n}).\nonumber
\end{align*}

If $p_{1}=\infty$, then, using Definition \ref{chz}, we find that
\begin{align*}
&\left\|\widehat{g}_{1,\epsilon}(\ell)(\cdot,x_{2},\ldots,x_{n})\right\|
_{\dot{K}^{-\alpha_{1},1}_{q_{1}'}(\rr)}\\
&\quad=
\sum_{k_{1}\in\zz}2^{-k_{1}\alpha_{1}}\cdot2^{k_{1}\alpha_{
1}}P^{(k_{1})}_{0,q_{1}}(\ell)(x_{2},\ldots,x_{n})h
_{k_{1},\epsilon}(\ell)(x_{2},\ldots,x_{n})\\
&\quad\quad\times[E_{1}(\ell)(x_{2},\ldots,x_{n})]^{\widetilde{q}_{2}-1}
\widetilde{g}_{2,\epsilon}(\ell)
(x_{2},\ldots,x_{n}).
\end{align*}
If $E_{1,\epsilon}(\ell)(x_{2},\ldots,x_{n})\neq \emptyset$,
since $\supp(\ell)\subset A_{k_{0}}$, from the definitions
of both $A_{k_{0}}$ and $E_{0}(\ell)$ in \eqref{Eiga}, it follows that $E_
{1,\epsilon}(\ell)(x_{2},\ldots,x_{n})$ is a finite set. By this, \eqref{hke}, \eqref{Esion}, and \eqref{2.41y},
we conclude that
\begin{align*}
&\left\|\widehat{g}_{1,\epsilon}(\ell)(\cdot,x_{2},\ldots,x_{n})\right\|
_{\dot{K}^{-\alpha_{1},1}_{q_{1}'}(\rr)}\\
&\quad=\sum_{k_{1}\in\zz}
P^{(k_{1})}_{0,q_{1}}(\ell)(x_{2},\ldots,x_{n})\f{\one_{E
_{1,\epsilon}(\ell)(x_{2},\ldots,x_{n})}(k_{1})}{
\#[E_{1,\epsilon}(\ell)(x_{2},\ldots,x_{n}) ]}\\
&\quad\quad\times
[E_{1}(\ell)(x_{2},\ldots,x_{n})]^{\widetilde{q}_{2}-1}
\widetilde{g}_{2,\epsilon}(\ell)(x_{2},\ldots,x_{n})\\
&\quad=[E_{1}(\ell)(x_{2},\ldots,x_{n})]^{\widetilde{q}_{2}-1}
\widetilde{g}_{2,\epsilon}(\ell)(x_{2},\ldots,x_{n})=\widehat{g}
_{2,\epsilon}(\ell)(x_{2},\ldots,x_{n}).
\end{align*}
If  $E_{1,\epsilon}(\ell)(x_{2},\ldots,x_{n})=\emptyset$,
then, from \eqref{Esion}, we infer that, for any $k_{1}\in\zz$,
$$\left\|E_{0}(\ell)
(\cdot,x_{2},\ldots,x_{n})\one_{R_{k_{1}}}(\cdot)\right\|_
{L^{q_{1}}(\rr)}=0.$$
Using this and \eqref{hke}, we find that
\begin{align*}
&\left\|\widehat{g}_{1,\epsilon}(\ell)(\cdot,x_{2},\ldots,x_{n})\right\|
_{\dot{K}^{-\alpha_{1},1}_{q_{1}'}(\rr)}\\
&\quad=0
=[E_{1}(\ell)(x_{2},\ldots,x_{n})]^{\widetilde{q}_{2}-1}
\widetilde{g}_{2,\epsilon}(\ell)(x_{2},\ldots,x_{n})\\
&\quad=\widehat{g}_{2,\epsilon}(\ell)(x_{2},\ldots,x_{n}).
\end{align*}
Therefore,  we conclude that, for any given $\epsilon\in(0,\infty)$, $p_{1}$, $q_{1}
\in[1,\infty],$ and $\alpha_{1}\in\rr$, and for any $x_{2},
\ldots,x_{n}\in\rr$,
\begin{align*}
\left\|\widehat{g}_{1,\epsilon}(\ell)(\cdot,x_{2},\ldots,x_{n})\right\|
_{\dky(\rr)}=\widehat{g}_{2,\epsilon}(\ell)(x_{2},\ldots,x_{n}).
\end{align*}
Similarly to this estimation, we have, for any given $\epsilon\in(0,\infty)$
and $j\in\{2,\ldots,n-1\}$, and for any $x_{j+1},\ldots,x_{n}\in\rr$,
\begin{align*}
\left\|\widehat{g}_{1,\epsilon}(\ell)(\cdot,x_{j+1},\ldots,x_{n})
\right\|_{\dot{E}_{\vec{q}'_{j}}^{-\vec{\alpha}_{j},\vec{p}
'_{j}}(\rr^{j})}&=\left\|\widehat{g}_{j,\epsilon}(\ell)(
\cdot,x_{j+1},\ldots,x_{n})\right\|_{\dot{K}_{q_{j}'}^{
-\alpha_{j},p_{j}'}(\rr)}\\
&=\widehat{g}_{j+1,
\epsilon}(\ell)(x_{j+1},\ldots,x_{n}).
\end{align*}
By this and \eqref{2.41y}, we conclude that, for any given $\epsilon\in(0,\infty)$,
\begin{equation}\label{deeo}
\left\|\widehat{g}_{1,\epsilon}(\ell)\right\|_{\diihz(\rn)}=\left\|\widehat{g}
_{n,\epsilon}(\ell)\right\|_{\dot{K}_{q'_{n}}^{-\alpha_{n},p'_{n}}
(\rr)}=\left\|[E_{n-1}(\ell)]^{\widetilde{q}_{n}-1
}\widetilde{g}_{n,\epsilon}(\ell)
\right\|_{\dot{K}_{q'_{n}}^{-\alpha_{n},p'_{n}}(\rr)}.
\end{equation}
Moreover, from \eqref{2.18xx} and \eqref{2.17x},
it follows that, for any given $\epsilon\in(0,\infty)$ and $k_{n}\in\zz$,
\begin{align*}
&\left\|[E_{n-1}(\ell)]^{\widetilde{q}_{n}-1}\widetilde{g}
_{n,\epsilon}(\ell)\one_{R_{k_{n}}}\right\|_{L^{q'_{n}}(\rr)}\\
&\quad= \left\|[E_{n-1}(\ell)]^{\widetilde{q}_{n}-1} g_{n,
\epsilon}(\ell) H_{n,\epsilon}(\ell)\one_{R_{k_{n}}}
\right\|_{L^{q'_{n}}(\rr)}[E_{n}(\ell)]^{1-\widetilde{p}_{n}}\\
&\quad=
2^{k_{n}\widetilde{p}_{n}\alpha_{n}}G_{k_{n},
\epsilon}(\ell)\left\|[E_{n-1}(\ell)]^{\widetilde{q}_{n}-1}
H_{n,\epsilon}(\ell)\one_{R_{k_{n}}} \right\|_{L^{q'_{n}}
(\rr)}[E_{n}(\ell)]^{1-\widetilde{p}_{n}}.
\end{align*}
Then, for any given $\epsilon\in(0,\infty)$ and $k_{n}\in\zz$,
we consider the following three cases on $q_{n}$.
If $q_{n}\in (1,\infty)$, then, by \eqref{Hje}, \eqref{pyw}, \eqref{2.12x},
and \eqref{2.9x},
we conclude that
\begin{align*}
&\left\|[E_{n-1}(\ell)]^{\widetilde{q}_{n}-1}\widetilde{g}_{n,
\epsilon}(\ell)\one_{R_{k_{n}}}\right\|_{L^{q'_{n}}(\rr)}\\
&\quad=
2^{k_{n}\widetilde{p}_{n}\alpha_{n}}G_{k_{n},\epsilon}
(\ell)\left\|[E_{n-1}(\ell)]^{q_{n}-1} \one_{R_{k_{n}}} \right\|
_{L^{q'_{n}}(\rr)}[E_{n}(\ell)]^{1-\widetilde{p}_{n}}\\
&\quad=
2^{k_{n}\widetilde{p}_{n}\alpha_{n}}G_{k_{n},\epsilon}
(\ell)\left\|E_{n-1}(\ell) \one_{R_{k_{n}}} \right\|_{L^{q_{n}}(\rr)}
^{\f{q_{n}}{q_{n}'}}[E_{n}(\ell)]^{1-\widetilde{p}_{n}}\\
&\quad=
2^{k_{n}\widetilde{p}_{n}\alpha_{n}}P_{\widetilde{p}
_{n}-1,q_{n}}^{(k_{n})}(\ell)h_{k_{n},\epsilon}
(\ell)[E_{n}(\ell)]^{1-\widetilde{p}_{n}}.
\end{align*}
If $q_{n}=1$, then, from \eqref{Hje}, \eqref{fksion}, \eqref{2.12x},
\eqref{pyw}, and \eqref{2.9x}, it follows that
\begin{align*}
&\left\|[E_{n-1}(\ell)]^{\widetilde{q}_{n}-1}\widetilde{g}
_{n,\epsilon}(\ell)\one_{R_{k_{n}}}\right\|_{L^{\infty}(\rr)}\\
&\quad=
2^{k_{n}\widetilde{p}_{n}\alpha_{n}}G_{k_{n},\epsilon}
(\ell)\left\|[E_{n-1}(\ell)]^{0} \one_{R_{k_{n}}} \right\|_{L^{\infty}
(\rr)}[E_{n}(\ell)]^{1-\widetilde{p}_{n}}\\
&\quad=
2^{k_{n}\widetilde{p}_{n}\alpha_{n}}P_{\widetilde{p}
_{n}-1,1}^{(k_{n})}(\ell)h_{k_{n},\epsilon}
(\ell)[E_{n}(\ell)]^{1-\widetilde{p}_{n}}.
\end{align*}
We now consider the case $q_{n}=\infty$.  In this case, we consider the following two
cases on $|F^{(k_{n})}_{\epsilon}(\ell)|$. If $|F^{(k_{n})}
_{\epsilon}(\ell)|\neq 0$, then we first observe that,
for any given $\epsilon\in(0,\infty)$ and
$k_{n}\in\zz$, $F^{(k_{n})}_{\epsilon}
(\ell)\subset R_{k_{n}}$. Applying this, the definitions of both
$H_{n,\epsilon}(\ell)$ in \eqref{Hje} and $f_{k_{n},\epsilon}(\ell)$,
and \eqref{2.12x},
we find that
\begin{align*}
&\left\|[E_{n-1}(\ell)]^{\widetilde{q}_{n}-1}\widetilde{g}_
{n,\epsilon}(\ell)\one_{R_{k_{n}}}\right\|_{L^{1}(\rr)}\\
&\quad=
2^{k_{n}\widetilde{p}_{n}\alpha_{n}}G_{k_{n},\epsilon}
(\ell)\left\|f_{k_{n},\epsilon}(\ell) \one_{R_{k_{n}}} \right\|_{L^{1}(\rr)
}[E_{n}(\ell)]^{1-\widetilde{p}_{n}}\\
&\quad=
2^{k_{n}\widetilde{p}_{n}\alpha_{n}}P_{\widetilde{p}_
{n}-1,\infty}^{(k_{n})}(\ell)h_{k_{n},\epsilon}
(\ell)[E_{n}(\ell)]^{1-\widetilde{p}_{n}}.
\end{align*}
If $|F^{(k_{n})}_{\epsilon}(\ell)|= 0$, then, similarly to the
estimation of
\eqref{FQar}, we have $\|E_{n-1}(\ell)\one_{R_{k_{n}}}\|
_{L^{\infty}(\rr)}=0$. From this, \eqref{Hje}, \eqref{fksion},
and \eqref{2.9x}, we deduce that
\begin{equation*}
\left\|[E_{n-1}(\ell)]^{\widetilde{q}_{n}-1}\widetilde{g}_{n,
\epsilon}(\ell)\one_{R_{k_{n}}}\right\|_{L^{1}(\rr)}=0=2^{k_{n}
\widetilde{p}_{n}\alpha_{n}}P_{\widetilde{p}_{n}-1,\infty}
^{(k_{n})}(\ell)h_{k_{n},\epsilon}
(\ell)[E_{n}(\ell)]^{1-\widetilde{p}_{n}}.
\end{equation*}
Thus, for any given $\epsilon\in(0,\infty)$, $p_{n},$ $q_{n}\in [1,\infty]$,
and $k_{n}\in\zz$, we obtain
\begin{equation}\label{2.28x}
\left\|[E_{n-1}(\ell)]^{\widetilde{q}_{n}-1}\widetilde{g}_{n,\epsilon}
(\ell)\one_{R_{k_{n}}}\right\|_{L^{q'_{n}}(\rr)}=2^{k_{n}\widetilde{p}_
{n}\alpha_{n}}P^{(k_{n})}_{\widetilde{p}_{n}-1,q_{n}}(\ell)h_
{k_{n},\epsilon}(\ell) [E_{n}(\ell)]^{1-\widetilde{p}_{n}}.
\end{equation}
Then we consider the following three cases on
$p_{n}$. If $p_{n}\in(1,\infty)$, then, by \eqref{deeo},
\eqref{2.28x}, $E_{n}(\ell)\neq 0$, and the definitions of both
$\dot{K}_{q'_{n}}^{-\alpha_{n},\infty}(\rr)$
and $P^{(k_{n})}_{\widetilde{p}_{n}-1,q_{n}}(\ell)$, we find that
\begin{align*}
\left\|\widehat{g}_{1,\epsilon}(\ell)\right\|_{\diihz(\rn)}&=\left\{
\sum_{k_{n}\in\zz}2^{k_{n}\alpha_{n}p_{n}'(p_{n}-1)}\left\|
E_{n-1}(\ell)\one_{R_{k_{n}}}\right\|_{L^{q_{n}}(\rr)}^{p_{n}'(p_{n}-1)}
\right\}^{\f{1}{p_{n}'}}[E_{n}(\ell)]^{1-p_{n}}\nonumber\\
&=[E_{n}(\ell)]^{\f{p_{n}}{p_{n}'}}[E_{n}(\ell)]^{1-p_{n}}
=1.\nonumber
\end{align*}
If $p_{n}=1$, then, by \eqref{deeo}, \eqref{2.28x},
\eqref{hke}, and $E_{n}(\ell)\neq 0$, we have
\begin{align*}
\left\|\widehat{g}_{1,\epsilon}(\ell)\right\|_{\diihz(\rn)}&=\left\|[E_{n-1}(\ell)]^
{q_{n}-1}\widetilde{g}_{n,\epsilon}(\ell)\right\|_{\dot{K}_{q'_{n}}^
{-\alpha_{n},\infty}(\rr)}\\
&=\sup_{k_{n}\in\zz}\left\{2^{-k_{n}\alpha_{n}}2^{k_{n}
\alpha_{n}} \left\|E_{n-1}(\ell)\one_{R_{k_{n}}}\right\|_
{L^{q_{n}}(\rr)}^{0} \right\}=1.
\end{align*}
If $p_{n}=\infty$, then, from \eqref{deeo}, \eqref{2.28x}, \eqref{pyw},
\eqref{hke}, \eqref{Esion}, \eqref{2.9x}, and $E_{n}(\ell)\neq 0$, we deduce
that
\begin{align*}
\|\widehat{g}_{1,\epsilon}\|_{\diihz(\rn)}&=\sum_{k_{n}\in\zz}
2^{-k_{n}\alpha_{n}}2^{k_{n}\alpha_{n}}P^{(k_{n})}_{0,
q_{n}}(\ell)h_{k_{n},\epsilon}(\ell)\\
&=\sum_{k_{n}\in\zz}P_{0,q_{n}}^{(k_{n})}(\ell)\f{\one_{E_{n,\epsilon}(\ell)}}
{\#[E_{n,\epsilon}(\ell)]}=1.
\end{align*}
Thus, \eqref{c1} holds true.

Next, we show \eqref{c22}. To this end, using \eqref{2.41x}, \eqref{2.40y}, and
\eqref{2.40x}, we conclude that, for any given $\epsilon\in(0,\infty)$ and for any
$x_{2},\ldots,x_{n}\in\rr$,
\begin{align}\label{ifg-2}
&\int_{\rr}\ell(x_{1},\ldots,x_{n})\widehat{g}_{1,\epsilon}(\ell)
(x_{1},\ldots,x_{n})\,dx_{1}\\\nonumber
&\quad=\int_{\rr}[E_{0}(\ell)(x_{1},\ldots,x_{n})]^
{\widetilde{q}_{1}}\widetilde{g}_{1,\epsilon}(\ell)(x_{1},
\ldots,x_{n})\,dx_{1}\\\nonumber
&\quad=\int_{\rr}[E_{0}(\ell)(x_{1},\ldots,x_{n})]^
{\widetilde{q}_{1}}g_{1,\epsilon}(\ell)(x_{1},\ldots,x_{n})
H_{1,\epsilon}(\ell)(x_{1},\ldots,x_{n})\,dx_{1}\\\nonumber
&\qquad\times [E_{1}(\ell)(x_{2},\ldots,x_{n})]^{
\widetilde{q}_{2}-\widetilde{p}_{1}}\widetilde{g}
_{2,\epsilon}(\ell)(x_{2},\ldots,x_{n})\\\nonumber
&\quad=\sum_{k_{1}\in\zz}2^{k_{1}\widetilde{p}_{1}
\alpha_{1}}G_{k_{1},\epsilon}(\ell)(x_{2},\ldots,x_{n})\\\nonumber
&\qquad\times\int_{\rr}[E_{0}(\ell)(x_{1},\ldots,x_{n})]^{\widetilde{q}_{1}}
H_{1,\epsilon}(\ell)(x_{1},\ldots,x_{n})\one_{R_{k_{1}}}
(x_{1})\,dx_{1}\\\nonumber
&\qquad\times
[E_{1}(\ell)(x_{2},\ldots,x_{n})]^{\widetilde{q}_{2}-
\widetilde{p}_{1}}\widetilde{g}_{2,\epsilon}
(\ell)(x_{2},\ldots,x_{n}).\nonumber
\end{align}
Then, for any given $\epsilon\in(0,\infty)$ and
for any $x_{2},\ldots,x_{n}\in\rr$,
we consider the following four cases on both $p_{1}$ and
$q_{1}$. If $p_{1}\in[1,\infty)$ and $q_{1}\in[1,\infty)$,
then $G_{k_{1},\epsilon}(\ell)(x_{2},\ldots,x_{n})=
P^{(k_{1})}_{p_{1}-q_{1},q_{1}}(\ell)(x_{2},\ldots,x_{n})$
and $H_{1,\epsilon}(\ell)(x)=1$.
By this, \eqref{ifg-2}, \eqref{pyw}, \eqref{2.9x}, and
\eqref{Eiga}, we conclude that
\begin{align*}
&\int_{\rr}\ell(x_{1},\ldots,x_{n})\widehat{g}_{1,\epsilon}
(\ell)(x_{1},\ldots,x_{n})\,dx_{1}\\
&\quad=\sum_{k_{1}\in\zz}2^{k_{1}p_{1}\alpha_{1}}
P^{(k_{1})}_{p_{1}-q_{1},q_{1}}(\ell)(x_{2},\ldots,x_{n})
\|E_{0}(\ell)(\cdot,x_{2},\ldots,x_{n})\one_{R_{k_{1}}}
(\cdot)\|_{L^{q_{1}}(\rr)}^{q_{1}}\\
&\quad\quad\times[E_{1}(\ell)(x_{2},\ldots,x_{n})]^
{\widetilde{q}_{2}-p_{1}}\widetilde{g}_{2,\epsilon}(\ell)
(x_{2},\ldots,x_{n})\\
&\quad=\sum_{k_{1}\in\zz}2^{k_{1}p_{1}\alpha_{1}}\|
E_{0}(\ell)(\cdot,x_{2},\ldots,x_{n})\one_{R_{k_{1}}}(\cdot)
\|_{L^{q_{1}}(\rr)}^{p_{1}}\\
&\quad\quad\times[E_{1}(\ell)(x_{2},\ldots,x_{n})]^
{\widetilde{q}_{2}-p_{1}}\widetilde{g}_{2,\epsilon}(\ell)
(x_{2},\ldots,x_{n})\\
&\quad=[E_{1}(\ell)(x_{2},\ldots,x_{n})]^{p_{1}}
[E_{1}(\ell)(x_{2},\ldots,x_{n})]^{\widetilde{q}_{2}-p_{1}}
\widetilde{g}_{2,\epsilon}(\ell)(x_{2},\ldots,x_{n})\\
&\quad\geq (1+\epsilon)^{-2}[E_{1}(\ell)(x_{2},\ldots,
x_{n})]^{\widetilde{q}_{2}}\widetilde{g}_{2,\epsilon}
(\ell)(x_{2},\ldots,x_{n}).
\end{align*}
If $p_{1}=\infty$ and $q_{1}\in[1,\infty)$, then
$$G_{k_{1},\epsilon}(\ell)(x_{2},\ldots,x_{n})=P^{(k_{1})}
_{1-q_{1},q_{1}}(\ell)(x_{2},\ldots,x_{n})
h_{k_{1},\epsilon}(\ell)(x_{2},\ldots,x_{n})$$
and $H_{1,\epsilon}(\ell)(x)=1$. Applying this, \eqref{ifg-2},
\eqref{pyw}, and \eqref{2.9x},
we obtain
\begin{align}\label{2.30xx}
&\int_{\rr}\ell(x_{1},\ldots,x_{n})\widehat{g}_{1,\epsilon}
(\ell)(x_{1},\ldots,x_{n})\,dx_{1}\\
&\quad=
\sum_{k_{1}\in\zz}2^{k_{1}\alpha_{1}}P^{(k_{1})}_
{1-q_{1},q_{1}}(\ell)(x_{2},\ldots,x_{n})h_{k_{1},
\epsilon}(\ell)(x_{2},\ldots,x_{n})\|E_{0}(\ell)(\cdot,
x_{2},\ldots,x_{n})\one_{R_{k_{1}}}(\cdot)\|_{L^{q_{1}}
(\rr)}^{q_{1}}\nonumber\\
&\quad\quad\times
[E_{1}(\ell)(x_{2},\ldots,x_{n})]^{\widetilde{q}_{2}-1}
\widetilde{g}_{2,\epsilon}(\ell)(x_{2},\ldots,x_{n})\nonumber\\
&\quad=
\sum_{k_{1}\in\zz}2^{k_{1}\alpha_{1}}P^{(k_{1})}_{1,
q_{1}}(\ell)(x_{2},\ldots,x_{n})h_{k_{1},\epsilon}(\ell)
(x_{2},\ldots,x_{n})\nonumber\\
&\quad\quad\times
[E_{1}(\ell)(x_{2},\ldots,x_{n})]^{\widetilde{q}_{2}-1}
\widetilde{g}_{2,\epsilon}(\ell)(x_{2},\ldots,x_{n}).\nonumber
\end{align}
Next, we consider the following two cases on
$E_{1,\epsilon}(\ell)(x_{2},\ldots,x_{n})$.
If $E_{1,\epsilon}(\ell)(x_{2},\ldots,x_{n})=\emptyset$, then,
by \eqref{hke} and \eqref{Esion}, we find that $h_{k_{1},\epsilon}(\ell)
(x_{2},\ldots,x_{n})=E_{1}(\ell)(x_{2},\ldots,x_{n})=0$. Thus,
\begin{align*}
&\sum_{k_{1}\in\zz}2^{k_{1}\alpha_{1}}P^{(k_{1})}_
{1,q_{1}}(\ell)(x_{2},\ldots,x_{n})h_{k_{1},\epsilon}(\ell)
(x_{2},\ldots,x_{n})\\
&\quad=0=(1+\epsilon)^{-1}E_{1}(\ell)(x_{2},\ldots,x_{n}).
\end{align*}
If $E_{1,\epsilon}(\ell)(x_{2},\ldots,x_{n})\neq\emptyset$, then, since, for
any $k_{1}\in E_{1,\epsilon}(\ell)(x_{2},\ldots,x_{n})$,
$$\|E_{0}(\ell)(\cdot,x_{2},\ldots,x_{n})
\one_{R_{k_{1}}}(\cdot)\|_{L^{q_{1}}(\rr)}>0,$$
from \eqref{2.30xx}, \eqref{Esion}, and \eqref{2.9x},
we deduce that, for any $k_{1}\in E_{1,\epsilon}(\ell)(x_{2},\ldots,x_{n}),$ $$2^{k_{1}\alpha_{1}}P^{(k_{1})}_{1,
q_{1}}(\ell)(x_{2},\ldots,x_{n})>(1+\epsilon)^{-1}E_{1}(\ell)(x_{2},\ldots,x_{n}).$$
Therefore, we have
\begin{align*}
\sum_{k_{1}\in\zz}2^{k_{1}\alpha_{1}}P^{(k_{1})}_
{1,q_{1}}(\ell)(x_{2},\ldots,x_{n})h_{k_{1},\epsilon}(\ell)
(x_{2},\ldots,x_{n})\geq (1+\epsilon)^{-1}E_{1}(\ell)(x_{2},\ldots,x_{n}).
\end{align*}
Using this and \eqref{2.30xx}, we conclude that
\begin{align*}
	&\int_{\rr}\ell(x_{1},\ldots,x_{n})\widehat{g}_{1,\epsilon}
	(\ell)(x_{1},\ldots,x_{n})\,dx_{1}\\
	&\quad\geq
	(1+\epsilon)^{-2}[E_{1}(\ell)(x_{2},\ldots,
	x_{n})]^{\widetilde{q}_{2}}\widetilde{g}_{2,\epsilon}
	(\ell)(x_{2},\ldots,x_{n}).
\end{align*}
If $p_{1}\in[1,\infty)$ and $q_{1}=\infty$, then
$G_{k_{1},\epsilon}(\ell)(x_{2},\ldots,x_{n})=P^{(k_{1})}_
{p_{1}-1,\infty}(\ell)(x_{2},\ldots,x_{n})$ and
$$H_{1,\epsilon}
(\ell)(x)=\sum_{k_{1}\in\zz}f_{k_{1},\epsilon}(\ell)(x_{1},\ldots,x_{n})
\one_{R_{k_{1}}}(x_{1}).$$
From this, \eqref{ifg-2}, \eqref{fksion},
and \eqref{Fsion}, it follows that
\begin{align}\label{2.30yy}
&\int_{\rr}\ell(x_{1},\ldots,x_{n})\widehat{g}_{1,\epsilon}(\ell)
(x_{1},\ldots,x_{n})\,dx_{1}\\
&\quad=\sum_{k_{1}\in\zz}2^{k_{1}p_{1}\alpha_{1}}P^{(k_{1})}
_{p_{1}-1,\infty}(\ell)(x_{2},\ldots,x_{n})\nonumber\\
&\qquad\times\int_{\rr}
E_{0}(\ell)(x_{1},\ldots,x_{n})f_{k_{1},\epsilon}(\ell)(x_{1},\ldots,
x_{n})\one_{R_{k_{1}}}(x_{1})\,dx_{1}\nonumber\\
&\qquad\times[E_{1}(\ell)(x_{2},\ldots,x_{n})]^{\widetilde{q}_{2}-p_{1}}
\widetilde{g}_{2,\epsilon}(\ell)(x_{2},\ldots,x_{n}).\nonumber
\end{align}
Then we consider the following two cases on $|F_{\epsilon}^{(k_{1})}(\ell)
(x_{2},\ldots,x_{n})|$. If $|F_{\epsilon}^{(k_{1})}(\ell)
(x_{2},\ldots,x_{n})|=0$, then we have
\begin{align*}
	&\int_{\rr}
	E_{0}(\ell)(x_{1},\ldots,x_{n})f_{k_{1},\epsilon}(\ell)(x_{1},\ldots,
	x_{n})\one_{R_{k_{1}}}(x_{1})\,dx_{1}\\
	&\quad=0=(1+\epsilon)^{-1}\|E_{0}(\ell)(\cdot,x_{2},
	\ldots,x_{n})\one_{R_{k_{1}}}(\cdot)\|_{L^{\infty}(\rr)}.
\end{align*}
If $|F_{\epsilon}^{(k_{1})}(\ell)(x_{2},\ldots,x_{n})|\neq 0$, then,
by both \eqref{Fsion} and \eqref{fksion},
we obtain
\begin{align*}
	&\int_{\rr}
	E_{0}(\ell)(x_{1},\ldots,x_{n})f_{k_{1},\epsilon}(\ell)(x_{1},\ldots,
	x_{n})\one_{R_{k_{1}}}(x_{1})\,dx_{1}\\
	&\quad>(1+\epsilon)^{-1}\|E_{0}(\ell)(\cdot,x_{2},
	\ldots,x_{n})\one_{R_{k_{1}}}(\cdot)\|_{L^{\infty}(\rr)}.
\end{align*}
Applying this and \eqref{2.30yy}, we find that
\begin{align*}
	&\int_{\rr}\ell(x_{1},\ldots,x_{n})\widehat{g}_{1,\epsilon}(\ell)
	(x_{1},\ldots,x_{n})\,dx_{1}\\
	&\quad\geq (1+\epsilon)^{-1} \sum_{k_{1}\in\zz}2^{k_{1}p_{1}
	\alpha_{1}}P^{(k_{1})}_{p_{1}
	-1,\infty}(\ell)(x_{2},\ldots,x_{n})\|E_{0}(\ell)(\cdot,x_{2},
	\ldots,x_{n})\one_{R_{k_{1}}}(\cdot)\|_{L^{\infty}(\rr)}\\
	&\qquad\times
	[E_{1}(\ell)(x_{2},\ldots,x_{n})]^{\widetilde{q}_{2}-p_{1}}
	\widetilde{g}_{2,\epsilon}(\ell)(x_{2},\ldots,x_{n})\\
	&\quad=(1+\epsilon)^{-1}[E_{1}(\ell)(x_{2},\ldots,x_{n})]
	^{p_{1}}[E_{1}(\ell)(x_{2},\ldots,x_{n})]^{\widetilde{q}_{2}
	-p_{1}}\widetilde{g}_{2,\epsilon}(\ell)(x_{2},
	\ldots,x_{n})\\
	&\quad\geq (1+\epsilon)^{-2}[E_{1}(\ell)(x_{2},\ldots,x_{n})]
	^{\widetilde{q}_{2}}\widetilde{g}_{2,\epsilon}(\ell)(x_{2},
	\ldots,x_{n}).
\end{align*}
If $p_{1}=q_{1}=\infty$, then
$$G_{k_{1},\epsilon}(\ell)(x_{2},\ldots,x_{n})=P^{(k_{1})}_{0,
\infty}(\ell)(x_{2},\ldots,x_{n})h_{k_{1},\epsilon}(\ell)(
x_{2},\ldots,x_{n})$$ and $$H_{1,\epsilon}(\ell)(x_{1},\ldots,
x_{n})=\sum_{k_{1}\in\zz}f_{k_{1},\epsilon}(\ell)(x_{1},\ldots,
x_{n})\one_{R_{k_{1}}}(x_{1}).$$ By this, similarly to the estimations of both
\eqref{2.30xx} and \eqref{2.30yy}, we conclude that
\begin{align*}
&\int_{\rr}\ell(x_{1},\ldots,x_{n})\widehat{g}_{1,\epsilon}(\ell)(
x_{1},\ldots,x_{n})\,dx_{1}\\
&\quad=\sum_{k_{1}\in\zz}2^{k_{1}\alpha_{1}}P^{(k_{1})}_
{0,\infty}(\ell)(x_{2},\ldots,x_{n})h_{k_{1},\epsilon}(\ell)(
x_{2},\ldots,x_{n})\\
&\qquad\times\int_{\rr}E_{0}(\ell)(x_{1},\ldots,x_{n})f
_{k_{1},\epsilon}(\ell)(x_{1},\ldots,x_{n})\one_{R_{k_{1}}}(x_
{1})\,dx_{1}\\
&\qquad\times
[E_{1}(\ell)(x_{2},\ldots,x_{n})]^{\widetilde{q}_{2}-1}
\widetilde{g}_{2,\epsilon}(\ell)(x_{2},\ldots,x_{n})\\
&\quad\geq (1+\epsilon)^{-2}
E_{1}(\ell)(x_{2},\ldots,x_{n})
[E_{1}(\ell)(x_{2},\ldots,x_{n})]^{\widetilde{q}_{2}-1}
\widetilde{g}_{2,\epsilon}(\ell)(x_{2},\ldots,x_{n})\\
&\quad= (1+\epsilon)^{-2}[E_{1}(\ell)(x_{2},\ldots,x_{n})]^{
\widetilde{q}_{2}}\widetilde{g}_{2,\epsilon}(\ell)(x_{2},
\ldots,x_{n}).
\end{align*}
Thus, for any given $\epsilon\in(0,\infty)$, $p_{1}$, $q_{1}\in[1,\infty]$,
and $\alpha_{1}\in\rr$, and for any $x_{2},\ldots,x_{n}\in\rr$, we have
\begin{align*}
&\int_{\rr}\ell(x_{1},\ldots,x_{n})\widehat{g}_{1,\epsilon}(\ell)(x_{1},
\ldots,x_{n})\,dx_{1}\\
&\quad\geq (1+\epsilon)^{-2}[E_{1}(\ell)(x_{2},\ldots,x_{n})]^{
\widetilde{q}_{2}}\widetilde{g}_{2,\epsilon}(\ell)(x_{2},
\ldots,x_{n}).
\end{align*}
Similarly to this estimation, by \eqref{2.18xx} and \eqref{2.17x}, we obtain
\begin{align*}
\int_{\rn}\ell(x)\widehat{g}_{1,\epsilon}(\ell)(x)\,dx&\geq (1+\epsilon)
^{-2n+2}\int_{\rr}
[E_{n-1}(\ell)(x_{n})]^{\widetilde{q}_{n}}\widetilde{g}_
{n,\epsilon}(\ell)(x_{n})\,dx_{n}\\
&=(1+\epsilon)^{-2n+2}[E_{n}(\ell)]^{1-\widetilde{p}_{n}}
\sum_{k_{n}\in\zz}2^{k_{n}\widetilde{p}_{n}\alpha_{n}}G_{k_{n},\epsilon}(\ell)\\
&\quad\times\int_{\rr}[E_{n-1}(\ell)(x_{n})]^{\widetilde{q}_{n}}H_{n,\epsilon}(\ell)
(x_{n})\one_{R_{k_{n}}}\,dx_{n}.
\end{align*}
From this, similarly to the estimation of \eqref{ifg-2},
we deduce that
\begin{equation*}
\int_{\rn}\ell(x)\widehat{g}_{1,\epsilon}(\ell)(x)\,dx\geq
(1+\epsilon)^{-2n}\|\ell\|_{\iihz(\rn)},
\end{equation*}
namely, \eqref{c22} and hence the above claim holds true.

For any $m\in\nn$ and
$f\in\iihz(\rn)$ with $\|f\|_{\iihz(\rn)}\neq 0$, let
$f_{m}:=f\one_{A_{m}}$. Then $f_{m}\rightarrow f$
almost everywhere and $|f_{m}|\uparrow |f|$ as $m\to\infty$. By this,
\eqref{c1}, \eqref{c22}, and Lemma \ref{donn}, we conclude
that
\begin{align}\label{kz}
&\sup\left\{\|fg\|_{L^{1}(\rn)}:\ \|g\|_{\diihz(\rn)}
=1\right\}\\\nonumber
&\quad\geq \sup_{\epsilon\in(0,\infty)}\sup_{m\in\nn}\int_{\rn}|
f(x)\widehat{g}_{1,\epsilon}(f_{m})(x)|\,dx\geq \sup_{
\epsilon\in(0,\infty)}\sup_{m\in\nn}\int_{\rn}|f_{m}(x)
\widehat{g}_{1,\epsilon}(f_{m})(x)|\,dx\\\nonumber
&\quad\geq \sup_{\epsilon\in(0,\infty)}(1+\epsilon)^{-2n}
\sup_{m\in\nn}\|f_{m}\|_{\iihz(\rn)}\\\nonumber
&\quad=\lim_{m\to\infty}\|f_{m}\|_{\iihz(\rn)}=
\|f\|_{\iihz(\rn)}\nonumber
\end{align}
and hence \eqref{gf} holds true.
This finishes the proof of Lemma \ref{assnd}.
\end{proof}
The following proposition is a generalization of the Fatou
lemma on $\iihz(\rn)$.
\begin{proposition}\label{E-Fatou}
Let $\vec{p}$,
$\vec{q}\in (0,\infty]^{n}$ and $
\vec{\alpha}\in\rn$. If $\{f_{k} \}_{k\in\nn}\subset
\MM(\rn)$ and $f_{k}\geq 0$ for any $k\in\nn$, then
\begin{equation}\label{fatou}
\left\|\varliminf_{k\to\infty}f_{k}\right\|_{\iihz(\rn)}
\leq \varliminf_{k\to\infty}\left\|f_{k}\right\|_{\iihz(\rn)}.
\end{equation}
\end{proposition}
\begin{proof}
Let $\vec{p}:=(p_{1},\ldots,p_{n}),$ $\vec{q}:=(q_{1},\ldots,q_{n})\in (0,\infty]^{n}$ and
$\vec{\alpha}:=(\alpha_{1},\ldots,\alpha_{n})\in\rn$.
We first consider the following two cases on $q_{1}$.
If $q_{1}\in(0,\infty)$, then,
from the Fatou lemma (see, for instance, \cite[p.\,52]
{folland}), we deduce that, for any given $k_{1}\in\zz$ and
for any $x_{2},\ldots,x_{n}\in\rr$,
\begin{equation}\label{2.32x}
\left\|\varliminf_{k\to\infty}f_{k}(\cdot,x_{2},\ldots,x_{n})
\one_{R_{k_{1}}}(\cdot)\right\|_{L^{q_{1}}(\rr)}\leq \varliminf_{k\to
\infty}\left\|f_{k}(\cdot,x_{2},\ldots,x_{n})\one_{R_{k_{1}}}(\cdot)\right
\|_{L^{q_{1}}(\rr)}.
\end{equation}
If $q_{1}=\infty$, then, for any given $k_{1}\in\zz$ and
for any $x_{2},\ldots,x_{n}\in\rr$,
$$\left\|\varliminf_{k\to\infty} f_{k}(\cdot,x_{2},\ldots,x_{n})
\one_{R_{k_{1}}}(\cdot)\right\|_{L^{\infty}(\rr)}\leq
\varliminf_{k\to\infty}\left\|f_{k}(\cdot,x_{2},\ldots,x_{n})
\one_{R_{k_{1}}}(\cdot)\right\|_{L^{\infty}(\rr)}.$$
Thus, for any given $q_{1}\in(0,\infty]$ and $k_{1}\in\zz$, and
for any
$x_{2},\ldots,x_{n}\in\rr$, \eqref{2.32x} holds true.
Now, for any given $q_{1}\in(0,\infty]$ and $\alpha_{1}\in\rr$, and for any
$x_{2},\ldots,x_{n}\in\rr$, we consider the following two cases on $p_{1}$.
If $p_{1}\in(0,\infty)$, then, from \eqref{2.32x}, the Fatou lemma, and
Definition \ref{chz}, we deduce that
\begin{align*}
\left\|\varliminf_{k\to\infty} f_{k}(\cdot,x_{2},\ldots,x_{n})
\right\|_{\ky(\rr)}&\leq
\left\{\sum_{k_{1}\in\zz}2^{k_{1}p_{1}\alpha_{1}}
\varliminf_{k\to\infty}\left\|f_{k}(\cdot,x_{2},\ldots,x_{n})
\one_{R_{k_{1}}}(\cdot)\right\|^{p_{1}}_{L^{q_{1}}(\rr)}\right\}
^{\f{1}{p_{1}}}\\
&\leq \varliminf_{k\to\infty} \|f_{k}(\cdot,x_{2},\ldots,
x_{n})\|_{\ky(\rr)}.
\end{align*}
If $p_{1}=\infty$, then we easily obtain
\begin{align*}
&\sup_{k_{1}\in\zz}\left\{2^{k_{1}\alpha_{1}}
\varliminf_{k\to\infty}\left\|f_{k}(\cdot,x_{2},\ldots,x_{n})
\one_{R_{k_{1}}}(\cdot)\right\|_{L^{\infty}(\rr)}\right\}\\
&\quad\leq
\varliminf_{k\to\infty}\sup_{k_{1}\in\zz}\left\{2^{k_{1}
\alpha_{1}}\left
\|f_{k}(\cdot,x_{2},\ldots,x_{n})\one_{R_{k_{1}}}(\cdot)
\right\|_{L^{\infty}(\rr)}\right\}.
\end{align*}
Using this, we find that, for any given $p_{1},$ $q_{1}\in(0,\infty]$ and
$\alpha_{1}\in\rr$, and
for any $x_{2},\ldots,x_{n}\in\rr$,
\begin{equation*}
\left\|\varliminf_{k\to\infty} f_{k}(\cdot,x_{2},\ldots,x_{n})\right\|_
{\ky(\rr)}\leq \varliminf_{k\to\infty} \|f_{k}(\cdot,x_{2},\ldots,
x_{n})\|_{\ky(\rr)}.
\end{equation*}
Similarly to this estimation, we conclude that \eqref{fatou}
holds
true. This finishes the proof of Proposition \ref{E-Fatou}.
\end{proof}
As an immediate consequence of both Propositions \ref{E-lat}
and \ref{E-Fatou}, we have
the following Fatou property; we omit the details.
\begin{proposition}\label{ballqfs-3}
Let $\vec{p}$, $\vec{q}\in(0,\infty]
^{n},\
\vec{\alpha}\in\rn$, $f\in\MM(\rn)$, and $\{f_{m}\}_{m\in\nn}
\subset \MM(\rn)$. If $0\leq f_{m}\uparrow f$ almost
everywhere as $m\to \infty$,
then $\|f_{m}\|_{\iihz(\rn)}\uparrow \|f\|_{\iihz(\rn)}$ as $m\to \infty$.
\end{proposition}
The following lemma is a generalization of \cite[p.\,189,\
Theorem 6.14]{folland}.
\begin{lemma}\label{fland}
Let $\vec{p}$, $\vec{q}\in[1,\infty
]^{n}$ and $\vec{\alpha}\in\rn$. Suppose that $f\in
\MM(\rn)$
is such that the quantity
\begin{equation*}
S(f):=\sup\left\{\left|\int_{\rn}f(x)g(x)\,dx\right|:
\ g\in U\
\text{and}\ \|g\|_{\diihz(\rn)}=1 \right\}
\end{equation*}
is finite, where $U$ denotes the set of all the simple
functions that vanish outside some $A_{m}$ with $m\in\nn$ and $A_{m}$
as in Lemma \ref{Am}. Then $f\in\iihz(\rn)$ and $S(f)=\|f\|_
{\iihz(\rn)}$.
\end{lemma}
\begin{proof}
Let all the symbols be as in the present lemma. We
consider the following two cases on $\|f\|_{\iihz(\rn)}$.
If $\|f\|_{\iihz(\rn)}=0$, then, by Proposition \ref{equ0},
we find that $f=0$ almost everywhere and hence $S(f)=\|f\|
_{\iihz(\rn)}=0$.

If $\|f\|_{\iihz(\rn)}\neq 0$,   then we claim that, for
any $m\in\nn$,
if $g_{m}$ is a bounded measurable function on $\rn$,
$\supp(g_{m})\subset A_{m}$, and $\|g_{m}\|_{\diihz(\rn)}=1$,
where $A_{m}$ is as in Lemma \ref{Am}, then
$$\left|\int_{\rn}f(x)g_{m}(x)\,dx\right|\leq S(f).$$
Indeed, fix an $m\in\nn$. From \cite[p.\,47,
Theorem 2.10]{folland}, we deduce that there exists a
sequence $\{g_{j,m}\}_{j\in\nn}$ of simple functions with
$\supp(g_{j,m})\subset A_{m}$ for any $j\in\nn$ such that
$|g_{j,m}|\leq |g_{m}|$ and $g_{j,m}\rightarrow g_{m}$ almost
everywhere on $\rn$ as $j\to \infty$. From $\one_{A_{m}}\in \diihz(\rn)$,
it follows that $	\mathrm{sgn}\,(f)\one_{A_{m}}/\|\one_
{A_{m}}\|_{\diihz(\rn)}\in U$. By this and $S(f)<\infty$,
we find that
\begin{equation*}
\left|\int_{\rn}f(x) [\mathrm{sgn}\,(f)(x)]\one_{A_{m}}(x)
\,dx\right|
\ls S(f)<\infty,
\end{equation*}
which implies that $|f|\one_{A_{m}}\in L^{1}(\rn)$. Applying
this and $|g_{j,m}|\leq \|g_{m}\|_{L^{\infty}(\rn)}\one_{A_{m}}$,
we obtain $|fg_{j,m}|\ls |f|\one_{A_{m}}
$, which, together with the dominated convergence theorem and
$$\|g_{j,m}\|_{\diihz(\rn)}\leq \|g_{m}\|_{\diihz(\rn)}=1,$$
further implies that
\begin{equation}\label{claim-1}
\left|\int_{\rn}f(x)g_{m}(x)\,dx\right| = \lim\limits_
{j\to \infty}\left|\int_{\rn}f(x)g_{j,m}(x)\,dx\right|\leq S(f).
\end{equation}
Thus, the above claim holds true.

Let $\{\phi_{k}\}_{k\in\nn}$
be a sequence of simple functions such that $\phi_{k}
\rightarrow f$
pointwise and $|\phi_{k}|\uparrow |f|$ as $k\to \infty$ (see \cite[p.\,47,
Theorem 2.10]{folland}). For any $k\in\nn$, let $f_{k}:=
\phi_{k}\one_{A_{k}}$. Then, as $k\to \infty$, $f_{k}\rightarrow f$ pointwise on
$\rn\setminus\{\mathbf{0} \}$ and
$|f_{k}|\uparrow |f|$, and $\supp(f_{k})\subset A_{k}$ for any $k\in\nn$. From this,
we deduce that, for any $k\in\nn$,
\begin{equation*}
\|f_{k}\|_{\iihz(\rn)}\ls \|\one_{A_{k}}\|_{\iihz(\rn)}<\infty,
\end{equation*}
which implies that $f_{k}\in\iihz(\rn)$ for any $k\in\nn$. Using Proposition
\ref{ballqfs-3} and $\|f\|_{\iihz(\rn)}\neq 0$, we conclude that
there exists an $N\in\nn$ such that, for any $k>N$,
$
\|f_{k}\|_{\iihz(\rn)}\neq 0.
$
Based on this, for any given $k>N$ and
$\epsilon\in(0,\infty)$,
and for any $x\in\rn$, let
\begin{equation*}
\widehat{g}_{1,\epsilon}^{(k)}(x):=\left[\mathrm{sgn}\,\overline{f(x)}\right]
[E_{0}(f_{k})(x)]^{\widetilde{q}_{1}-1}\widetilde{g}_{1,\epsilon}
(f_{k})(x),
\end{equation*}
where $\widetilde{q}_{1}$ is as in \eqref{pyw}, and $E_{0}(f_{k})
(x)$ and $\widetilde{g}_{1,\epsilon}(f_{k})(x)$ are, respectively, as in
\eqref{Eiga} and \eqref{2.40y} with $\ell$ there replaced by $f_{k}$.
By this, \eqref{2.41x}, and \eqref{c1}, we conclude that $\|\widehat{g}^{(k)}_{1,\epsilon}\|
_{\diihz(\rn)}=1$. Note that $\widehat{g}_{1,\epsilon}^{(k)}$ is a
bounded
measurable function and $\supp(\widehat{g}_{1,\epsilon}^{(k)})\subset
A_{k}$.
Indeed, since $f_{k}$ is bounded
and $\|\one_{\supp(\widehat{g}_{1,\epsilon}^{(k)})}\|_{\iihz(\rn)}
\leq
\|\one_{A_{k}}\|_{\iihz(\rn)}<\infty$, we deduce that $\widehat{g}_
{1,\epsilon}^{(k)}$ is bounded.
Using this, Proposition \ref{E-Fatou}, \eqref{c22} with $\ell$ and
$\widehat{g}_{1,\epsilon}$ replaced, respectively, by $f_{k}$ and
$\widehat{g}_{1,\epsilon}^{(k)}$, $|f_{k}|\leq |f|$, the definition
of $\widehat{g}_{1,\epsilon}^{(k)}$, and the above
claim, we conclude that
\begin{align*}
\|f\|_{\iihz(\rn)}&\leq \varliminf_{k\to\infty}\|f_{k}\|_
{\iihz(\rn)}\leq (1+\epsilon)^{2n}
\varliminf_{k\to\infty}\int_{\rn}\left|f_{k}(x)\widehat{g}^{(k)}_{1,
\epsilon}(x)\right|
\,dx\\
&\leq (1+\epsilon)^{2n} \varliminf_{k\to\infty}\int_{\rn}\left|f(x)
\widehat{g}
^{(k)}_{1,\epsilon}(x)\right|\,dx\\
&=(1+\epsilon)^{2n}\varliminf_{k\to\infty}\int_{\rn}f(x)
\widehat{g}^{(k)}
_{1,\epsilon}(x)\,dx\leq
(1+\epsilon)^{2n}S(f)<\infty.
\end{align*}
Letting $\epsilon \rightarrow 0^{+}$, we have $\|f\|_{\iihz(\rn)}
\leq
S(f)$, here and thereafter, $\epsilon \to 0^{+}$ means
that $\epsilon\in(0,\infty)$ and $\epsilon\to 0$.

On the other hand, by the H\"{o}lder inequality, we easily
obtain
$S(f)\leq \|f\|_{\iihz(\rn)}$. This finishes the proof of
Lemma \ref{fland}.
\end{proof}
Now, using Lemma \ref{fland}, we can prove Theorem \ref{dii}.
\begin{proof}[Proof of Theorem \ref{dii}]
Let all the symbols be as in the present theorem. Let $g\in
\diihz(\rn)$.
From Lemma \ref{mhr}, it follows that, for any $f
\in\iihz(\rn)$,
\begin{equation*}
|J_{g}(f)|=\left|\int_{\rn}f(x)g(x)\,dx\right|\leq \|f\|_{
\iihz(\rn)}
\|g\|_{\diihz(\rn)}.
\end{equation*}
Thus, $J_{g} \in[\iihz(\rn)]^{*}$.

Conversely, let $J$ be a continuous linear functional on
$\iihz(\rn)$. Observe that $\rn\setminus\{\mathbf{0}\}=\bigcup_{m\in\nn}A_{m}$,
where,
for any $m\in\nn$, $A_{m}$ is as in Lemma \ref{Am}. Now,
we claim
that, for any given $m\in\nn$,
the set function $\nu$ defined by setting, for any measurable set $E\subset A_{m}$,
\begin{equation*}
\nu(E):=J(\one_{E})
\end{equation*}
is a finite signed measure on $A_{m}$ (see, for instance, \cite[p.\,285]
{ra} for the precise definition) and absolutely
continuous with respect to the Lebesgue measure
on $\rn$ (see, for instance, \cite[p.\,288]{ra} for the precise definition).
Indeed, for any given $m\in\nn$, let
$\{E_{i}\}_{i\in\nn}
\subset A_{m}$ be a
countable family of disjoint sets. Since $J$ is a
continuous linear functional on $\iihz(\rn)$, we
deduce that, for any
$N\in\nn$,
\begin{equation*}
\left|\nu\left(\bigcup_{i=1}^{\infty}E_{i} \right)-
\nu\left(\bigcup_
{i=1}^{N}E_{i} \right)\right|=\left|\nu\left(
\bigcup_{i=N+1}^{\infty}
E_{i} \right) \right|\ls \left\|\one_{
\bigcup_{i=N+1}^{\infty}E_{i}}
\right\|_{\iihz(\rn)}.
\end{equation*}
Observe that, for any given $m\in\nn$ and for any
$N\in\nn$, $\bigcup_{
i=N+1}^{\infty}E_{i}\subset A_{m}$.
By this and Lemma \ref{donn}, we conclude
that
\begin{align*}
\left|\nu\left(\bigcup_{i=1}^{\infty}E_{i}
\right)-\sum_{i=1}^{\infty}
\nu\left(E_{i}\right)\right|&=\lim_{N\to
\infty}\left|\nu\left(\bigcup
_{i=1}^{\infty}E_{i} \right)- \nu\left(
\bigcup_{i=1}^{N}E_{i}
\right)\right|\\
&\ls\lim_{N\to \infty}\|\one_{
\bigcup_{i=N+1}^{\infty}
E_{i}}\|_{\iihz(\rn)}=0,
\end{align*}
which implies that $\nu$ is a signed measure.
Moreover, since $J$ is a continuous linear functional on
$\iihz(\rn)$, it follows that, for any
measurable set $E\subset A_{m}$,
$$|\nu(E)|
\ls \|\one_{E}\|_{\iihz
(\rn)}\ls\|\one_{A_{m}}\|_{\iihz(\rn)}
<\infty$$
and hence
$\nu$ is a finite signed measure. In
addition, for any given
$m\in\nn$, assume that $E\subset A_{m}$
with $|E|=0$. Then we
conclude that
\begin{equation*}
|\nu(E)|\ls \|\one_{E}\|_{\iihz(\rn)}=0
\end{equation*}
and hence $\nu$ is absolutely continuous
with respect to the Lebesgue measure. Thus, the above
claim holds true.
For any given $m\in\nn$, applying the
Radon--Nikodym theorem
(see, for instance, \cite[p.\,290, Theorem 4.3]{ra} for
the precise theorem),
we obtain a measurable function $h_{m}$
such that $\nu(E)=\int_{E}
h_{m}(x)\,dx$ for
any $E\subset A_{m}$. Then we obtain a
function $h\in\MM(\rn)$ which is defined to be $h_{m}$ almost everywhere on
$A_{m}$ for any $m\in\nn$. It is easy to show that $h$ is well defined on $\rn$
and, moreover, for any simple function $f$ with
$\supp(f)\subset A_{m}$,
\begin{equation}\label{2.33x}
J(f)=\int_{\rn} h(x)f(x)\,dx.
\end{equation}
By this and Lemma \ref{fland},
we conclude that
\begin{align*}
S(h)
:=&\,\sup\left\{\left|\int_{\rn}h(x)g(x)\,
dx\right|
:\ g\in U\ \text{and}\ \|g\|_{\iihz(\rn)}
=1\right\}\\
\leq&\, \sup\left\{\left|J(g)\right|
:\  \|g\|_{\iihz(\rn)}=1\right\}\\
=&\,\|J\|_{[\iihz(\rn)]^{*}}<\infty,
\end{align*}
which, combined with Lemma \ref{fland},
further implies that
\begin{equation}\label{dcc}
\|h\|_{\diihz(\rn)}=S(h)<\infty
\end{equation} and hence $h\in\diihz(\rn)$.
Now, let $f\in\iihz(\rn)$. Then, by
\cite[p.\,47,\
Theorem 2.10]
{folland}, we conclude that there exists
a sequence $\{g_{m}\}_{m\in\nn}$ of simple
functions such that
$g_{m}\rightarrow f$
almost everywhere and $|g_{m}|\uparrow|f|$ as $m\to \infty$.
For any $m\in\nn$,
let $f_{m}:=g_{m}\one_{A_{m}}$.
Observe that $f_{m}\rightarrow f$
as $m\rightarrow\infty$ and,
for any $m\in\nn$, $\supp(f_{m})\subset A_{m}$
and $|f_{m}-f|\leq 2|f|\in\iihz(\rn).$
From this and Lemma \ref{donn}, we deduce
that $\|f_{m}-f\|_{\iihz(\rn)}\rightarrow 0$ as $m\to\infty$,
which implies that, as $m\to\infty$,
\begin{equation}\label{J1}
|J(f_{m})-J(f)|\ls \|f_{m}-f\|_{\iihz(\rn)}
\rightarrow 0.
\end{equation}
Note that, from Lemma \ref{mhr}
and \eqref{dcc}, we
deduce that
$$\int_{\rn}|h(x)f(x)|\,dx\leq
\|h\|_{\diihz(\rn)}\|f\|_{\iihz(\rn)}
<\infty.$$
By this, \eqref{J1}, \eqref{2.33x}, $|hf_{m}|
\leq |hf|$ for any $m\in\nn$, the continuity
of $J$, and the dominated convergence theorem,
we conclude that
\begin{equation*}
J(f)=\lim_{m\to\infty}J(f_{m})=\lim_{m\to\infty}
\int_{\rn}h(x)f_{m}(
x)\,dx=\int_{\rn} h(x)f(x)\,dx.
\end{equation*}
The uniqueness of $h$ is obvious. This finishes
the proof of Theorem \ref{dii}.
\end{proof}
For any $\vec{p}$, $\vec{q}\in(0,\infty]^{n}$ and
$\vec{\alpha}\in\rn$,
the \emph{associate space (K\"othe dual)}
$[\iihz(\rn)]'$ of $\iihz(\rn)$
is defined by setting
\begin{align*}
[\iihz(\rn)]':=\left\{f\in \MM(\rn):\ \|f\|_
{[\iihz(\rn)]'}<\infty \right\},
\end{align*}
where, for any $f\in\MM(\rn)$, $$\|f\|_{[\iihz(\rn)]'}:=
\sup\left\{\|fg\|_{L^{1}(\rn)}:\ g\in
\iihz(\rn),\ \|g\|_{\iihz(\rn)}=1\right\}.$$

As a consequence of Lemma \ref{assnd}, we obtain the
associate spaces of $\iihz(\rn)$ as follows.
\begin{corollary}\label{Eass}
Let $\vec{p},$ $\vec{q}\in[1,\infty]^{n}$ and
$\vec{\alpha}\in\rn$. Then
$f\in[\iihz(\rn)]'$ if and only if
$f\in\diihz(\rn)$. Moreover,
\begin{equation*}
\|f\|_{[\iihz(\rn)]'}=\|f\|_{\diihz(\rn)}.
\end{equation*}
\end{corollary}
\begin{proof}
Let all the symbols be as in the present
corollary. For any given $m\in\nn$ and $f\in\MM(\rn)$,
and for any $x\in\rn$, let
$f_{m}(x):=f(x)$ if
$|f(x)|\leq m$, and $f_{m}(x):=m$ if $|f(x)|>m$.
Without loss of
generality, we may assume that $f$ is nonnegative.
Then, for any given $m\in\nn$ and for any $x\in\rn$, let
$$g_{m}(x):=f_{m}(x)\one_{A_{m}}(x),$$
where $A_{m}$ is
as in Lemma \ref{Am}.
By this, we easily obtain $g_{m}\in\iihz(\rn)$
and $g_{m}\uparrow f$ as $m\to\infty$.
This, together with Proposition \ref{ballqfs-3} and
Lemma \ref{assnd}, further implies that
\begin{align*}
\|f\|_{\iihz(\rn)}&=\lim_{m\to\infty}\|g_{m}\|
_{\iihz(\rn)}
=\lim_{m\to\infty}\sup\left\{\|g_{m}h\|_{L^{1}
(\rn)}:\ \|h\|_
{\diihz(\rn)}=1 \right\}\\
&\leq \sup\left\{\|fh\|_{L^{1}(\rn)}:\ \|h\|_
{\diihz(\rn)}=1 \right\}.
\end{align*}
On the other hand, from Lemma \ref{mhr},
we deduce that
$$\|f\|_{\iihz(\rn)}\geq \sup\left\{\|fh\|_{L^{1}(\rn)}:
\ \|h\|_{\diihz(\rn)}
=1\right\}.$$
This finishes the proof of Corollary \ref{Eass}.
\end{proof}

\subsection{Ball Quasi-Banach Function Spaces}\label{s2.2}

The following concept of ball quasi-Banach function
spaces on $\rn$
is from \cite{SHYY}.
\begin{definition}\label{ballqB}
Let $X\subset\MM(\rn)$ be a quasi-normed linear space,
equipped with a quasi-norm $\|\cdot\|_X$ which makes sense for all
measurable functions on $\rn$.
Then $X$ is called a \emph{ball quasi-Banach
function space} on $\rn$ if it satisfies:
\begin{enumerate}
\item[$\mathrm{(i)}$] for any $f\in\MM(\rn)$, $\|f\|_{X}=0$ implies
that $f=0$ almost everywhere;
\item[$\mathrm{(ii)}$] for any $f$, $g\in\MM(\rn)$, $|g|\le |f|$ almost
everywhere implies that $\|g\|_X\le\|f\|_X$;
\item[$\mathrm{(iii)}$] for any $\{f_m\}_{m\in\nn}\subset\MM(\rn)$ and $f\in\MM(\rn)$,
$0\le f_m\uparrow f$ almost everywhere as $m\to\infty$
implies that $\|f_m\|_X\uparrow\|f\|_X$ as $m\to\infty$;
\item[$\mathrm{(iv)}$] $B\in\BB$ implies
that $\one_B\in X$,
where $\BB$ is as in \eqref{Eqball}.
\end{enumerate}

Moreover, a ball quasi-Banach function space $X$
is called a
\emph{ball Banach function space} if it satisfies:
\begin{enumerate}
\item[$\mathrm{(v)}$] for any $f,$ $g\in X$,
\begin{equation*}
\|f+g\|_X\le \|f\|_X+\|g\|_X;
\end{equation*}
\item[$\mathrm{(vi)}$] for any $B\in \BB$,
there exists a positive constant $C_{(B)}$,
depending on $B$, such that, for any $f\in X$,
\begin{equation*}
\int_B|f(x)|\,dx\le C_{(B)}\|f\|_X.
\end{equation*}
\end{enumerate}
\end{definition}
\begin{remark}
\begin{enumerate}
\item[$\mathrm{(i)}$] Let $X$ be a ball quasi-Banach
function space on $\rn$. By \cite[Remark 2.6(i)]{aaa} (see also \cite{yhyy21b}),
we conclude that, for any $f\in\MM(\rn)$, $\|f\|_{X}=0$ if and only if $f=0$
almost everywhere on $\rn$.
\item[$\mathrm{(ii)}$] As was mentioned in
\cite[Remark 2.6(ii)]{aaa}, we obtain an
equivalent formulation of Definition \ref{ballqB}
via replacing any ball $B$ there by any
bounded measurable set $E$.
\item[$\mathrm{(iii)}$] We should point out that,
in Definition \ref{ballqB}, if we
replace any ball $B$ by any measurable set $E$ with
finite measure, we obtain the
definition of (quasi-)Banach function spaces, which were originally
introduced in \cite[Definitions 1.1 and 1.3]{IO}. Thus,
a (quasi-)Banach function space
is also a ball (quasi-)Banach function space.
\item[$\mathrm{(iv)}$] By \cite[Theorem 2]{Orlicz},
we conclude that both (ii) and (iii) of
Definition \ref{ballqB} imply that any ball quasi-Banach
function space is complete.
\end{enumerate}
\end{remark}
We have the following conclusions on mixed-norm Herz spaces.
\begin{proposition}\label{bqb-2}
Let $\vec{p}$, $\vec{q}:=(q_1,\ldots,q_n)\in(0,\infty]^{n}$
and $\vec{\alpha}:=(\alpha_1,\ldots,\alpha_n)\in\rn$. Then
\begin{enumerate}
\item[$\mathrm{(i)}$] the mixed-norm Herz space $\iihz(\rn)$ is
a ball quasi-Banach function space if and only if,\
for any $i\in\{1,\ldots,n\}$,
$\alpha_{i}\in(-\f{1}{q_{i}},\infty)$;
\item[$\mathrm{(ii)}$] the mixed-norm Herz
space $\iihz(\rn)$ is
a ball Banach function space if $\vec{p}$, $\vec{q}\in[1,\infty]^{n}$ and, for any $i\in\{1,\ldots
,n\}$, $\alpha_{i}\in(-\f{1}{q_{i}},1-\f{1}{q_{i}})$.
\end{enumerate}
\end{proposition}	
\begin{proof}
Let $\vec{p}:=(p_{1},\ldots,p_{n})$, $\vec{q}:=(q_1,\ldots,q_n)\in(0,\infty]^{n}$ and
$\vec{\alpha}:=(\alpha_1,\ldots,\alpha_n)\in\rn$. We first prove the sufficiency
of (i). By Propositions \ref{equ0}, \ref{E-lat},
and \ref{ballqfs-3}, we find that the space $\iihz(\rn)$
satisfies (i), (ii), and (iii) of Definition \ref{ballqB}.
For any $i\in\{1,\ldots,n \}$, let $\alpha_{i}
\in(-\f{1}{q_{i}},\infty)$ .
Observe that, for any $B\in \BB$, there exist an $r\in\zz$
and a cube $Q(\textbf{0},2^{r+1})=I_{1}\times\cdots\times I_{n}$
such that $B\subset Q(\textbf{0},2^{r+1})$, where
$I_{i}:=(-2^{r},2^{r})$ for any $i\in\{1,\ldots,n \}$.
Note that $\one_{Q(\textbf{0},2^{r+1})}(x)=\one_{I_{1}}(x_{1})
\times\cdots\times \one_{I_{n}}(x_{n})$.
From this and Proposition \ref{E-lat}, we deduce that
\begin{align}\label{2.30x}
\|\one_{B}\|_{\iihz(\rn)}\leq \|\one_{Q(\textbf{0},2^{r+1})}
\|_{\iihz(\rn)}=\prod_{i=1}^{n}\|\one_{I_{i}}\|_{\ki(\rr)}.
\end{align}
For any given $i\in\{1,\ldots,n \}$, we estimate $\|\one_{I_{i}}\|_
{\ki(\rr)}$. To achieve it, we consider the following two
cases on $p_{i}$. If $p_{i}\in(0,\infty)$, then, by
$\alpha_{i}\in (-\f{1}{q_{i}},\infty)$, we conclude that
\begin{equation}\label{2.30y}
\|\one_{I_{i}}\|_{\ki(\rr)}= \left[\sum_{k_{i} \in \zz}2^
{k_{i}p_{i}\alpha_{i}}\left\|\one_{I_{i}}\one_{R_{k_{i}}}
\right\|^{p_{i}}_{L^{q_{i}}(\rr)}\right]^{\f{1}{p_{i}}}
\sim\left[ \sum_{k_{i}=-\infty}^{r}2^{k_{i}p_{i}(\alpha_{i}+
\f{1}{q_{i}
})}\right]^{\f{1}{p_{i}}} < \infty.
\end{equation}
If $p_{i}=\infty$, then we find that
\begin{align}\label{2.30z}
\|\one_{I_{i}}\|_{\ki(\rr)} = \sup_{k_{i}\in\zz}\left[2^{k_{i}
\alpha_{i}}\|\one_{I_{i}}\one_{R_{k_{i}}}\|_{L^{q_{i}}(\rn)}
\right]\sim\sup_{k_{i}\in(-\infty,r]\cap \zz}
2^{k_{i}(\alpha_{i}+\f{1}{q_{i}})} < \infty.
\end{align}
Thus, combining \eqref{2.30x} with \eqref{2.30y} and \eqref{2.30z}, we
obtain $\|\one_{B}\|_{\iihz(\rn)}<\infty$ and hence the
space $\iihz(\rn)$ satisfies Definition \ref{ballqB}(iv),
which completes the proof of sufficiency of (i).

Now, we prove the necessity of (i).
Let $\alpha_{m}\in(-\infty,-\f{1}{q_{m}}]$ with some $m\in
\{1,\ldots,n\}$. Note that there exists a
$j\in\zz\cap(-\infty,0]$
such that $B(\mathbf{0},1)\supset Q(\mathbf{0},2^{j+1})$.
Observe that
$Q(\mathbf{0},2^{j+1})=J_{1}\times\cdots\times J_{n},$
where $J_{i}:=(-2^{j},2^{j})$ for any $i\in\{1,\ldots,n\}$.
Similarly to the estimation of $\|\one_{I_{i}}\|_{\ki(\rr)}$,
using both Proposition \ref{E-lat} and Definition \ref{mhz}, we
conclude that
\begin{equation*}
\|\one_{B(\mathbf{0},1)}\|_{\iihz(\rn)}\geq \|\one_
{Q(\mathbf{0},2^{j+1})}\|_{\iihz(\rn)}=\prod_{i=1}^{n}\|
\one_{J_{i}}\|_{\ki(\rr)}=
\infty,
\end{equation*}
which implies that $\one_{B(\mathbf{0},1)}\notin \iihz(\rn)$
and hence $\iihz(\rn)$ does not satisfy Definition \ref{ballqB}(iv).
Thus, in this case, $\iihz(\rn)$ is not a ball quasi-Banach
function space. This finishes the proof of (i).

Next, we show (ii). Let $\vec{p},$ $\vec{q}\in [1,\infty]^{n}$
and $\vec{\alpha}\in\rn$ with $\alpha_{i}\in(-\f{1}{q_{i}},
1-\f{1}{q_{i}})$ for any $i\in\{1,\ldots, n \}$. By (i), we know that
$\iihz(\rn)$ is a ball
quasi-Banach function space. Moreover, it is easy to prove that
$\|\cdot\|_{\iihz(\rn)}$ satisfies the norm triangle inequality.
Now, we show that the space $\iihz(\rn)$ satisfies Definition
\ref{ballqB}(vi). Indeed, for any given $B\in\BB$ and for any
$f\in\iihz(\rn)$,
\begin{align}\label{2.18x}
\int_{B}|f(x)|\,dx&\leq \int_{Q(\mathbf{0},2^{r+1})}|f(x)|\,dx
=\int_{I_{n}}\cdots \int_{I_{1 }}|f(x_{1},\ldots,x_{n})|\,
dx_{1} \cdots \,dx_{n}\\
&=\sum_{k_{n}=-\infty}^{r}\int_{R_{kn}}
\cdots \left[\sum_{k_{1}=-\infty}^{r}\int_{R_{k_{1} }} |
f(x_{1},\ldots,x_{n})|\, dx_{1}\right]\cdots \,dx_{n}.\nonumber
\end{align}
Using the H\"{o}lder inequality, Definition \ref{chz}, and
$\alpha_{1}<1-\f{1}{q_{1}}$,
we find that, for any $x_{2},\ldots,x_{n}\in\rr$,
\begin{align*}
&\sum_{k_{1}=-\infty}^{r}\int_{R_{k_{1} }} |f(x_{1},\ldots,x_{n})|
\, dx_{1}\\
&\quad\ls \sum_{k_{1}=-\infty}^{r}2^{k_{1}\alpha
_{1}}\|f(\cdot,x_{2},\ldots,x_{n})\one_{R_{k_{1}}}(\cdot)\|
_{L^{q_{1}}(\rr)}\cdot
2^{k_{1}[(1-\f{1}{q_{1}})-\alpha_{1} ]}\\
&\quad\ls \|f(\cdot,x_{2},\ldots,x_{n})\|_{\ky(\rr)} \left\{\sum_{k_{1}
=-\infty}^{r}2^{k_{1}p_{1}'[(1-\f{1}{q_{1}})-\alpha_{1} ]}
\right\}^{\f{1}{p_{1}'}}\\
&\quad\sim \|f(\cdot,x_{2},\ldots,x_{n})\|_{\ky(\rr)}.
\end{align*}
Similarly to this estimation, from Definition \ref{mhz} and
\eqref{2.18x}, we deduce that
\begin{equation*}
\int_{B}|f(x)|\,dx\ls \|f\|_{\iihz(\rn)},
\end{equation*}
where the implicit constant depends on $B$ and is
independent of $f$. Thus, $\iihz(\rn)$ satisfies Definition
\ref{ballqB}(vi) and hence it is a ball Banach function space.
This finishes the proof of Proposition \ref{bqb-2}.
\end{proof}
\begin{remark}
We should point out that, from \cite[p.\,26,\ Example 1.2.25]{lyh},
it follows that the conditions of Proposition \ref{bqb-2}(ii) are sharp in some sense.
\end{remark}
In what follows, the classical \emph{centered Hardy--Littlewood
maximal operator M}
is defined by setting, for any $f \in \MM(\rn)$ and $x\in \rn$,
\begin{equation}\label{Max}
M(f)(x):=\sup_{r\in (0,\infty)}\frac{1}{|B(x,r)|}\int_{B(x,r)}
|f(y)|\,dy.
\end{equation}

The following maximal operators can be found, for instance,
in \cite{mxM}.
\begin{definition}\label{itmax}
Let $k\in [1,n]\cap\zz$. The \emph{maximal operator} $M_{k}$
for the $k$-th variable is defined by
setting, for any given $f\in \MM(\rn)$ and for any $x:=(x_{1},\ldots,x_{n})\in\rn$,
\begin{equation}\label{Mk}
M_{k}(f)(x):=\sup_{B\subset \rr,B\ni x_{k}}\f{1}{|B|}\int_{B}|
f(x_{1},\ldots,x_{k-1},y_{k},x_{k+1},\ldots,x_{n})|\,dy_{k},
\end{equation}
where $B$ denotes any ball in $\rr$ containing $x_{k}$.
Furthermore, for any given $t\in(0,\infty)$, the iterated maximal operator $\cm_{t}$ is
defined by setting, for any given $f\in\MM(\rn)$ and for any $x\in \rn$,
\begin{equation*}
\cm_{t}(f)(x):=\left[M_{n}\cdots M_{1}\left(|f|^{t}\right)(x)\right]^{\f{1}{t}}.
\end{equation*}
\end{definition}
\begin{remark}\label{max-cp}
In Definition \ref{itmax}, when $t:=1$, the iterated maximal
operator $\cm_{1}$
is larger than $M$, where $M$ is as
in \eqref{Max}.
Indeed, by \cite[(2.5)]{duo}
and the Tonelli theorem, we conclude that, for any given $f\in\MM(\rn)$
and for any $x\in\rn$,
\begin{align*}
M(f)(x) &\ls \sup_{r\in(0,\infty)}\f{1}{(2r)^{n}}
\int_{[-r,r]^{n}}|f(x-y)|\,dy\\
&\sim\sup_{r\in(0,\infty)}\f{1}{(2r)^{n}}\int_{[-r,r]}\cdots
\int_{[-r,r]}|f(x_{1}-y_{1},\ldots,x_{n}-y_{n})|\,dy_{1}
\cdots dy_{n}\\
&\ls\cm_{1}(f)(x),
\end{align*}
where the implicit constants only depend on $n$.
\end{remark}
For any given ball Banach function space $X$,
its \emph{associate space (K\"othe dual)} $X'$
is defined by setting
\begin{equation}\label{ass-X}
X':=\left\{f\in \MM(\rn):\ \|f\|_{X'}<\infty\right\},
\end{equation}
where, for any $f\in \MM(\rn)$,
$$\|f\|_{X'}:=\sup\left\{\|fg\|_{L^{1}(\rn)}:
\ g\in X,\ \|g\|_{X}=1\right\}$$
(see \cite[Chapter 1, Section 2]{IO} for details).
Let us recall the concept of the convexification of ball
quasi-Banach function spaces (see, for instance,
\cite[p.\,10,\ Definition 2.6]{SHYY}).
\begin{definition}\label{cvx}
Let $X$ be a ball quasi-Banach function space and $p\in (0,\infty)$.
Then the \emph{p-convexification} $X^{p}$ of $X$ is defined by
setting $$X^{p}:=\{f\in \mathscr M(\rn): \ |f|^{p} \in X\}$$
equipped
with the quasi-norm $\|f\|_{X^{p}}:=\|\,|f|^{p}\|_{X}^{1/p}$ for any $f\in X^{p}$.
\end{definition}
Next,
we recall the definition of
$A_{p}(\rn)$-weights (see, for instance, \cite{HA}).

\begin{definition}
Let $p\in(1,\infty)$ and $\oz$ be a weight.
Then $\oz$ is said to be an $A_{p}(\rn)$-\emph{weight} if
\begin{equation}
[\oz]_{A_{p}(\rn)}:=\sup_{Q\subset \rn}
\left[\f{1}{|Q|}\int_{Q}\oz(x)\,dx \right]
\left\{\f{1}{|Q|}\int_{Q}[\oz(x)]^{-\f{1}{p-1}}\,dx \right\}^{p-1}<\infty,
\end{equation}
and a weight $\oz$ is said to be an $A_{1}(\rn)$-\emph{weight} if
\begin{equation}
[\oz]_{A_{1}(\rn)}:=\sup_{Q\subset \rn}
\left[\f{1}{|Q|}\int_{Q}\oz(x)\,dx \right]
\left\|\oz^{-1} \right\|_{L^{\infty}(Q)}<\infty,
\end{equation}
where $Q$ is any cube of $\rn$.
\end{definition}
The following lemma is the Fefferman--Stein vector-valued maximal
inequality on $M_{n}$, whose proof borrows some ideas from the proof
of the extrapolation theorem (see, \cite[pp.\,18-19]{Cruz}).
\begin{lemma}\label{ax-Fs}
Let $X$ be a ball quasi-Banach function space and $r$, $p\in (1,\infty)$.
Assume that $X^{1/p}$ is a ball Banach function space and $M_{n}$
is bounded on $(X^{1/p})'$. Then there exists a positive
constant $C$ such that, for any $\{f_{i} \}_{i\in\zz}\subset
\MM(\rn)$,
\begin{equation}\label{2.36x}
\left\|\left\{\sum_{i\in\zz}\left[M_{n}(f_{i}) \right]^{r}
\right\}^{\f{1}{r}} \right\|_{X}\leq
C\left\|\left(\sum_{i\in\zz}|f_{i} |^{r} \right)^{\f{1}{r}}
\right\|_{X},
\end{equation}
where $M_{n}$ is as in \eqref{Mk}.
\end{lemma}
\begin{proof}
Let all the symbols be as in the present lemma.
Since $M_{n}$ is bounded on $(X^{1/p})'$, we can define $\RR$ by setting,
for any given non-negative measurable function $h$ and for any $x\in\rn$,
\begin{equation*}
\RR h(x):=\sum_{k=0}^{\infty}\f{M_{n}^{k}h(x)}{2^{k}\|M_{n}\|^{k}_
{(X^{1/p})'\rightarrow (X^{1/p})'}},
\end{equation*}
where, for any $k\in\nn$, $M^{k}_{n}:=M_{n}\circ\cdots\circ M_{n}$ is the $k$-times
iteration of the maximal operator $M_{n}$, and $M_{n}^{0}h:=h$.
Moreover,
$$\|M_{n}\|_{(X^{1/p})'\rightarrow (X^{1/p})'}:=\sup\left\{\|M_{n}h\|_
{(X^{1/p})'}:\ \|h\|_{(X^{1/p})'}\leq 1 \right\}.$$
Now, we claim that the operator $\RR$ has the following properties:
for any given non-negative measurable function $h$,
\begin{enumerate}
\item[$\mathrm{(i)}$] for any $x\in\rn$, $h(x)\leq \RR h(x)$ ;
\item[$\mathrm{(ii)}$] $\|\RR h\|_{(X^{1/p})'} \leq 2
\|h\|_{(X^{1/p})'}$;
\item[$\mathrm{(iii)}$] for any given $x_{1},\ldots,x_{n-1}\in\rr$,
$\RR h(x_{1},\ldots,x_{n-1},\cdot)$ is an
$A_{1}(\rr)$-weight and
$$[\RR h(x_{1},\ldots,x_{n-1})]_
{A_{1}(\rr)}\leq 2\|M_{n}\|_{(X^{1/p})'\rightarrow (X^{1/p})'}.$$
\end{enumerate}
Indeed, both (i) and (ii) follow immediately from the definition of $\RR h$.
To show (iii), note that $M_{n}$ is
sublinear. From this, we further infer that, for any given non-negative
measurable function
$h$ and for any $x\in\rn$,
\begin{equation*}
M_{n}(\RR h)(x) \leq \sum_{k=0}^{\infty}\f{M_{n}^{k+1}h(x)}{2^{k}
\|M_{n}\|^{k}_{(X^{1/p})'\rightarrow (X^{1/p})'}}\leq 2\|M_{n}\|_
{(X^{1/p})'\rightarrow (X^{1/p})'} \RR h(x),
\end{equation*}
which, combined with \cite[p.\,134, (7.5)]{duo}, further implies that,
for any given $x_{1},\ldots,x_{n-1}\in\rr$, $\RR h(x_{1},\ldots,
x_{n-1},\cdot)$ is an $A_{1}(\rr)$-weight and
$$[\RR h(x_{1},\ldots,
x_{n-1})]_{A_{1}(\rr)}\leq 2\|M_{n}\|_{(X^{1/p})'\rightarrow (X^{1/p})'}.$$
Thus, (iii) and hence the above whole claim hold true.

For any $\{f_{i}\}_{i\in\zz}\subset \MM(\rn)$, let
\begin{equation*}
f:=\left\{\sum_{i\in\zz}[M_{n}(f_{i})]^{r} \right\}^{\f{1}{r}}\quad \mathrm{and}\quad
g:=\left(\sum_{i\in\zz}|f_{i}|^{r} \right)^{\f{1}{r}}.
\end{equation*}
From \eqref{ass-X} and Definition \ref{cvx}, we deduce that
\begin{equation}\label{S1}
\|f\|_{X}^{p}=\||f|^{p}\|_{X^{1/p}}= \sup\left\{\left\||f|^{p}h\right\|
_{L^{1}(\rn)}:\
h\in (X^{1/p})',\ \|h\|_{(X^{1/p})'}=1 \right\}.
\end{equation}
By this, we find that, to prove \eqref{2.36x}, it suffices
to show that, for any non-negative measurable
function $h$ satisfying $\|h\|_{(X^{1/p})'}=1$,
$$\left\{\int_{\rn}[f(x)]^{p}h(x)\,dx\right\}^{\f{1}{p}} \ls \|g\|_{X}.$$
Indeed, using the above claim, \cite[p.\,22,\ Theorem 3.1]{AJ},
and \cite[p.\,2013,\ Lemma 2.5]{wk-1}, we have
\begin{align*}
\int_{\rn}[f(x)]^{p}h(x)\,dx&\leq \int_{\rn}[f(x)]^{p}\RR h(x)\,dx
\ls\int_{\rn}[g(x)]^{p}\RR h(x)\,dx\\
&\ls \left\||g|^{p}\right\|_{X^{1/p}}\|\RR h\|_{(X^{1/p})'}
\ls \|g\|_{X}^{p}\|h\|_{(X^{1/p})'}
\sim \|g\|_{X}^{p},
\end{align*}
which, combined with \eqref{S1},
further implies that $\|f\|_{X}\ls \|g\|_{X}$.
This finishes the proof of Lemma \ref{ax-Fs}.
\end{proof}

\section{Riesz--Thorin Interpolation Theorem on Mixed-Norm Spaces\label{s3}}

Recall that, for any given $r\in (0,\infty]$ and for any sequence
$\{a_j\}_{j\in\zz}\subset\cc$,
$$\left\|\{a_j\}_{j\in\zz}\right\|_{l^r}:
=\left[\sum_{j\in\zz}|a_j|^r\right]^{1/r}$$
with the usual modification made when $r=\infty$.

In this section, we introduce a special
mixed norm space which is called the $l^{r}$-Herz type mixed-norm space and
proves to be crucial in the proof of the boundedness of the maximal operator
on $\iihz(\rn)$. Moreover, we establish the Riesz--Thorin
interpolation theorem associated with the linear operators
from the mixed-norm Herz space to the $l^{\infty}$-Herz type mixed-norm space.
\begin{definition}\label{EL}
Let $r\in (0,\infty]$, $\vec{p}$,
$\vec{q}\in (0,\infty]^{n},$ and $\vec{\alpha}\in\rn$.
The $l^{r}$\emph{-Herz type mixed-norm space}
$(\iihz(\rn),l^{r})$ is
defined to be the set of all the sequences of functions,
$f:=\{f(j,\cdot)\}_{j\in\zz}\subset \MM(\rn)$, such that
\begin{equation*}
\|f\|_{(\iihz(\rn),l^{r})}:=\left\|\left[\sum_{j\in\zz}
|f(j,\cdot)|^{r} \right]^{\f{1}{r}} \right\|_{\iihz(\rn)}<\infty
\end{equation*}
with the usual modifications made when $r=\infty$.
\end{definition}
Now, we have the following variant of Lemma \ref{assnd} on $(\iihz(\rn),l^{r})$.
\begin{lemma}\label{lherz-and}
Let $r\in [1,\infty]$, $\vec{p}$, $\vec{q}\in
[1,\infty]^{n},$ and $\vec{\alpha}\in\rn$. Then,
for any $f\in
(\iihz(\rn),l^{r})$,
\begin{align}\label{lherz-as}
&\|f\|_{(\iihz(\rn),l^{r})}\\\nonumber
&\quad=\sup\bigg\{
\bigg\|\sum_{j\in\zz}|f(j,\cdot)g(j,\cdot)| \bigg\|_{L^{1}
(\rn)}:\  g\in (\diihz(\rn),l^{r'}),\  \|g\|_
{(\diihz(\rn),l^{r'})}=1\bigg\}\\\nonumber
&\quad=:I(f).\nonumber
\end{align}
\end{lemma}
\begin{proof}
Let all the symbols be as in the present lemma. We
consider the following two cases on $\|f\|_{(\iihz(\rn),l^{r})}$
with $f\in (\iihz(\rn),l^{r})$. If $\|f\|_{(\iihz(\rn),l^{r})}=0$,
then, from Proposition \ref{equ0}, it follows that,
for any $j\in\zz$,
$f(j,\cdot)=0$ almost everywhere on $\rn$, which
implies that
\eqref{lherz-as} holds true in this case.

If $\|f\|_{(\iihz(\rn),l^{r})}\neq 0$, then, by the H\"{o}lder
inequality and Lemma \ref{mhr}, we easily find that,
for any $g\in (\diihz(\rn),l^{r'})$ satisfying
$\|g\|_{(\diihz(\rn),l^{r'})}=1$,
\begin{equation*}
\left\|\sum_{j\in\zz}|f(j,\cdot)g(j,\cdot)| \right\|_
{L^{1}(\rn)}\leq \|f\|_{(\iihz(\rn),l^{r})}\|g\|_
{(\diihz(\rn),l^{r'})}=\|f\|_{(\iihz(\rn),l^{r})},
\end{equation*}
which further implies that
\begin{equation*}
I(f)\leq \|f\|_{(\iihz(\rn),l^{r})}.
\end{equation*}
Conversely, we aim to show that
\begin{equation}\label{lherz-gf}
I(f)\geq \|f\|_{(\iihz(\rn),l^{r})}.
\end{equation}
To achieve it, we first consider that $\ell:=\{\ell(j,\cdot) \}_{j\in\zz}\in (\iihz(\rn),l^{r})$ is such that
$\ell_{r}:=\|\{\ell(j,\cdot)\}_{j\in\zz}\|_{l^{r}}$
vanishes outside from $A_{k_{0}}$ for some $k_{0}\in\nn$
and $\ell(j,\cdot):=0$ if $j\in \zz\cap [B(0,k_{0})]^{\com}$.
For any $j\in \zz$ and $x\in\rn$, let $[\ell(j,x)]^{0}:=0$ if $\ell(j,x)=0$.
For any given $\epsilon\in(0,\infty)$ and for any $x\in\rn$,  let
\begin{equation}\label{Ld}
\Lambda_{\epsilon}(x):=\left\{j\in\zz:\ |\ell(j,x)|>
(1+\epsilon)^{-1}\|\{ \ell(j,x)\}_{j\in\zz}\|_{l^{\infty}} \right\}.
\end{equation}
For any given $r\in[1,\infty]$ and $\epsilon\in(0,\infty)$,
and for any $j\in\zz$ and $x\in\rn$, let
\begin{equation}\label{3.2x}
g_{0,\epsilon}(\ell_{r})(j,x):=
\begin{cases}
\displaystyle\f{\one_{\Lambda_{\epsilon}(x)}(j)}{\#[\Lambda_
{\epsilon}(x)]} & \text{if}\ r=\infty\ \text{and}\
\Lambda_{\epsilon}(x)\neq \emptyset,\\
0 & \text{if}\ r=\infty\ \text{and}\
\Lambda_{\epsilon}(x)=\emptyset,\\
1 & \text{if}\ r\in[1,\infty)
\end{cases}
\end{equation}
and
\begin{align}\label{SSE}
\widetilde{g}_{0,\epsilon}(\ell_{r})(j,x):=&\,[E_{n}(\ell_{r})]
^{1-\widetilde{p}_{n}}\prod_{i=0}^{n-1}[E_{i}(\ell_{r})
(x_{i+1},\ldots,x_{n})]^{\widetilde{q}_{i+1}-
\widetilde{p}_{i}}\\
&\,\times \prod_{i=1}^{n}[g_{i,\epsilon}(\ell_{r})
(x_{i},\ldots,x_{n})H_{i,\epsilon}(\ell_{r})(x_{i},
\ldots,x_{n})]g_{0,\epsilon}(\ell_{r})(j,x),\nonumber
\end{align}
where $\widetilde{p}_{0}:=r$ if $r\in[1,\infty)$,
$\widetilde{p}_{0}:=1$ if $r=\infty$,
$\widetilde{p}_{i}$ and $\widetilde{q}_{i}$ for any $i\in\{1,\ldots,n\}$
are as in \eqref{pyw}, $H_{i,\epsilon}(\ell_{r})$
for any $i\in\{1,\ldots,n \}$ is as
in \eqref{Hje} with $\ell$ there replaced by $\ell_{r}$,
$[E_{i}(\ell_{r})]^{\gamma}$ for any $i\in\{ 0,\ldots,n\}$ and $\gamma\in\rr$ is
as in \eqref{Eiga} with $\ell$ replaced by $\ell_{r}$,
and $g_{i,\epsilon}(\ell_{r})$ for any $i\in\{1,\ldots,n \}$
is as in \eqref{2.40x} with $\ell$ replaced by $\ell_{r}$.
For any given $r\in[1,\infty]$ and $i\in\{1,\ldots,n\}$,
and for any $\epsilon\in(0,\infty)$,
let $\widetilde{g}_{i,\epsilon}(\ell_{r})$ be as
\eqref{2.40y} with $\ell$ replaced
by $\ell_{r}$.
Then, for any given $r\in[1,\infty]$ and  $\epsilon\in(0,\infty)$,
and for any $x\in\rn$, let
\begin{equation}\label{3.4x}
\widehat{g}_{1,\epsilon}(\ell_{r})(x):=[E_{0}(\ell_{r})(x)]^
{\widetilde{q}_{1}-1}\widetilde{g}_{1,\epsilon}(\ell_{r})(x),
\end{equation}
and, for any $j\in\zz$,
\begin{equation}\label{3.5-x}
\widehat{g}_{0,\epsilon}(\ell_{r})(j,x):=\left[\mathrm{sgn}\,
\overline{\ell(j,x)}\right]|\ell(j,x)|^{\widetilde{p}_{0}-1}
\widetilde{g}_{0,\epsilon}(\ell_{r})(j,x).
\end{equation}
Moreover, for any given $r\in[1,\infty]$ and $i\in\{2,\ldots,n\}$,
and for any $\epsilon\in(0,\infty)$,
let $\widehat{g}_{i,\epsilon}(\ell_{r})$
be as in \eqref{2.41y} with $\ell$ replaced by $\ell_{r}$.

Then we claim that, for any given $\vec{p},$ $\vec{q}\in[1,\infty]^{n}$,
$\vec{\alpha}\in\rn$, and
$r\in[1,\infty]$, and for any $\epsilon\in(0,\infty)$,
\begin{equation}\label{c11}
\left\|\widehat{g}_{0,\epsilon}(\ell_{r})\right\|_
{(\diihz(\rn),l^{r'})}=1
\end{equation}
and
\begin{equation}\label{lherz-B}
\int_{\rn}\sum_{j\in\zz}\ell(j,x)\widehat{g}_{0,\epsilon}(\ell_{r})
(j,x)\,dx\geq (1+\epsilon)^{-2n-1}\|\ell\|_{(\iihz(\rn),l^{r})}.
\end{equation}

Indeed, for any given $\epsilon\in(0,\infty)$, we consider the
following three cases on $r$. If
$r\in(1,\infty)$, then $g_{0,\epsilon}(\ell_{r})(j,\cdot)=1$
for any $j\in\zz$ and $\epsilon\in(0,\infty)$. In this case,
by this, \eqref{3.5-x}, \eqref{SSE}, and
the definition of
$\widetilde{g}_{1,\epsilon}(\ell_{r})$,
we have
\begin{align*}
\left\|\{\widehat{g}_{0,\epsilon}(\ell_{r})(j,\cdot)\}_{j\in\zz}
\right\|_{l^{r'}}&=
\left[\sum_{j\in\zz}|\widehat{g}_{0,\epsilon}(\ell_{r})(j,\cdot)|^
{r'} \right]^{\f{1}{r'}}=\left[\sum_{j\in\zz}|\ell(j,\cdot)|^
{r} [\widetilde{g}_{0,\epsilon}(\ell_{r})(j,\cdot)]^{r'} \right]^
{\f{1}{r'}}\\
&=[E_{0}(\ell_{r})]^{\f{r}{r'}}[E_{0}(\ell_{r})]^{\widetilde{q}_
{1}-r}\widetilde{g}_{1,\epsilon}(\ell_{r})\\
&=[E_{0}(\ell_{r})]^{\widetilde{q}_{1}-1}\widetilde{g}_{1,\epsilon}
(\ell_{r})=\widehat{g}_{1,\epsilon}(\ell_{r}).
\end{align*}
If $r=1$, then $g_{0,\epsilon}(\ell_{r})(j,\cdot)=1$ for any $j\in\zz$
and $\epsilon\in(0,\infty)$. In this case, from this,
\eqref{3.5-x}, and \eqref{SSE}, we deduce that
\begin{align*}
\left\|\{\widehat{g}_{0,\epsilon}(\ell_{r}) (j,\cdot)\}_{j\in\zz}
\right\|_{l^{r'}}&= \sup_{j\in\zz}\left\{|\widehat{g}_{0,\epsilon}
(\ell_{r})(j,\cdot)|\right\}=\sup_{j\in\zz}\left\{|\ell(j,\cdot)|^{0}
\widetilde{g}_{0,\epsilon}(\ell_{r}) (j,\cdot)\right\}\\
&=[E_{0}(\ell_{r})]^{\widetilde{q}_{1}-1}\widetilde{g}_{1,\epsilon}
(\ell_{r})=\widehat{g}_{1,\epsilon}(\ell_{r}).
\end{align*}
If $r=\infty$, then we consider the following two cases on $\Lambda_{\epsilon}(x)$.
If $\Lambda_{\epsilon}(x)\neq \emptyset$,
then, from $\ell(j,x)=0$ if $j\in\zz\cap [B(0,k_{0})]^{\com}$
and $x\in\rn$, we deduce that, for any $\epsilon\in(0,\infty)$ and $x\in\rn$, $\Lambda_
{\epsilon}(x)$ is a finite set and, for any $j\in \Lambda_{\epsilon}(x)$,
$\ell(j,x)>0$.
By this, \eqref{3.2x}, \eqref{SSE}, and \eqref{3.4x},
we obtain
\begin{align*}
\left\|\{\widehat{g}_{0,\epsilon}(\ell_{r}) (j,\cdot)\}_{j\in\zz}
\right\|_{l^{r'}}&= \sum_{j\in\zz}|\widehat{g}_{0,\epsilon}
(\ell_{r})(j,\cdot)|=
\sum_{j\in\zz}|\ell(j,\cdot)|^{0}[E_{0}(\ell_{r})]^
{\widetilde{q}_{1}-1}\f{\one_{\Lambda_{\epsilon}(\cdot)}(j)}
{\#[\Lambda_{\epsilon}(\cdot)]}\widetilde{g}_
{1,\epsilon}(\ell_{r})	\\
&=	[E_{0}(\ell_{r})]^{\widetilde{q}_{1}-1}	
\widetilde{g}_{1,\epsilon}(\ell_{r}) =\widehat{g}_{1,\epsilon}(\ell_{r}).
\end{align*}
If $\Lambda_{\epsilon}(x)=\emptyset$,
then, for any $j\in\zz$, $|\ell(j,\cdot)|=0$. Using this, we have
\begin{equation*}
\left\|\{\widehat{g}_{0,\epsilon}(\ell_{r})(j,\cdot)\}_{j\in\zz}
\right\|_{l^{r'}}=0=[E_{0}(\ell_{r})]^{\widetilde{q}_{1}-1}	
\widetilde{g}_{1,\epsilon}(\ell_{r}) =\widehat{g}_{1,\epsilon}(\ell_{r}).
\end{equation*}
Thus, we conclude that, for any $r\in[1,\infty]$ and $\epsilon\in(0,\infty)$,
\begin{equation*}
\left\|\left\{\widehat{g}_{0,\epsilon}(\ell_{r})(j,\cdot)
\right\}_{j\in\zz}\right\|_{l^{r'}}=\widehat{g}_{1,\epsilon}
(\ell_{r}).
\end{equation*}
This, combined with the estimation of \eqref{c1}, further
implies that $\|\widehat{g}_{0,\epsilon}(\ell_{r})\|_{(\diihz(\rn),
l^{r'})}=1$, namely, \eqref{c11} holds true.

Now, we prove \eqref{lherz-B}.
From \eqref{3.5-x} and \eqref{SSE}, it follows that, for
any $\epsilon\in(0,\infty)$ and $r\in [1,\infty]$,
\begin{align*}
\sum_{j\in\zz}\ell(j,\cdot)\widehat{g}_{0,\epsilon}(\ell_{r})
(j,\cdot)&=
\sum_{j\in\zz}|\ell(j,\cdot)|^{\widetilde{p}_{0}}\widetilde{g}_
{0,\epsilon}(\ell_{r})(j,\cdot)\\
&=\left[\sum_{j\in\zz}|\ell(j,\cdot)|^
{\widetilde{p}_{0}}g_{0,\epsilon}(\ell_{r})(j,\cdot)\right]
[E_{0}(\ell_{r})]^{\widetilde{q}_{1}-\widetilde{p}_{0}}
\widetilde{g}_{1,\epsilon}(\ell_{r}).
\end{align*}
Now, for any given $\epsilon\in(0,\infty)$, we consider the following
two cases on $r$. If $r\in[1,\infty)$,
then $g_{0,\epsilon}(\ell_{r})=1$. Using this, we find that
\begin{align*}
\sum_{j\in\zz}\ell(j,\cdot)\widehat{g}_{0,\epsilon}(\ell_{r})(j,\cdot)=
[E_{0}(\ell_{r})]^{\widetilde{q}_{1}}\widetilde{g}
_{1,\epsilon}(\ell_{r})=\ell_{r}\widehat{g}_{1,\epsilon}(\ell_{r}).
\end{align*}
If $r=\infty$, then we consider the following two cases on $\Lambda_{\epsilon}(\cdot)$.
If $\Lambda_{\epsilon}(\cdot)\neq \emptyset$, then, from \eqref{Ld},
we deduce that
\begin{align*}
\sum_{j\in\zz}\ell(j,\cdot)\widehat{g}_{0,\epsilon}(\ell_{r})(j,\cdot)&
=\left[\sum_{j\in\zz}|\ell(j,\cdot)|^
{\widetilde{p}_{0}}\f{\one_{\Lambda_{\epsilon}(\cdot)}(j)}{\#[\Lambda_
{\epsilon}(\cdot)]}\right][E_{0}(\ell_{r})]^{\widetilde{q}_{1}-\widetilde{p}_{0}}
\widetilde{g}_{1,\epsilon}(\ell_{r})\\
&>(1+\epsilon)^{-1}E_{0}(\ell_{r})[E_{0}(\ell_{r})]^
{\widetilde{q_{1}}-1}\widetilde{g}_{1,\epsilon}(\ell_{r})\\
&=(1+\epsilon)^{-1}[E_{0}(\ell_{r})]^{\widetilde{q}_{1}}
\widetilde{g}_{1,\epsilon}(\ell_{r})\\
&=(1+\epsilon)^{-1}\ell_{r}\widehat{g}_{1,\epsilon}(\ell_{r}).
\end{align*}
If $\Lambda_{\epsilon}(\cdot)= \emptyset$, then
\begin{align*}
&\sum_{j\in\zz}\ell(j,\cdot)\widehat{g}_{0,\epsilon}(\ell_{r})(j,\cdot)\\
&\quad=0=(1+\epsilon)^{-1}[E_{0}(\ell_{r})]^{\widetilde{q}_{1}}
\widetilde{g}_{1,\epsilon}(\ell_{r})=(1+\epsilon)^{-1}\ell_{r}
\widehat{g}_{1,\epsilon}(\ell_{r}).
\end{align*}
Thus, we obtain, for any $r\in[1,\infty]$ and $\epsilon\in(0,\infty)$,
\begin{equation*}
\sum_{j\in\zz}\ell(j,\cdot)\widehat{g}_{0,\epsilon}(\ell_{r})(j,\cdot)
\geq
(1+\epsilon)^{-1}\ell_{r}\widehat{g}_{1,\epsilon}(\ell_{r}),
\end{equation*}
which, similarly to the estimation of \eqref{c22}, further
implies that
\eqref{lherz-B} holds true for any $r\in [1,\infty]$
and hence the claim holds true. By this, similarly to the
estimation of \eqref{kz}, we conclude that \eqref{lherz-gf}
holds true for any $r\in [1,\infty]$.
This finishes the proof of Lemma \ref{lherz-and}.
\end{proof}
The following lemma is a direct consequence of Lemma \ref{lherz-and}.
\begin{lemma}\label{Eass-2}
Let $r\in [1,\infty],$ $\vec{p}$, $\vec{q}\in
[1,\infty]^{n}$, and $\vec{\alpha}\in\rn$. Then, for any
$f:=\{f(j,\cdot)\}_{j\in\zz}\subset\MM(\rn)$,
\begin{equation*}
\|f\|_{(\iihz(\rn),l^{r})}=I(f),
\end{equation*}
where $I(f)$ is as in Lemma \ref{lherz-and}.
\begin{proof}
Let all the symbols be as in the present lemma. For any
given $m\in\nn$ and for any
$f:=\{f(j,\cdot)\}_{j\in\zz}\subset \MM(\rn)$, $j\in\zz$, and
$x\in\rn$, let
\begin{equation*}
f_{m}(j,x):=\begin{cases}
f(j,x) & \text{if}\ |f(j,x)|\leq m,\\
m & \text{otherwise},
\end{cases}
\end{equation*}
and $g_{m}(j,x):=f_{m}(j,x)\one_{A_{m}}(x)$ if
$j\in\zz\cap B(0,m)$, and $g_{m}(j,x):=0$ if
$j\in\zz\cap [B(0,m)]^{\com}$, where $A_{m}$ is as in Lemma
\ref{Am}. Obviously, for any given $m\in\nn$, $g_{m}:=
\{g_{m}(j,\cdot)\}_{j\in\zz}\in (\iihz(\rn),l^{r})$ and, as $m\to\infty$,
$|g_{m}|\uparrow |f|$ almost everywhere. From this and Lemma \ref{lherz-and},
it follows that, for any $m\in\nn$,
\begin{align*}
&\|g_{m}\|_{(\iihz(\rn),l^{r})}\\
&\quad=\sup\left\{\left\|\sum_
{j\in\zz}|g_{m}(j,\cdot)h(j,\cdot)| \right\|_{L^{1}(\rn)}:\
h\in(\diihz(\rn),l^{r'}),\  \|h\|_{(\diihz(\rn),l^{r'})}=
1 \right\}\\
&\quad\leq
\sup\left\{\left\|\sum_{j\in\zz}|f(j,\cdot)h(j,\cdot)|
\right\|_{L^{1}(\rn)}:\ h\in(\diihz(\rn),l^{r'}) ,\ \|h\|_
{(\diihz(\rn),l^{r'})}=1 \right\},
\end{align*}
which, combined with the monotone convergence theorem and
Proposition \ref{ballqfs-3}, further implies that
\begin{align*}
&\|f\|_{(\iihz(\rn),l^{r})}\\
&\quad\leq \sup\left\{\left\|\sum_{j\in\zz}|f(j,\cdot)
h(j,\cdot)|\right\|_{L^{1}(\rn)}:\ h\in(\diihz(\rn),l^{r'}) ,\
\|h\|_{(\diihz(\rn),l^{r'})}=1\right\}.
\end{align*}
On the other hand,
applying Lemma \ref{mhr} and the H\"{o}lder inequality,
we easily obtain the reverse inequality. This finishes
the proof of Lemma \ref{Eass-2}.
\end{proof}
\end{lemma}
As a generalization of Lemma \ref{fland}, we have the following
conclusion.
\begin{lemma}\label{fland-2}
Let $\vec{p}$, $\vec{q}\in [1,\infty]^{n}$
and $\vec{\alpha}\in \rn$. Suppose that $f:=\{f(j,\cdot)\}_
{j\in\zz}\subset \MM(\rn)$ is such that the quantity
\begin{equation*}
S_{\infty}(f):=\sup\left\{\left|\int_{\rn}\sum_{j\in\zz}f(j,x)
g(j,x)\,dx \right|:\ g\in V \ \text{and}\ \|g\|_{(\diihz(\rn),
l^{1})}=1\right\}
\end{equation*}
is finite, where $V$ denotes the set of all the sequences of
simple functions, $\{g(j,\cdot)\}_{j\in\zz}$, satisfying, for some $m\in\nn$,
$\supp(g(j,\cdot))\subset A_{m}$ if $j\in \zz\cap B(0,m)$, and  $g(j,\cdot)=0$ if
$j\in \zz\cap [B(0,m)]^{\com}$, where $A_{m}$ is as in Lemma
\ref{Am}. Then $f\in (\iihz(\rn),l^{\infty})$ and $S_{\infty}(f)
=\|f\|_{(\iihz(\rn),l^{\infty})}$.
\end{lemma}
\begin{proof}
Let all the symbols be as in the present lemma. We consider
the following two cases on $\|f\|_{(\iihz(\rn),l^{\infty})}$.
If $\|f\|_{(\iihz(\rn),l^{\infty})}=0$, then, by Proposition
\ref{equ0}, we have
$$S_{\infty}(f)=\|f\|_{(\iihz(\rn),l^{\infty})}=0.$$
If $\|f\|_{(\iihz(\rn),l^{\infty})}\neq 0$, we then claim
that, for any $m\in\nn$, if
$g_{m}:=\{g_{m}(j,\cdot)\}_{j\in\zz}$ is a
sequence of bounded measurable functions on $\rn,$ $\supp(g_{m}
(j,\cdot))\subset A_{m}$ for any $j\in\zz\cap B(0,m)$, $g_{m}(j,\cdot)=0$
if $j\in \zz\cap[B(0,m)]^{\com}$, and $\|g_{m}\|_{(\diihz(\rn),l^{1})}=1,$
then
\begin{equation*}
\left|\int_{\rn}\sum_{j\in\zz}f(j,x)g_{m}(j,x) \,dx\right|
\leq S_{\infty}(f).
\end{equation*}
Indeed, by \cite[p.\,47,\ Theorem 2.10]{folland} and the
dominated convergence theorem, similarly to the estimation
of \eqref{claim-1}, we conclude that the above claim holds true.

For any given $k\in\nn$ and for any $j\in\zz$ and $x\in\rn$, let
\begin{equation*}
\phi_{k}(j,x):=\begin{cases}
f(j,x) & \text{if}\ |f(j,x)|\leq k,\\
k & \text{otherwise},
\end{cases}
\end{equation*}
and $f_{k}(j,x):=\phi_{k}(j,x)\one_{A_{k}}(x)$ if
$j\in\zz\cap B(0,k)$, and $f_{k}(j,x):=0$ if
$j\in\zz\cap [B(0,k)]^{\com}$.
Obviously, for any given $k\in\nn$, $\{f_{k}(j,\cdot) \}_{j\in\zz}\in
(\iihz(\rn),l^{\infty})$ and, as $k\to \infty$, $|f_{k}|\uparrow |f|$
almost everywhere.
By this and Proposition
\ref{ballqfs-3}, we conclude that there exists an $N\in\nn$
such that, for any
$k>N$, $\|f_{k}\|_{(\iihz(\rn),l^{\infty})}\neq 0$. Then,
for any $\epsilon\in(0,\infty)$, $k>N$, $j\in\zz$, and $x\in\rn$,
let
\begin{equation*}
\widehat{g}_{\epsilon,k}(j,x):=\left[\mathrm{sgn}\,\overline{f(j,x)}\right]
\widetilde{g}_{\epsilon,k}(j,x),
\end{equation*}
where $\widetilde{g}_{\epsilon,k}$ is as in \eqref{SSE}
with $\ell_{r}(x)$ there replaced by $\|\{f_{k}(j,x) \}_
{j\in\zz}\|_{l^{\infty}}$. By \eqref{c11}, we have, for
any $k>N$, $\|\widehat{g}_{\epsilon,k}\|_{(\diihz(\rn),
l^{1})}=1$. Note that, for any given $\epsilon\in (0,\infty)$
and $k\in\nn$, and for any $j\in\zz$, $\widehat{g}_
{\epsilon,k}(j,\cdot)$ is a bounded measurable function on
$\rn$, $\supp(\widehat{g}_{\epsilon,k}(j,\cdot))\subset A_{k}$,
and $\widehat{g}_{\epsilon,k}(j,\cdot)=0$ if $j\in \zz\cap
[B(0,k)]^{\com}$. Indeed, since $\|\{f_{k}(j,\cdot) \}_{j\in\zz}\|
_{l^{\infty}}$ is bounded and $\|\one_{\supp(\widehat{g}_{\epsilon,k}(j,\cdot))}\|_{\iihz(\rn)}
\leq \|\one_{A_{k}}\|_{\iihz(\rn)}<\infty$, we infer that, for any
$j\in\zz$, $\widehat{g}_{\epsilon,k}(j,\cdot)$ is bounded.
From this, Proposition \ref{ballqfs-3},
\eqref{lherz-B}, $|f_{k}(j,\cdot)|\leq |f(j,\cdot)|$
for any $j\in\zz$, and the above claim, we deduce that
\begin{align*}
\|f\|_{(\iihz(\rn),l^{\infty})}&\leq  \varliminf_
{k\to\infty}\|f_{k}\|_{(\iihz(\rn),l^{\infty})}\leq
(1+\epsilon)^{2n+1}\varliminf_{k\to\infty}\int_{\rn}
\sum_{j\in\zz}|f_{k}(j,x)\widehat{g}_{\epsilon,k}(j,x)|
\,dx\\
&\leq (1+\epsilon)^{2n+1}\varliminf_{k\to\infty}
\int_{\rn}\sum_{j\in\zz}|f(j,x)\widehat{g}_{\epsilon,k}
(j,x)|\,dx\\
&=(1+\epsilon)^{2n+1}\varliminf_{k\to\infty}\int_
{\rn}\sum_{j\in\zz}f(j,x)\widehat{g}_{\epsilon,k}(j,x)\,dx
\leq (1+\epsilon)^{2n+1}S_{\infty}(f).
\end{align*}
Letting $\epsilon\rightarrow 0^{+}$, we obtain
$$\|f\|_{(\iihz(\rn),l^{\infty})}\leq S_{\infty}(f)<\infty.$$
On the other hand, applying the H\"{o}lder inequality and Lemma \ref{mhr},
we easily find that
$$S_{\infty}(f)\leq \|f\|_{(\iihz(\rn),
l^{\infty})}.$$
This finishes the proof of Lemma \ref{fland-2}.
\end{proof}
The following conclusion is a simple consequence of
Lemma \ref{fland-2}.
\begin{corollary}\label{sppAm}
Let $\vec{p}$, $\vec{q}\in [1,\infty]
^{n}$ and $\vec{\alpha}\in\rn$. Suppose that
$f:=\{f(j,\cdot) \}_{j\in\zz}\subset \MM(\rn)$ is such
that the quantity
\begin{equation*}
S_{\infty}'(f):=\sup\left\{\left|\int_{\rn}\sum_
{j\in\zz}f(j,x)g(j,x)\,dx \right|:\ g\in V' \
\text{and}\ \|g\|_{(\diihz(\rn),l^{1})}=1\right\}
\end{equation*}
is finite, where $V'$ denotes the set of all the sequences
of simple functions, $\{g(j,\cdot)\}_{j\in\zz}$, satisfying
that,
for some $m\in\nn$, $\supp(g(j,\cdot))=A_{m}$ if $j\in \zz
\cap B(0,m)$, where $A_{m}$ is as in Lemma \ref{Am}, and
$g(j,\cdot)=0$ if $j\in\zz\cap [B(0,m)]^{\com}$. Then $f\in
(\iihz(\rn),l^{\infty})$ and $S_{\infty}'(f)=\|f\|_
{(\iihz(\rn),l^{\infty})}$.	
\end{corollary}
\begin{proof}
Let all the symbols be as in the present corollary. For
any given $g\in V$ with $\|g\|_{(\diihz(\rn),l^{1})}=1$,
where $V$ is as in Lemma \ref{fland-2}, assume that, for
some $m\in\nn$, $\supp(g(j,\cdot))\subset A_{m}$ if
$j\in\zz\cap B(0,m)$, and $g(j,\cdot)=0$ if $j\in\zz\cap[ B
(0,m)]^{\com}$.
Then, for any given $\epsilon\in (0,\infty)$ and for any
$j\in\zz\cap B(0,m)$ and $x\in\rn$, let
\begin{equation}\label{geon}
g_{\epsilon}(j,x):=\begin{cases}
g(j,x) & \text{if}\ x\in A_{m}\ \text{and}\ g(j,x)
\neq 0,\\
\epsilon & \text{if}\ x\in A_{m}\ \text{and}\ g(j,x)
=0,\\
0 & \text{otherwise}
\end{cases}
\end{equation}
and, for any $j\in\zz\cap [B(0,m)]^{\com}$ and $x\in\rn$,
$g_{\epsilon}(j,x):=0$. Obviously, $\supp(g_{\epsilon}
(j,\cdot))=A_{m}$ if $j\in\zz\cap B(0,m)$. Moreover, by
Proposition \ref{E-lat}, we obtain
$$\|g_{\epsilon}\|_
{(\diihz(\rn),l^{1})}\geq \|g\|_{(\diihz(\rn),l^{1})}=1$$
and
$$\|g_{\epsilon}\|_{(\diihz(\rn),l^{1})}\ls \left\|
\one_{A_{m}}\right\|_{\diihz(\rn)}<
\infty.$$
Thus, $g_{\epsilon}/\|g_{\epsilon}\|_{(\diihz(\rn),l^{1})}
\in V'$. From this and the definition of $S_{\infty}'(f)$,
it follows that, for any $\epsilon\in(0,\infty)$,
\begin{equation}\label{Sf-A}
S_{\infty}'(f)\geq \left|\int_{\rn}\sum_{j\in\zz}
f(j,x)\f{g_{\epsilon}(j,x)}{\|g_{\epsilon}\|_
{(\diihz(\rn),l^{1})}}\,dx \right|.
\end{equation}
Using \eqref{geon}, we conclude
that, for any $\delta\in(0,\infty)$, there exists
a $K:=\delta\in(0,\infty)$ such that, for any $\epsilon
\in (0,K)$, $j\in\zz$, and $x\in\rn$,
$$|g_{\epsilon}(j,x)|
\leq |g(j,x)|+\delta\one_{\zz\cap B(0,m)}(j)\one_{A_{m}}(x),$$
which, combined with Proposition \ref{E-lat}, further
implies that $$\|g_{\epsilon}\|_{(\diihz(\rn),l^{1})}
\leq \|g\|_{(\diihz(\rn),l^{1})}+\delta\left\|\sum_{j\in
\zz\cap B(0,m)}\one_{A_{m}}\right\|_{\diihz(\rn)}.$$
Letting $\delta\rightarrow 0^{+}$, we obtain $\lim_{\epsilon
\to 0^{+}}\|g_{\epsilon}\|_{(\diihz(\rn),l^{1})}\leq
\|g\|_{(\diihz(\rn),l^{1})}=1$ and hence
$$\lim_{\epsilon
\to 0^{+}}\|g_{\epsilon}\|_{(\diihz(\rn),l^{1})}=\|g\|_
{(\diihz(\rn),l^{1})}=1.$$
By \eqref{geon} and the definition of $S_{\infty}'(f)$, we conclude that
\begin{align*}
&\left|\int_{\rn}\sum_{j\in\zz}f(j,x)\f{g_{\epsilon}
	(j,x)}{\|g_{\epsilon}\|_{(\diihz(\rn),l^{1})}}
\,dx \right|\\
&\quad\ls
\int_{\rn}\sum_{j\in\zz\cap B(0,m)}|f(j,x)|\one_
{A_{m}}(x)\,dx\ls S_{\infty}'(f)<\infty.\nonumber
\end{align*}
This, together with \eqref{Sf-A} and the dominated convergence theorem,
further implies that
\begin{align*}
S_{\infty}'(f)
&\geq \lim_{\epsilon\to 0^{+}} \left|\int_{\rn}\sum_
{j\in\zz}f(j,x)\f{g_{\epsilon}(j,x)}{\|g_{\epsilon}
\|_{(\diihz(\rn),l^{1})}}\,dx \right|\\
&= \left|\int_{\rn}\sum_{j\in\zz}f(j,x)\lim_
{\epsilon\to 0^{+}}\f{g_{\epsilon}(j,x)}{\|g_{\epsilon}
\|_{(\diihz(\rn),l^{1})}}\,dx \right|\\
&= \left|\int_{\rn}\sum_{j\in\zz}f(j,x)g(j,x)\,dx
\right|.
\end{align*}
By the arbitrariness of $g\in V$ with $\|g\|_{(\diihz(\rn),
l^{1})}=1$, we conclude that $S_{\infty}(f)\leq
S_{\infty}'(f)<\infty$, where $S_{\infty}(f)$ is as
in Lemma \ref{fland-2}. On the other hand, it is
easy to show that $S_{\infty}(f)\geq S_{\infty}'(f)$.
This, combined with Lemma \ref{fland-2}, further
implies that
$$\|f\|_{(\iihz(\rn),l^{\infty})}=S_
{\infty}(f)=S_{\infty}'(f)<\infty,$$
which completes
the proof of Corollary \ref{sppAm}.
\end{proof}
Now, we aim to establish the Riesz--Thorin interpolation
theorem
on the mixed-norm Herz space $\iihz(\rn)$ and the
$l^{\infty}$-Herz type mixed-norm space.
Recall that Benedek and Panzone proved the Riesz--Thorin
interpolation theorem for the mixed Lebesgue spaces
$L^{\vec{p}}(\rn)$ (see \cite[p.\,316, Theorem 2]{BP}).

In what follows,
we denote by $H(\rn)$ the class of all simple functions as follows:
\begin{equation}\label{Hsim}
f=\sum_{k=1}^{N}a_{k}\one_{E_{k}},
\end{equation}
where $N\in\nn$, $\{a_{k} \}_{k=1}^{N}\subset\cc$, and, for any $k\in\{1,\ldots,N \}$,
$E_{k}\subset A_{m}$ for some $m\in\nn$, $E_{k}$ is of the form
\begin{equation}\label{pppeds}
E_{k}=E_{k_{1}}\times \cdots\times E_{k_{n}}
\end{equation}
with $E_{k_{i}}\subset \rr$ for any $i\in\{1,\ldots,n \}$, and $\{E_{k} \}_{k=1}^{N}$
are pairwise disjoint. Recall that $E_{k}$ in \eqref{pppeds} is called the \emph{parallelepiped}.

Now, we have the following generalization of the
Riesz--Thorin interpolation theorem.
\begin{theorem}\label{threeL}
For any $j\in\{0,1 \}$, let $\vec{\alpha}^{(j)}$,
$\vec{\beta}^{(j)}\in\rn$, $\vec{p}^{(j)},$ $\vec{q}
^{(j)}$, $\vec{s}^{(j)},$ $\vec{t}^{(j)}\in[1,\infty]^{n}$,
$\theta\in[0,1]$, and $\vec{p},$ $\vec{q},$ $\vec{\alpha}$,
$\vec{s},$ $\vec{t}$, $\vec{\beta}\in\rn$ satisfy
$$\begin{cases}\dis
\f{1}{\vec{p}}=\f{1-\theta}{\vec{p}^{(0)}}+\f{\theta}
{\vec{p}^{(1)}}, \quad\dis
&\dis\f{1}{\vec{q}}=\f{1-\theta}{\vec{q}^{(0)}}+\f{\theta}
{\vec{q}^{(1)}},\\
\dis\f{1}{\vec{s}}=\f{1-\theta}{\vec{s}^{(0)}}+\f{\theta}
{\vec{s}^{(1)}},  \quad
&\dis\f{1}{\vec{t}}=\f{1-\theta}{\vec{t}^{(0)}}+\f{\theta}
{\vec{t}^{(1)}},\\
\vec{\alpha}=(1-\theta)\vec{\alpha}^{(0)}+\theta\vec
{\alpha}^{(1)},\quad &\dis\vec{\beta}=(1-\theta)\vec{\beta}
^{(0)}+\theta\vec{\beta}^{(1)}.
\end{cases}$$
Let $T$ be a linear operator satisfying that there exist
positive constants $M_{0}$ and $M_{1}$ such that,
for any $f\in H(\rn)$,

\begin{equation}\label{point-0}
\|T(f)\|_{\mEz}\leq M_{0}\|f\|_{\nEz}
\end{equation}
and
\begin{equation}\label{point-1}
\|T(f)\|_{\mEo}\leq M_{1}\|f\|_{\nEo}.
\end{equation}
Then, for any $f\in H(\rn)$,
\begin{equation}\label{000}
\|T(f)\|_{\mE}\leq M_{0}^{1-\theta}M_{1}^{\theta}
\|f\|_{\nE}.
\end{equation}
\end{theorem}
Let all the symbols be as in Theorem \ref{threeL}. In the
remainder of this
section, we introduce the following more symbols.
For any $j\in\{0,1 \}$, let
$\vec{\alpha}^{(j)}:=(\alpha_{1}^{(j)},\ldots,\alpha_{n}^{(j)})$,
$\vec{\beta}^{(j)}:=(\beta_{1}^{(j)},\ldots,\beta_{n}^{(j)})$,
$\vec{p}^{(j)}:=(p_{1}^{(j)},\ldots,p_{n}^{(j)})$,
$\vec{q}^{(j)}:=(q_{1}^{(j)},\ldots,q_{n}^{(j)})$,
$\vec{s}^{(j)}:=(s_{1}^{(j)},\ldots,s_{n}^{(j)})$,
$\vec{t}^{(j)}:=(t_{1}^{(j)},\ldots,t_{n}^{(j)})$,
$\vec{\alpha}:=(\alpha_{1},\ldots,\alpha_{n})$,
$\vec{\beta}:=(\beta_{1},\ldots,\beta_{n})$,
$\vec{p}:=(p_{1},\ldots,p_{n})$,
$\vec{q}:=(q_{1},\ldots,q_{n})$,
$\vec{s}:=(s_{1},\ldots,s_{n})$, and
$\vec{t}:=(t_{1},\ldots,t_{n})$. For any $k\in\{1,\ldots,n \}$ and
$z\in\cc$, let
$$\alpha_{k}(z):=(1-z)\alpha_{k}^{(0)}+z\alpha_{k}^{(1)},
\quad\ \  \beta_{k}(z):=(1-z)\beta_{k}^{(0)}+z
\beta_{k}^{(1)},$$
$$\xi_{k}(z):=(1-z)/p_{k}^{(0)}+z/p_{k}^{(1)},\quad\ \
\eta_{k}(z):=(1-z)/q_{k}^{(0)}+z/q_{k}^{(1)},$$
and
$$\lambda_{k}(z):=(1-z)/s_{k}^{(0)}+z/s_{k}^{(1)},\quad\ \
\mu_{k}(z):=(1-z)/t_{k}^{(0)}+z/t_{k}^{(1)}.
$$
Moreover, for any $k\in\{1,\ldots,n \}$, let $\xi_{k}:
=1/p_{k}=\xi_{k}(\theta)$, $\eta_{k}:=1/q_{k}=\eta_{k}
(\theta)$, $\lambda_{k}:=1/s_{k}=\lambda_{k}(\theta
)$, and $\mu_{k}:=1/t_{k}=\mu_{k}(\theta)$. In particular,
we have $\alpha_{k}(\theta)=\alpha_{k}$
and $\beta_{k}(\theta)=\beta_{k}$.

In what follows, for any given $m\in\nn$, we denote by $H_{m}(\zz\times\rn)$
the set of all the sequences $\phi:=\{\phi(j,\cdot) \}_{j\in\zz}$ of simple functions satisfying that
$\supp(\phi(j,\cdot))= A_{m}$
if $j\in\zz\cap B(0,m)$, and $\phi(j,\cdot)=0$ if
$j\in\zz\cap [B(0,m)]^{\com}$, where $A_{m}$ is as in
Lemma \ref{Am}. Let $\psi\in H(\rn)$ and $\phi\in H_{m}(\zz\times\rn)$ for some
$m\in\nn$.
Now, for any given $k\in\{1,\ldots,n-1 \}$ and $z\in\cc$, we
define $\cf_{z,k}(\psi)$ by setting, for any $x:=(x_{1},
\ldots,x_{n})\in\rn$,
\begin{align}\label{Fzk}
&\cf_{z,k}(\psi)(x_{k+1},\ldots,x_{n})\\
&\quad:=
\begin{cases}
\left[\left\|\psi(\cdot,x_{k+1},\ldots,x_{n})\right\|
_{\dot{E}^{\vec{\beta}_{k},\vec{s}_{k}}_{\vec{t}_{k}}
(\rr^{k})} \right]^{\widetilde{\mu}_{k+1}(z)-
\widetilde{\lambda}_{k}(z)}\\
&\hspace{-3cm} \text{if}\  \|
\psi(\cdot,x_{k+1},\ldots,x_{n})\|
_{\dot{E}^{\vec{\beta}_{k},\vec{s}_{k}}_{\vec{t}_{k}}
(\rr^{k})}\neq 0,\\
0 &\hspace{-3cm}\text{otherwise},
\end{cases}\nonumber
\end{align}
where, for any $k\in\{1,\ldots,n \}$ and $z\in\cc$,
\begin{equation}\label{mulam}
\widetilde{\mu}_{k}(z):=
\begin{cases}
\mu_{k}(z)/\mu_{k} &\text{if}\ \mu_{k}\in (0,1],\\
1 &\text{if}\ \mu_{k}=0
\end{cases}
\end{equation}
and
\begin{equation}\label{mulam1}
\widetilde{\lambda}_{k}(z):=
\begin{cases}
\lambda_{k}(z)/\lambda_{k} &\text{if}\ \lambda_{k}\in(0,1],\\
1 &\text{if}\ \lambda_{k}=0.
\end{cases}
\end{equation}
For any given $k\in\{1,\ldots,n \}$ and $i_{k}\in\zz$, and for any
$x:=(x_{1},\ldots,x_{n})\in\rn$, let
\begin{equation}\label{Pik-1}
P_{i_{k}}(\psi)(x_{k},\ldots,x_{n}):=\left\|\psi(\cdot,
x_{k},\ldots,x_{n}) \right\|_{\dot{E}^{\vec{\beta}_{k-1},
\vec{s}_{k-1}}_{\vec{t}_{k-1}}(\rr^{k-1})}
\one_{R_{i_{k}}} (x_{k}),
\end{equation}
where $\|\psi\|_{\dot{E}^{\vec{\beta}_{0},\vec{s}_{0}}_
{\vec{t}_{0}}(\rr^{0})}:=|\psi|$. Moreover, for any given
$k\in\{1,\ldots,n-1 \}$, $i_{k}\in\zz$, and $s\in\cc$,
and for any $x:=(x_{1},\ldots,x_{n})\in\rn$, let
\begin{align}\label{Pik}
&P_{i_{k},s}(\psi)(x_{k+1},\ldots,x_{n})\\
&\quad:=\begin{cases}
\left\|P_{i_{k}}(\psi)(\cdot,x_{k+1},\ldots,x_{n})\right
\|_{L^{t_{k}}(\rr)}^{s}\\
&\hspace{-3cm}\text{if}\ \left\|P_{i_{k}}
(\psi)(\cdot,x_{k+1},\ldots,x_{n})\right\|_{L^{t_{k}}(\rr)}
\neq 0,\\
0 &\hspace{-3cm}\text{otherwise},
\end{cases}\nonumber
\end{align}
and $P_{i_{n},s}(\psi)$ be as in \eqref{Pik} with
$t_{k}$ and $P_{i_{k}}(\psi)(x_{k},\ldots,x_{n})$ there replaced,
respectively, by $t_{n}$ and $P_{i_{n}}(\psi)(x_{n})$.
Then, for any given $k\in\{ 1,\ldots,n-1\}$ and $z\in\cc$, and
for any $x:=(x_{1},\ldots,x_{n})\in\rn$, let
\begin{align}\label{Fhatzk}
\widehat{\cf}_{z,k}(\psi)(x_{k},\ldots,x_{n}):=&\,
\sum_{i_{k}\in\zz}2^{-i_{k}\beta_{k}(z)+i_{k}\beta_{k}
\widetilde{\lambda}_{k}(z)}\\
&\,\times
P_{i_{k},\widetilde{\lambda}
_{k}(z)-\widetilde{\mu}_{k}(z)}(\psi)(x_{k+1},\ldots,
x_{n})\one_{R_{i_{k}}}(x_{k})\nonumber
\end{align}
and
\begin{equation}\label{3.18x}
\widehat{\cf}_{z,n}(\psi)(x_{n}):=\sum_{i_{n}\in\zz}2^
{-i_{n}\beta_{n}(z)+i_{n}\beta_{n}\widetilde{\lambda}_{n}(z) }
P_{i_{n},\widetilde{\lambda}_{n}(z)-\widetilde{\mu}_{n}(z)}
(\psi)\one_{R_{i_{n}}}(x_{n}).
\end{equation}
For any given $z\in\cc$, if
$n\in\nn$ and $n>1$, then, for any $x:=(x_{1},\ldots,x_{n})\in\rn$, let
\begin{align}\label{Fzpsi}
F_{z}^{(n)}(\psi)(x):=&\,	
[\psi(x)]^{\widetilde{\mu}_{1}(z)}\left[\mathrm{sgn}\,(\overline
{\psi(x)})\right]\prod_{k=1}^{n-1}\cf_{z,k}(\psi)(x_{k+1},
\ldots,x_{n})\\
&\,\times\prod_{k=1}^{n}\widehat{\cf}_{z,k}(\psi)
(x_{k},\ldots,x_{n}),\nonumber
\end{align}
if $n:=1$, then, for any $x\in\rr$, let
\begin{equation}\label{3.19x}
F_{z}^{(n)}(\psi)(x):=[\psi(x)]^{\widetilde{\mu}
_{1}(z)}\left[\mathrm{sgn}\,(\overline{\psi(x)})
\right]\widehat{\cf}_{z,1}(\psi)(x),
\end{equation}
where, for any given
$x\in\rn$, if $\psi(x)=0$, then let $[\psi(x)]^{\widetilde
{\mu}_{1}(z)}:=0$; if $\psi(x)\neq 0$, then let $[\psi(x)]
^{\widetilde{\mu}_{1}(z)}:=e^{\widetilde{\mu}_{1}(z)\ln(|\psi(x)|)}$.
Similarly, for any $y:=(y_{1},\ldots,y_{n})\in\rr^{n}$, let
$\varphi(y):=\|\{\phi(j,y)\}_{j\in\zz}\|_{l^{1}}$. For any
$k\in\{1,\ldots, n \}$ and $z\in\cc$, let
\begin{equation}\label{etaxi}
\widetilde{\eta}_{k}(z):=
\begin{cases}\displaystyle
\f{1-\eta_{k}(z)}{1-\eta_{k}} &\text{if}\ \eta_{k}\in[0,1),\\
1 &\text{if}\ \eta_{k}=1,
\end{cases}
\end{equation}
\begin{equation}\label{etaxi1}
\widetilde{\xi}_{k}(z):=
\begin{cases}\displaystyle
\f{1-\xi_{k}(z)}{1-\xi_{k}} &\text{if}\ \xi_{k}\in[0,1),\\
1 &\text{if}\ \xi_{k}=1,
\end{cases}
\end{equation}
and $\widetilde{\xi}_{0}(z):=1$. Then, for any given $k\in\{0,
\ldots,n-1 \}$ and $z\in\cc$, we define $\cg_{z,k}(\varphi)$
by setting, for any $y:=(y_{1},\ldots,y_{n})\in\rr^{n}$,
\begin{align}\label{ggzk}
&\cg_{z,k}(\varphi)(y_{k+1},\ldots,y_{n})\\
&\quad:=\begin{cases}
\left[\|\varphi(\cdot,y_{k+1},\ldots,y_{n})\|_{\dot{E}^{-\vec{
\alpha}_{k},\vec{p}_{k}'}_{\vec{q}_{k}'}(\rr^{k})}
\right]^{\widetilde{\eta}_{k+1}(z)-\widetilde{\xi}_{k}(z)}\\
&\hspace{-3cm} \text{if}\ \|\varphi(\cdot,y_{k+1},\ldots,y_{n})\|_
{\dot{E}^{-\vec{\alpha}_{k},\vec{p}_{k}'}_{\vec{q}_{k}'}
(\rr^{k})}\neq 0,\\
0 &\hspace{-3cm}\text{otherwise},
\end{cases}\nonumber
\end{align}
where $\|\varphi\|_{\dot{E}^{-\vec{\alpha}_{0},\vec{p}'_{0}}_{
\vec{q}'_{0}}(\rr^{0})}:=|\varphi|$.
For any given $k\in\{1,\ldots,n \}$ and $i_{k}\in\zz$, and for any $y:=(y_{1},
\ldots,y_{n})\in\rn$, let
\begin{equation*}
Q_{i_{k}}(\varphi)(y_{k},\ldots,y_{n}):=\|\varphi(\cdot,y_{k},
\ldots,y_{n}) \|_{\dot{E}^{-\vec{\alpha}_{k-1},\vec{p}'_{k-1}}_
{\vec{q}'_{k-1}}(\rr^{k-1})}\one_{R_{i_{k}}}(y_{k})
\end{equation*}
and, for any given $k\in\{1,\ldots,n-1 \}$ and $s\in\cc$, and for
any $y_{k+1},\ldots,y_{n}\in\rr$, let
\begin{align}\label{Qim}
&Q_{i_{k},s}(\varphi)(y_{k+1},\ldots,y_{n})\\
&\quad:=\begin{cases}
\left\|Q_{i_{k}}(\varphi)(\cdot,y_{k+1},\ldots,y_{n})\right\|_
{L^{q_{k}'}(\rr)}^{s} &\text{if}\ \left\|Q_{i_{k}}(\varphi)(
\cdot,y_{k+1},\ldots,y_{n})\right\|_{L^{q_{k}'}(\rr)}\neq 0,\\
0 &\text{otherwise}
\end{cases}\nonumber
\end{align}
and $Q_{i_{n},s}(\varphi)$ be as in \eqref{Qim} with $q_{k}'$
and $Q_{i_{k}}(\varphi)(\cdot, y_{k+1},\ldots,y_{n})$ replaced,
respectively, by $q_{n}'$ and $Q_{i_{n}}(\varphi)(\cdot)$.
Then, for any
$k\in\{1,\ldots,n-1 \}$ and $z\in\cc$, we define $\widehat{\cg}_{z,k}
(\varphi)$ by setting, for any
$y:=(y_{1},\ldots,y_{n})\in\rr^{n}$,
\begin{equation}\label{hatgg}
\widehat{\cg}_{z,k}(\varphi)(y_{k},\ldots,y_{n}):=
\sum_{i_{k}\in\zz}2^{i_{k}\alpha_{k}(z)-i_{k}\alpha_{k}
\widetilde{\xi}_{k}(z)} Q_{i_{k},\widetilde{\xi}_{k}(z)-
\widetilde{\eta}_{k}(z)}(\varphi)(y_{k+1},\ldots,y_{n})
\one_{R_{i_{k}}}(y_{k})
\end{equation}
and
$$\widehat{\cg}_{z,n}(\varphi)(y_{n}):=
\sum_{i_{n}\in\zz}2^{i_{n}\alpha_{n}(z)-i_{n}\alpha_{n}\widetilde
{\xi}_{n}(z)} Q_{i_{n},\widetilde{\xi}_{n}(z)-\widetilde{\eta}
_{n}(z)}(\varphi)\one_{R_{i_{n}}}(y_{n}).$$
For any given $z\in\cc$ and for any $j\in\zz$ and $y:=(y_{1},\ldots,y_{n})\in\rn$, let
\begin{align}\label{Gzphi}
G_{z}^{(n)}(\phi)(j,y):=&\,
|\phi(j,y)|\left[\mathrm{sgn}\,(\overline{\phi(j,y)})\right]\prod_{k=0}
^{n-1}\cg_{z,k}(\varphi)(y_{k+1},\ldots,y_{n})\\
&\,\times\prod_{k=1}^{n}\widehat{\cg}_{z,k}(\varphi)(y_{k},\ldots,y_{n}).\nonumber
\end{align}

To show Theorem \ref{threeL}, we need the following several
conclusions.
\begin{proposition}\label{ae=0}
Let $n\in\nn$. For any $z\in\cc$, let $F_{z}^{(n)}(\psi)$ and $G_{z}^{(n)}(\phi)$ be, respectively, as in
\eqref{Fzpsi}, \eqref{3.19x}, and \eqref{Gzphi}. Then, for any $\theta\in [0,1]$,
$F_{\theta}^{(n)}(\psi)=\psi$
almost everywhere on $\rn$ and, for any $j\in\zz$, $G_{\theta}^{(n)}(
\phi)(j,\cdot)=\phi(j,\cdot)$ almost everywhere on $\rn$.
\end{proposition}
\begin{proof}
Let all the symbols be as in the present
proposition. We only show $F_{\theta}^{(n)}(\psi)=\psi$ almost
everywhere on $\rn$ because the case of $G_{\theta}^{(n)}(\phi)$ is
similar and we omit the details. We perform induction on $n$.
If $n:=1$, then, by \eqref{3.19x} and \eqref{3.18x}, we conclude
that
\begin{align*}
F_{\theta}^{(1)}(\psi)(x_{1})=\psi(x_{1})\left[\mathrm{sgn}\,(\overline{\psi(x_{1})})
\right]\sum_{i_{1}\in\zz}P_{i_{1},0}(\psi)\one_
{R_{i_{1}}}(x_{1}).
\end{align*}
Then, for any given $R_{i_{1}}$ with $i_1\in\zz$, we consider the following two cases on
$\|\psi \one_{R_{i_{1}}}\|_{L^{t_{1}}(\rr)}$. If $\|\psi \one_
{R_{i_{1}}}\|_{L^{t_{1}}(\rr)}\neq 0$,
then $P_{i_{1},0}(\psi)=1$ and hence $F_{\theta}^{(1)}(\psi)(x_{1})
=\psi(x_{1})$ for
any $x_{1}\in R_{i_{1}}$. If $\|\psi \one_{R_{i_{1}}}
\|_{L^{t_{1}}(\rr)}= 0$, then $\psi\one_{R_{i_{1}}}=0$ almost
everywhere on $R_{i_{1}}$ and $P_{i_{1},0}(\psi)=0$. Thus,
$\psi(x_{1})=0$ for almost every $x_{1}\in R_{i_{1}}$, and
$F_{\theta}^{(1)}(\psi)(x_{1})=0$ for any $x_{1}\in R_{i_{1}}$.
Therefore, for any given $R_{i_{1}}$ with $i_1\in\zz$,
$F_{\theta}^{(1)}(\psi)=\psi$ almost everywhere on $R_{i_{1}}$
and hence almost everywhere on $\rr$.

If $n:=2$, then, by \eqref{Fzpsi}, we obtain
\begin{align*}
F_{\theta}^{(2)}(\psi)(x_{1},x_{2})=\psi(x_{1},x_{2})
\left[\mathrm{sgn}\,(\overline{\psi(x_{1},x_{2})})\right]
\cf_{\theta,1}(\psi)(x_{2}) \widehat{\cf}_{\theta,1}
(\psi)(x_{1},x_{2})\widehat{\cf}_{\theta,2}(\psi)(x_{2}).
\end{align*}
Now, we claim that, for almost every given $x_{2}\in\rr$,
$F_{\theta}^{(2)}(\psi)(x_{1},x_{2})=\psi(x_{1},x_{2})$ for almost
every $x_{1}\in\rr$. Indeed,  using the conclusion
of the case $n:=1$, we conclude that, for any given $x_{2}
\in\rr$ and for almost every $x_{1}\in \rr$,
\begin{equation*}
F_{\theta}^{(2)}(\psi)(x_{1},x_{2})=\psi(x_{1},x_{2})
\cf_{\theta,1}(\psi)(x_{2})\widehat{\cf}_{\theta,2}(\psi)(x_{2}).
\end{equation*}
Then we consider two cases on $R_{i_{2}}$ with $i_2\in\zz$. If
$R_{i_{2}}$ satisfies $\|\|\psi\|_{\dot{E}_{t_{1}}^
{\beta_{1},s_{1}}(\rr)} \one_{R_{i_{2}}}\|_{L^{t_{2}}
(\rr)}=0$, then,
applying Proposition \ref{equ0}, we find that, for
almost every $x_{2}\in R_{i_{2}}$, $\psi(\cdot,x_{2})=0$ almost
everywhere on $\rr$ and hence $F_{\theta}^{(2)}(\psi)(\cdot,x_{2})=
\psi(\cdot,x_{2})$ almost everywhere on $\rr$.
If $R_{i_{2}}$ satisfies $\|\|\psi\|_{\dot{E}_{t_{1}}^
{\beta_{1},s_{1}}(\rr)} \one_{R_{i_{2}}}\|_{L^{t_{2}}
(\rr)}\neq 0$, then, from \eqref{Fhatzk}, it follows that,
for any $x_{2}\in R_{i_{2}}$,
$\widehat{\cf}_{\theta,2}(\psi)(x_{2})=1$.
In this case, we consider the following two cases on
$x_{2}\in R_{i_{2}}$. If $x_{2}\in R_{i_{2}}$
satisfies $\|\psi(\cdot,x_{2})\|_{\dot{E}_{t_{1}}^
{\beta_{1},s_{1}}(\rr)}=0$, then, using
Proposition \ref{equ0}, we find that $\psi(x_{1},x_{2})=0$ for
almost every $x_{1}\in\rr$; moreover, by \eqref{Fzk}, we have
$\cf_{\theta,1}(\psi)(x_{2})=0$ and hence $F_
{\theta}^{(2)}(\psi)(x_{1},x_{2})=\psi(x_{1},x_{2})$
for almost every $x_{1}\in\rr$.
If $x_{2}\in R_{i_{2}}$ satisfies $\|\psi(\cdot,x
_{2})\|_{\dot{E}_{t_{1}}^{\beta_{1},s_{1}}(\rr)}\neq 0$,
then
$$\cf_{\theta,1}(\psi)(x_{2})=\widehat{\cf}
_{\theta,2}(\psi)(x_{2})=1.$$
In this case, $F_{\theta}^{(2)}(\psi)(\cdot,x_{2})=\psi(\cdot,x_{2})$ almost
everywhere on $\rr$. Therefore, for any $R_{i_{2}}$
satisfying $\|\|\psi\|_{\dot{E}_{t_{1}}^{\beta
_{1},s_{1}}(\rr)} \one_{R_{i_{2}}}\|_{L^{t_{2}}(\rr)}
\neq 0$ and for any $x_{2}\in R_{i_{2}}$,
$F_{\theta}^{(2)}(\psi)(\cdot,x_{2})=\psi(\cdot,x_{2})$ almost
everywhere on $\rr$, which further implies that the above
claim holds true. By this claim and the Fubini theorem,
we conclude that
$F_{\theta}^{(2)}(\psi)=\psi$ almost everywhere on $\rr^{2}$.

Assume that, when $n:=m\geq 2$, $F_{\theta}^{(m)}(\psi^{(m)})=
\psi^{(m)}$ almost everywhere on $\rr^{m}$, where $\psi^{(m)}\in
H(\rr^{m})$.
Then, we aim to show that $F_{\theta}^{(m+1)}(\psi^{(m+1)})=\psi^{(m+1)}$
almost everywhere on $\rr^{m+1}$. From \eqref{Fzpsi}
and $F_{\theta}^{(m)}(\psi^{(m)})=\psi^{(m)}$ almost everywhere
on $\rr^{m}$, we deduce that, for any given $x_{m+1}
\in\rr$,
\begin{equation*}
F_{\theta}^{(m+1)}(\psi^{(m+1)})(\cdot,x_{m+1})
=\psi^{(m+1)}(\cdot,x_{m+1})\cf_{\theta
,m}(\psi^{(m+1)})(x_{m+1})\widehat{\cf}_{\theta
,m+1}(\psi^{(m+1)})(x_{m+1})
\end{equation*}
almost everywhere on $\rr^{m}$.  Then we consider
two cases on $R_{i_{m+1}}$ with $i_{m+1}\in\zz$. If
$R_{i_{m+1}}$ satisfies $\|\|\psi^{(m+1)}\|_
{\dot{E}_{
\vec{t}_{m}}^{\vec{\beta}_{m},\vec{s}
_{m}}(\rr)}
\one_{R_{i_{m+1}}}\|_{L^{t_{m+1}}(\rr)}=0$,
then, by \eqref{3.18x}, we obtain, for any $x_{m+1}\in R_{i_{m+1}}$,
$\widehat{\cf}_{\theta,m+1}(\psi^{(m+1)})(x_{m+1})=0$ and
hence $F_{\theta}^{(m+1)}(\psi^{(m+1)})=0$. Furthermore,
from Proposition \ref{equ0}, it follows that, for
almost every $x_{m+1}\in R_{i_{m+1}}$,
$$\psi^{(m+1)}(x_{1},\ldots,x_{m},x_{m+1})=0$$
for almost every $(x_{1},\ldots,x_{m})\in \rr^{m}$.
Therefore, for almost every $x_{m+1}
\in R_{i_{m+1}}$, we conclude that
$F_{\theta}^{(m+1)}(\psi^{(m+1)})(x_{1},\ldots,x_{m+1})=\psi
^{(m+1)}(x_{1},\ldots,x_{m+1})$ for almost every $(x_{1},\ldots,x_{m})\in \rr^{m}$.
On the other hand, if
$R_{i_{m+1}}$ satisfies $\|\|\psi^{(m+1)}\|_{\dot{E}
_{\vec{t}_{m}}^{\vec{\beta}_{m},\vec{s}_{m}}
(\rr^{m})} \one_{R_{i_{m+1}}}\|_{L^{t_{m+1}}(\rr)}
\neq 0$, then, for any $x_{m+1}\in R_{i_{m+1}}$,
$$\widehat{\cf}_{\theta,m+1}(\psi^{(m+1)})(x_{m+1})=1.$$
In this case, we consider the following two cases on
$x_{m+1}\in R_{i_{m+1}}$. If $x_{m+1}\in R_{i_{m+1}}$
satisfies $\|\psi^{(m+1)}(\cdot,x_{m+1})\|_
{\dot{E}_{\vec{t}_{m}}
^{\vec{\beta}_{m},\vec{s}_{m}}(\rr^{m})}=0$, then,
applying Proposition \ref{equ0}, we find that
$\psi^{(m+1)}(x_{1},\ldots,x_{m+1})=0$ almost every $(x_{1},\ldots,x_{m})\in\rr^{m}$.
Moreover, using \eqref{Fzk}, we have $\cf_{\theta,m}(\psi^{(m+1)})(x_{m+1})
=0$ and hence $$F_{\theta}^{(m+1)}(\psi^{(m+1)})(x_{1},\ldots,x_{m+1})=\psi^{(m+1)}(x_{1},\ldots,x_{m+1})$$
for almost every $(x_{1},\ldots,x_{m})\in\rr^{m}$. If $x_{m+1}\in R_{i_
{m+1}}$ satisfies $\|\psi^{(m+1)}(\cdot,x_{m+1})
\|_{\dot{E}_{\vec{t}_{m}}^{\vec{\beta}_{m},\vec{s}_
{m}}(\rr^{m})}\neq 0$, then, from \eqref{Fzk} and \eqref{3.18x},
it follows that
$$\cf_{\theta,m}(\psi^{(m+1)})(x_{m+1})=\widehat{\cf}_{\theta,m+1}(\psi^{(m+1)}
)(x_{m+1})=1.$$
Thus,
$$F_{\theta}^{(m+1)}(\psi^{(m+1)})=\psi^{(m+1)},$$
which further implies that, for any $x_{m+1}\in
R_{i_{m+1}}$ satisfies $\|\|\psi^{(m+1)}\|_{\dot{E}
_{\vec{t}_{m}}^{\vec{\beta}_{m},\vec{s}_{m}}
(\rr^{m})} \one_{R_{i_{m+1}}}\|_{L^{t_{m+1}}(\rr)}
\neq 0$,
$F_{\theta}^{(m+1)}(\psi^{(m+1)})(x_{1},\ldots,x_{m+1})=\psi^{(m+1)}(x_{1},\ldots,x_{m+1})$
for almost every
$(x_{1},\ldots,x_{m})\in\rr^{m}$.
Therefore, for almost every $x_{m+1}\in\rr$,
$$F_{\theta}^{(m+1)}(\psi^{(m+1)})(x_{1},\ldots,x_{m+1})=\psi^{(m+1)}(x_{1},\ldots,x_{m+1})$$
for almost every $(x_{1},\ldots,x_{m})\in\rr^{m}$. This, combined with the
Fubini theorem, further implies that
$$F_{\theta}^{(m+1)}(\psi^{(m+1)})=\psi^{(m+1)}$$
almost everywhere on
$\rr^{m+1}.$ This finishes the proof of Proposition \ref{ae=0}.
\end{proof}
\begin{lemma}\label{piped}
Let $n\in\nn$, $n>1$, and
$\psi:=\sum_{i=1}^{N}a_{i}\one_{E_{i}}\in H(\rn)$,
where $N\in\nn$ and, for any $i\in\{1,\ldots,N \}$, $a_{i}
\in\cc$ with $a_{i}\neq 0$ and $E_{i}:=E_{i,1}\times\cdots
\times E_{i,n}$ with $E_{i,k}\subset I_{m}$ for any $k
\in\{1,\ldots,n \}$, some $m\in\nn$, and $I_{m}$
as in Lemma \ref{Am}. Then, for any given $k\in\{1,
\ldots,n-1 \}$, there exist an $M_{k}\in\nn$, a sequence
$\{U_{k,l_{k}}\}_
{l_{k}=1}^{M_{k}}\subset I_{m}^{n-k}$ of
pairwise disjoint parallelepipeds, depending only
on $\{E_{i,k+1}\times\cdots\times E_{i,n} \}_{i=1}^{N}$,
and a sequence $\{a_{k,l_{k}}\}_{l_{k}=1}^{M_{k}}
\subset L^{\infty}(\rr^{k})$ such that, for any $(x_{1}
,\ldots,x_{n})\in\rn$,
\begin{equation}\label{psi-111}
\psi(x_{1},\ldots,x_{n})=\sum_{l_{k}=1}^{M_{k}}a_
{k,l_{k}}(x_{1},\ldots,x_{k})\one_{U_{k,l_{k}}}
(x_{k+1},\ldots,x_{n}).
\end{equation}
Moreover, for any $k\in\{2,\ldots,n-1 \}$ and $l_{k-1}
\in\{1,\ldots,M_{k-1} \}$, there exists
an $l_{k}\in\{1,\ldots,M_{k} \}$ such that, for any $
(x_{k},\ldots,x_{n})\in U_{k-1,l_{k-1}}$,  $(x_{k+1},
\ldots,x_{n})\in U_{k,l_{k}}$.
\end{lemma}
\begin{proof}
Let all the symbols be as in the present lemma.  For any
$k\in \{1,\ldots,n-1 \}$ and $i\in \{1,\ldots, N \}$,
let
\begin{equation*}
\EE_{i,k}:=E_{i,k+1}\times \cdots \times E_{i,n}.
\end{equation*}
Then we obtain the following different partitions on $\rr
^{n-k}$:
\begin{equation*}
\cp_{1,k}:=\left\{ \EE_{1,k}, \EE_{1,k}^{\com}\right\},\
\ldots,\
\cp_{N,k}:=\left\{
\EE_{N,k},\EE_{N,k}^{\com}
\right\}.
\end{equation*}
Let
\begin{equation}\label{Vk}
\vv_{k}:=\left\{V_{k}:=\bigcap_{i=1}^{N}V_{k}^{(i)}:\
V_{k}^{(i)}\in \cp_{i,k},\ i\in\{1,\ldots,N \} \right\}.
\end{equation}
By this, we find that, for any $k\in\{1,\ldots,n-1 \}$,
$\#(\vv_{k})=2^{N}$, $\{V_{k} \}_{V_{k}\in \vv_{k}}$
are pairwise disjoint, and $\bigcup_{V_{k}\in\vv_{k}}V_{k}=
\rr^{n-k}$. Now, for any $k\in\{1,\ldots,n-1 \}$, we consider the following two cases on $(x_{1}
,\ldots,x_{k})\in \rr^{k}$. If $(x_{1},\ldots,x_{k})\in
E_{i,1}\times\cdots\times E_{i,k}$ for some $i\in\{1,\ldots,
N \}$, then let
$$I:=\left\{i\in\{1,\ldots, N \}:\ (x_{1},
\ldots,x_{k})\in E_{i,1}\times\cdots\times E_{i,k}\right\}.$$
Next, we consider the following two cases on $V_{k}\in\vv_{k}$
with $V_{k}\neq \emptyset$. If there exists an $i\in I$
such that $V_{k}\subset \EE_{i,k}$, then $\psi(x)=a_{i}$
for any $(x_{k+1},\ldots,x_{n})\in V_{k}$. If $V_{k}
\subset \bigcap_{i\in I}\EE_{i,k}^{\com}$, then $\psi
(x)=0$
for any $(x_{k+1},\ldots,x_{n})\in V_{k}$. If $(x_{1}
,\ldots,x_{k})\notin E_{i,1}\times\cdots\times E_{i,k}$
for any $i\in\{1,\ldots, N \}$, then $\psi(x)=0$ for
any $(x_{k+1},\ldots,x_{n})\in \rr^{n-k}$. Thus, we
find that, for any given $(x_{1},\ldots,x_{k})\in
\rr^{k}$, $\psi$ is a fixed constant on
$V_{k}\in\vv_{k}$ with $V_{k}\neq \emptyset$, denoted
by $b_{V_{k}}(x_{1},\ldots,x_{k})$. Using this,
we easily obtain
\begin{equation}\label{psi11}
\psi(x_{1},\ldots,x_{n})=\sum_{\gfz{V_{k}\in
\vv_{k}}{V_{k}\neq\emptyset}}b_{V_{k}}
(x_{1},\ldots,x_{k})\one_{V_{k}}(x_{k+1},
\ldots,x_{n}).
\end{equation}
Note that, for any $x:=(x_{1},\ldots,x_{n})\in\rn$ satisfying $(x_{k+1},
\ldots,x_{n})\in \bigcap_{i=1}^{N}\EE_{i,k}^{\com}$,
we find that, for any $i\in\{1,\ldots,N \}$, $x
\notin E_{i}$ and hence $\psi(x)=0$. For any given
$k\in\{1,\ldots,n-1 \}$, let
\begin{equation}\label{Vk'}
\vv_{k}':=\left\{V_{k}\in\vv_{k}:\ V_{k}\neq
\emptyset,\ V_{k}\neq \bigcap_{i=1}^{N}
\EE_{i,k}^{\com} \right\}.
\end{equation}
Using this, we find that \eqref{psi11} becomes
\begin{equation}\label{psi22}
\psi(x_{1},\ldots,x_{n})=\sum_{V_{k}\in \vv_{k}'}
b_{V_{k}}(x_{1},\ldots,x_{k})\one_{V_{k}}(x_
{k+1},\ldots,x_{n}),
\end{equation}
where, for any $V_{k}\in \vv_{k}'$, there exists
an
$\EE_{i,k}$ for some $i\in\{1,\ldots,N \}$ such
that $V_{k}\subset \EE_{i,k}$. Thus, for any $V_{k}
\in \vv_{k}'$, $V_{k}\subset I_{m}^{n-k}$.
Now, we prove that, for any $V_{k}\in \vv_{k}'$,
there exists a sequence of pairwise disjoint
parallelepipeds such that their union is
equal to $V_{k}$. To this end,
for any $i\in\{1,\ldots,N \}$ and $k\in\{1,
\ldots,n-1 \}$, let
\begin{align*}
F_{i,k}:=&\,\bigg\{\widetilde{E}_{i,k+1}\times
\cdots \times \widetilde{E}_{i,n}:\
\quad\widetilde{E}_{i,l}\in\{E_{i,l},E_{i,l}^
{\com} \},\ l\in\{k+1,\ldots,n \},\\ &
\quad(\widetilde{E}_{i,k+1}\times\cdots\times
\widetilde{E}_{i,n})\neq (E_{i,k+1}\times
\cdots\times E_{i,n}) \bigg\}.
\end{align*}
Then we find that $\#(F_{i,k})=2^{(n-k)}-1$.
From this, we deduce that, for any $i\in\{1,
\ldots,N \}$,
\begin{equation*}
\EE_{i,k}^{\com}=\bigcup_{F\in F_{i,k}} F,
\end{equation*}
where, for any $F\in F_{i,k}$, $F$ is a
parallelepiped.
By this, we obtain the finer partition on
$\rr^{n-k}$ as follows: for any $i\in\{1,\ldots,N \}$,
\begin{equation*}
\cp_{i,k}':=\{\EE_{i,k}\}\cup F_{i,k}.
\end{equation*}
Let
\begin{equation*}
\uu_{k}:=\left\{U_{k}:=\bigcap_{i=1}^{N}
U_{k}^{(i)}:\ U_{k}^{(i)}\in \cp_{i,k}',\ i
\in\{1,\ldots, N \} \right\}.
\end{equation*}
From this, we deduce that, for any $U_{k}\in\uu_{k}$,
$U_{k}$ is of the following form:
\begin{equation*}
U_{k}=\bigcap_{i=1}^{N}\widetilde{E}_{i,k+1}
\times \cdots\times\bigcap_{i=1}^{N}
\widetilde{E}_{i,n},
\end{equation*}
where, for any $i\in\{1,\ldots,N \}$ and $l\in
\{k+1,\ldots,n \}$, $\widetilde{E}_{i,l}\in
\{E_{i,l},E_{i,l}^{\com} \}$. Thus, $U_{k}$ is a
parallelepiped.
By this and \eqref{Vk'}, we conclude that, for any
$V_{k}\in \vv_{k}'$,
\begin{equation*}
V_{k}=\bigcup_{U_{k}\in \uu_{k},U_{k}\subset
V_{k}}U_{k}.
\end{equation*}
This, combined with \eqref{psi22}, further implies
that \eqref{psi-111} holds true. Moreover, for any
$k\in\{2,\ldots,n-1 \}$ and $U_{k-1}\in\uu_{k-1}$
satisfying $U_{k-1}\subset V_{k-1}$ for some $V_{k-1}
\in\vv_{k-1}'$, let
$$U_{k-1}:=\bigcap_{i=1}^{N}\widetilde{E}_{i,k}
\times \cdots\times\bigcap_{i=1}^{N}\widetilde{E}_{i,n}$$
for some $\widetilde{E}_{i,l}\in\{E_{i,l},E_{i,l}^{\com}
\}$ with $l\in\{k,\ldots,n \}$. By this, let
\begin{equation*}
U_{k}:=\bigcap_{i=1}^{N}\widetilde{E}_{i,k+1}
\times \cdots\times\bigcap_{i=1}^{N}\widetilde{E}_{i,n}
=\bigcap_{i=1}^{N}U_{k}^{(i)},
\end{equation*}
where, for any $i\in\{1,\ldots,N \}$, $U_{k}^{(i)}=
\EE_{i,k}$ or $U_{k}^{(i)}\in F_{i,k}\subset \EE_{i,k}
^{\com}$. Thus, $U_{k}\subset V_{k}$ for some
$V_{k}\in\vv_{k}$.
Obviously, for any $(x_{k},\ldots,x_{n})\in U_{k-1}$,
$(x_{k+1},\ldots,x_{n})\in U_{k}$. Then we show that
$V_{k}\in\vv_{k}'$. Indeed, for any $(x_{k+1},\ldots,
x_{n})\in U_{k}$, there exists an $x_{k}\in\bigcap_{i=1
}^{N}\widetilde{E}_{i,k}$ such that
$$(x_{k},x_{k+1},\ldots,x_{n})\in U_{k-1}\subset V_
{k-1}\in\vv_{k-1}'$$ and hence $(x_{k},x_{k+1},\ldots,
x_{n})\in \EE_{i_{0},k-1}$ for some $i_{0}\in\{1,
\ldots,N \}$. By this, we conclude that
$$(x_{k+1},\ldots,x_{n})\in \EE_{i_{0},k}.$$
Thus,
$$V_{k}\neq \bigcap_{i=1}^{N}\EE_{i,k}^{\com},$$
which further
implies that $V_{k}\in \vv_{k}'$.
This finishes the proof of Lemma \ref{piped}.
\end{proof}
\begin{proposition}\label{simple-H}
Let $n\in\nn$, $n>1$, $\psi\in H(\rn)$, and $F_{z}^{(n)}(\psi)$ be as in
\eqref{Fzpsi}. Then $F_{z}^{(n)}(\psi)\in H(\rn)$.
\end{proposition}
\begin{proof}
Let all the symbols be as in the present proposition.
Let
$\psi=\sum_{i=1}^{N}a_{i}\one_{E_{i}}$ be as in Lemma
\ref{piped} with the same symbols as there. By Lemma \ref{piped}, we find that, for
any given $k\in\{1,\ldots,n-1 \}$, there exists a
sequence of pairwise disjoint parallelepipeds, $\{U_{k,
j_{k}}\}_{j_{k}=1}^{M_{k}}\subset I_{m}^{n-k}$,
$M_{k}\in\nn$, and a sequence $\{a_{k,j_{k}}\}_{j
_{k}=1}^{M_{k}}\subset L^{\infty}(\rr^{k})$ such
that $\psi$ is as in \eqref{psi-111}. Using this, we
conclude that, for any $(x_{k+1},\ldots,x_{n})\in
U_{k,j_{k}}$ with some $j_{k}\in\{1,\ldots,M_{k} \}$,
\begin{equation}\label{3.31x}
\|\psi(\cdot,x_{k+1},\ldots,x_{n})\|_{\dot{E}^{
\vec{\beta}_{k},\vec{s}_{k}}_{\vec{t}_{k}}
(\rr^{k})}=\|a_{k,j_{k}}\|_{\dot{E}^{\vec{\beta}_{k},
\vec{s}_{k}}_{\vec{t}_{k}}(\rr^{k})}=:b_{k,
j_{k}}\in[0,\infty).
\end{equation}
Thus, for any $(x_{k+1},
\ldots,x_{n})\in \rr^{n-k}$,
\begin{equation*}
\left\|\psi(\cdot,x_{k+1},\ldots,x_{n})\right\|_
{\dot{E}^{\vec{\beta}_{k},\vec{s}_{k}}_{\vec{t}
_{k}}(\rr^{k})}=\sum_{j_{k}=1}^{M_{k}}b
_{k,j_{k}}\one_{U_{k,j_{k}}}(x_{k+1},\ldots,
x_{n}).
\end{equation*}
From this, \eqref{Fzk}, and \eqref{Pik-1}, we easily
deduce that, for any $k\in\{1,\ldots,n-1 \}$ and $z
\in\cc$,
\begin{equation}\label{F2}
\cf_{z,k}(\psi)(x_{k+1},\ldots,x_{n})=\sum_{j_{k}
=1}^{M_{k}}c_{k,j_{k}}(z)\one_{U_{k,j_{k}}}
(x_{k+1},\ldots,x_{n})
\end{equation}
and, for any $k\in\{1,\ldots,n-1 \}$ and $i_{k}\in\zz$,
\begin{equation}\label{Pik-S}
P_{i_{k}}(\psi)(x_{k},\ldots,x_{n})=\sum_{j_{
k-1}=1}^{M_{k-1}}b_{k-1,j_{k-1}}\one_{(
R_{i_{k}}\times \rr^{n-k})\cap U_{k-1,j_{k-1
}}}(x_{k},\ldots,x_{n})
\end{equation}
and
$$P_{i_{n}}(\psi)(x_{n})=\sum_{j_{n-1}=1}^{M_{n-1}}
b_{n-1,j_{n-1}}\one_{R_{i_{n}}\cap U_{n-1,j_{n-1}}}
(x_{n}),$$
where, for any $k\in\{1,\ldots,n-1\}$ and $z\in\cc$,
$c_{k,j_{k}}(z):=b_{k,j_{k}}^{\widetilde{\mu}_{k+1}
(z)-\widetilde{\lambda}_{k}(z)}$ if $b_{k,j_{k}}
\neq 0$, and $c_{k,j_{k}}(z):=0$ if $b_{k,j_{k}}= 0$.
Here and thereafter, for any $k\in\{1,\ldots,n \}$ and
$z\in\cc$, $\widetilde{\lambda}_{k}(z)$ and
$\widetilde{\mu}_{k}(z)$ are, respectively, as in \eqref{mulam1} and \eqref{mulam}.
In particular, let $M_{0}:=N$ and, for any $j_{0}\in
\{1,\ldots,N \}$, let $b_{0,j_{0}}:=a_{j_{0}}$ and
$U_{0,j_{0}}:=E_{j_{0}}$. By \eqref{Pik-S}, we conclude
that, for any $k\in\{1,\ldots,n\}$ and $i_{k}\in\zz
\cap \{(-m,m]^{\com}\}$, $P_{i_{k}}(\psi)=0$ and, for
any $k\in\{1,\ldots,n \}$ and $i_{k}\in\zz\cap (-m,m]$,
$P_{i_{k}}(\psi)\in H(\rr^{n-k+1})$. From this and
Lemma \ref{piped} with $\psi$ replaced by $P_{i_{k}}
(\psi)$ for any given $k\in\{1,\ldots,n-1 \}$ and
$i_{k}\in\zz\cap (-m,m]$, it follows that
there exists a sequence of pairwise disjoint
parallelepipeds, $\{V_{k,l_{k}}\}_{l_{k}=1}^{L_{k}}
\subset I_{m}^{n-k}$, independent of $i_{k}$, $L_{k}
\in\nn$, and a sequence $\{a_{k,l_{k}}^{(i_{k})}\}
_{l_{k}=1}^{L_{k}}\subset L^{\infty}(\rr)$ such that,
for any $x_{k},\ldots,x_{n}\in\rr$,
\begin{equation*}
P_{i_{k}}(\psi)(x_{k},\ldots,x_{n})=\sum_{l_{k}=1}
^{L_{k}}a_{k,l_{k}}^{(i_{k})}(x_{k})\one_{V_{k,l_{k}}}
(x_{k+1},\ldots,x_{n}).
\end{equation*}
By this, we obtain, for any given $k\in \{1,\ldots,n-1\}$
and $i_{k}\in\zz\cap (-m,m]$, and for any
$(x_{k+1},\ldots,x_{n})\in V_{k,l_{k}}$ with $
l_{k}\in\{1,\ldots,L_{k} \}$,
\begin{equation}\label{3.32x}
\|P_{i_{k}}(\psi)(\cdot,x_{k+1},
\ldots,x_{n})\|_{L^{t_{k}}(\rr)}=\|a_{k,l_{k}}^{(i_{k})}
\|_{L^{t_{k}}(\rr)}=:d_{k,l_{k}}^{(i_{k})}\in[0,\infty).
\end{equation}
Thus, for any given $k\in \{1,\ldots,n-1\}$
and $i_{k}\in\zz\cap (-m,m]$, and for any $(x_{k+1},
\ldots,x_{n})\in\rr^{n-k}$,
\begin{equation*}
\|P_{i_{k}}(\psi)(\cdot,x_{k+1},\ldots,x_{n})\|_{L
^{t_{k}}(\rr)}=\sum_{l_{k}=1}^{L_{k}}d_{k,l_{k}}
^{(i_{k})}\one_{V_{k,l_{k}}}(x_{k+1},\ldots,x_{n}).
\end{equation*}
For any given $i_{n}\in\zz\cap (-m,m]$, let
\begin{equation}\label{3.32y}
d_{n}^{(i_{n})}:
=\|P_{i_{n}}(\psi)\|_{L^{t_{n}}(\rr)}\in [0,\infty).
\end{equation}
Thus, from \eqref{Pik}, we deduce that, for any given $k\in\{1,
\ldots,n-1 \}$ and $i_{k}\in\zz\cap (-m,m]$, and for any $z\in\cc$ and $
(x_{k+1},\ldots,x_{n})\in\rr^{n-k}$,
\begin{equation*}
P_{i_{k},\widetilde{\lambda}_{k}(z)-\widetilde{\mu}_{k}
(z)}(\psi)(x_{k+1},\ldots,x_{n})=\sum_{l_{k}=1}^{
L_{k}}e_{k,l_{k}}^{(i_{k})}(z)\one_{V_{k,l_{k}}}(
x_{k+1},\ldots,x_{n})
\end{equation*}
and, for any given $i_{n}\in\zz\cap (-m,m]$ and for any $z\in\cc$,
$$P_{i_{n},\widetilde{\lambda}_{n}(z)-\widetilde{\mu}_{n}
(z)}(\psi)=e_{n}^{(i_{n})}(z),$$
where, for any $k\in\{1,\ldots,n-1 \}$, $e_{k,l_{k}}^{(i_
{k})}(z):=[d_{k,l_{k}}^{(i_{k})}]^{\widetilde{\lambda}
_{k}(z)-\widetilde{\mu}_{k}(z)}$ and $e_{n}^{(i_{n})}
(z):=[d_{n}^{(i_{n})}]^{\widetilde{\lambda}_{n}(z)-
\widetilde{\mu}_{n}(z)}$.  For any $k\in\{1,\ldots,n \}$
and $i_{k}\in\zz\cap \{(-m,m]^{\com}\}$, $	P_{i_{k},\widetilde
{\lambda}_{k}(z)-\widetilde{\mu}_{k}(z)}(\psi)=0$.
Using this and \eqref{Fhatzk}, we conclude that, for any given $
k\in\{1,\ldots,n-1 \}$ and $i_{k}\in\zz\cap (-m,m]$,
and for any $z\in\cc$ and $(x_{k},\ldots,x_{n})
\in\rr^{n-k+1}$,
\begin{align}\label{F3}
&\widehat{\cf}_{z,k}(\psi)(x_{k},\ldots,x_{n})\\
&\quad=\sum_{i_{k}\in\zz\cap(-m,m]}2^{-i_{k}\beta_{k}(z)+i_{k}
\beta_{k}\widetilde{\lambda}_{k}(z) }\sum_{l_{k}=1}^
{L_{k}}e_{k,l_{k}}^{(i_{k})}(z)\one_{V_{k,l_{k}}}(x_{k+1}
,\ldots,x_{n})\one_{R_{i_{k}}}(x_{k})\nonumber\\
&\quad=\sum_{i_{k}\in\zz\cap(-m,m]}\sum_{l_{k}=1}^{L_{k}}
f_{k,l_{k}}^{(i_{k})}(z)\one_{R_{i_{k}}}(x_{k})\one_{V_
{k,l_{k}}}(x_{k+1},\ldots,x_{n})\nonumber
\end{align}
and, for any $z\in\cc$ and $x_{n}\in\rr$,
$$\widehat{\cf}_{z,n}(\psi)(x_{n})=\sum_{i_{n}\in\zz\cap(-m,
m]}f_{n}^{(i_{n})}(z)\one_{R_{i_{n}}}(x_{n}),$$
where, for any $l_{k}\in \{1,\ldots,L_{k} \}$,
$$f_{k,l_{k}}^{(i_{k})}(z):=2^{-i_{k}\beta_{k}(z)+
i_{k}\beta_{k}\widetilde{\lambda}_{k}(z) }e_{k,l_{k}}
^{(i_{k})}(z)$$
and
$$f_
{n}^{(i_{n})}(z):=2^{-i_{n}\beta_{n}(z)+i_{n}\beta_{n}
\widetilde{\lambda}_{n}(z) }e_{n}^{(i_{n})}(z).$$
By the above arguments, \eqref{Fzpsi}, \eqref{F2},
and \eqref{F3}, we conclude that,
for any $z\in\cc$ and $x:=(x_{1},\ldots,x_{n})\in\rn$,
\begin{align}\label{3.36x}
F_{z}^{(n)}(\psi)(x)&=\left[\sum_{i=1}^{N}a_{i}\one_{E_{i}}(x)
\right]^{\widetilde{\mu}_{1}(z)}\left[\mathrm{sgn}\,(
\overline{\psi(x)})\right]\prod_{k=1}^{n-1}\left[\sum_
{j_{k}=1}^{M_{k}}c_{k,j_{k}}(z)\one_{U_{k,j_{k}}}(x_{
k+1},\ldots,x_{n})\right]\\
&\quad\times\prod_{k=1}^{n-1}\left[\sum_{i_{k}\in\zz\cap(
-m,m]}\sum_{l_{k}=1}^{L_{k}}
f_{k,l_{k}}^{(i_{k})}(z)\one_{R_{i_{k}}}(x_{k})\one_
{V_{k,l_{k}}}(x_{k+1},\ldots,x_{n})\right]\nonumber\\
&\quad\times \left[\sum_{i_{n}\in\zz\cap(-m,m]}f_{n}^
{(i_{n})}(z)\one_{R_{i_{n}}}(x_{n})\right].\nonumber
\end{align}
Now, we consider the following two cases on $x\in\rn$. If
$x:=(x_{1},\ldots,x_{n})\in (\prod_{k=1}^{n}E_{i,k})\cap (\rr\times U_{1,j_{1}})
\cap (\rr\times V_{1,l_{1}})\cap (\prod_{k=1}^{n}R_{i_{k}})$
for some $i\in\{1,\ldots, N \}$, $j_{1}\in\{1,\ldots,M_{1}\}$,
$l_{1}\in\{1,\ldots,L_{1} \}$, and $i_{k}\in\zz\cap (-m,m]$
with $k\in\{1,\ldots,n \}$, then,
from Lemma \ref{piped}, we deduce that, for any $k\in \{1,
\ldots,n-1\}$, there exist a $j_{k}\in \{1,\ldots,M_{k} \}$
and an $l_{k}\in\{1,\ldots,L_{k} \}$ such that
$$(x_{k+1},\ldots,x_{n})\in\left(\prod_{l=k+1}^{n}E_{i,l}\right)\cap
U_{k,j_{k}}\cap  V_{k,l_{k}}\cap\left(\prod_{l=k+1}^{n}R_{i_{k}}\right).$$
This, combined with both \eqref{F2} and \eqref{F3}, further
implies that, for any $k\in\{1,\ldots,n-1 \}$,
\begin{equation*}
\cf_{z,k}(\psi)(x_{k+1},\ldots,x_{n})= c_{k,j_{k}}(z),
\end{equation*}
\begin{equation*}
\widehat{\cf}_{z,k}(\psi)(x_{k},\ldots,x_{n})=f^{(i_{k})}_
{k,l_{k}}(z),
\end{equation*}
and $\widehat{\cf}_{z,n}(\psi)(x_{n})=f_{n}^{(i_{n})}(z)$. Moreover,
$\psi(x)=a_{i}$.
By this and \eqref{3.36x}, we conclude that
\begin{equation}\label{coff}
F_{z}^{(n)}(\psi)(x)=a_{i}^{\widetilde{\mu}_{1}(z)}
[\mathrm{sgn}\,(\overline{a_{i}})]\prod_{k=1}^{n-1}
c_{k,j_{k}}(z)\prod_{k=1}^{n-1}f_{k,l_{k}}^{(i_
{k})}(z) f_{n}^{(i_{n})}(z).
\end{equation}
If $x:=(x_{1},\ldots,x_{n})\notin (\prod_{k=1}^{n}E_{i,k})\cap (\rr\times U_{1,j_
{1}})\cap (\rr\times V_{1,l_{1}})\cap (\prod_{k=1}^{n}R_{
i_{k}})$ for any $i\in\{1,\ldots, N \}$, $j_{1}\in\{1,
\ldots,M_{1}\}$, $l_{1}\in\{1,\ldots,L_{1} \}$, and $i_{k}\in\zz\cap
(-m,m]$ with $k\in\{1,\ldots,n \}$, we then consider the following
three cases on $x$ in this case.
If $x\in \left(\bigcup_{i=1}^{N}E_{i}\right)^{\com}$, then, in this
case, we find that $\psi(x)=0$ and hence $F_{z}^{(n)}(\psi)(x)=0$.
If $x:=(x_{1},\ldots,x_{n})\in \rr\times (\bigcup_{j_{1}=1}^{M_{1}}U_{1,j_{1}})^
{\com}$, then, in this case, we have $\cf_{z,1}(\psi)(x_{2},
\ldots,x_{n})=0$ and hence $F_{z}^{(n)}(\psi)(x)=0$. If $x\in \rr
\times (\bigcup_{l_{1}=1}^{L_{1}}V_{1,l_{1}})^{\com} $, then, in
this case, we obtain $\widehat{\cf}_{z,1}(\psi)(x)=0$ and
hence $F_{z}^{(n)}(\psi)(x)=0$.
Therefore, for any $x\notin (\prod_{k=1}^{n}E_{i,k})\cap
(\rr\times U_{1,j_{1}})\cap (\rr\times V_{1,l_{1}})\cap (
\prod_{k=1}^{n}R_{i_{k}})$ for any $i\in\{1,\ldots, N \}$,
$j_{1}\in\{1,\ldots,M_{1}\}$, $l_{1}\in\{1,\ldots,L_{1} \}$,
and $i_{k}\in\zz\cap (-m,m]$ with $k\in\{1,\ldots,n \}$, $F_
{z}^{(n)}(\psi)(x)=0$. Thus, $F_{z}^{(n)}(\psi)\in H(\rn)$.
This finishes the proof of Proposition \ref{simple-H}.
\end{proof}
Recall that a function $f$ is said to be an \emph{entire
function} if $f$ is holomorphic in the whole complex plane.
\begin{lemma}\label{entire}
Let $m\in\nn$, $\phi\in H_{m}(\zz\times \rn)$, and $G_{z}^{(n)}(\phi)$ be as in \eqref{Gzphi}. Then, for any
given $j\in\zz\cap B(0,m)$ and $y\in A_{m}$, $G_{z}^{(n)}(\phi)
(j,y)$ is an entire function on $z\in\cc$. Moreover, there
exists a positive constant $C$ such that, for any $j\in
\zz\cap B(0,m)$, $y\in A_{m}$, and $z\in \{z\in\cc:\
\Re(z)\in [-1,2] \}$,
\begin{equation*}
\left|G_{z}^{(n)}(\phi)(j,y)\right|\leq C,
\end{equation*}
where $\Re(z)$ denotes the real part of a complex
number $z$.
\end{lemma}
\begin{proof}
Let all the symbols be as in the present lemma.
Since, for any $j\in\zz\cap B(0,m)$, $\phi(j,\cdot)$ is
a simple function and $\supp(\phi(j,\cdot))= A_{m}$,
we deduce that there
exist two positive constants $\epsilon$ and $M$
such that
\begin{equation}\label{pij}
|\phi(j,\cdot)|\in (\epsilon, M)
\end{equation}
on $A_{m}$ for any $j\in \zz\cap B(0,m)$. For any $y:=(
y_{1},\ldots,y_{n})\in\rr^{n}$, let $\varphi(y):=\|\{\phi
(j,y)\}_{j\in\zz}\|_{l^{1}}$. By \eqref{pij}, we find that
$(2m-1)\epsilon<\varphi<(2m-1)M$ on $A_{m}$ and hence,
for any $k\in \{0,\ldots,n-1 \}$ and $(y_{k+1},\ldots,
y_{n})\in I_{m}^{n-k}$,
\begin{equation}\label{epM}
\epsilon\ls\|\varphi(\cdot,y_{k+1},\ldots,y_{n})\|_{
\dot{E}^{-\vec{\alpha}_{k},\vec{p}_{k}'}_{\vec{q}_
{k}'}(\rr^{k})} \ls M.
\end{equation}
From this and \eqref{ggzk}, it follows that, for any given $k\in
\{0,\ldots,n-1 \}$ and $(y_{k+1},\ldots,y_{n})\in I_{m}^{n-k}$,
and for any $z\in\cc$,
\begin{equation*}
\cg_{z,k}(\varphi)(y_{k+1},\ldots,y_{n})=\left[\|\varphi
(\cdot,y_{k+1},\ldots,y_{n})\|_{\dot{E}^{-\vec{\alpha}_
{k},\vec{p}_{k}'}_{\vec{q}_{k}'}(\rr^{k})} \right]
^{\widetilde{\eta}_{k+1}(z)-\widetilde{\xi}_{k}(z)}.
\end{equation*}
Thus, using both \eqref{etaxi} and \eqref{etaxi1}, we conclude that, for any given
$k\in\{0,\ldots,n-1 \}$ and $(y_{k+1},\ldots,y_{n})\in I_{m}^{n-k} $,
$$\cg_{z,k}(\varphi)(y_{k+1},\ldots,y_{n})$$
is an entire function on $z\in\cc$.
Similarly, from \eqref{epM}, \eqref{Qim}, \eqref{etaxi},
and \eqref{etaxi1},
we deduce that, for any given $k\in\{1,\ldots,n-1 \}$,
$i_{k}\in\zz\cap (-m,m]$, and $(y_{k+1},\ldots,y_{n})\in
I_{m}^{n-k}$,
\begin{equation*}
Q_{i_{k},\widetilde{\xi}_{k}(z)-\widetilde{\eta}_{k}(z)}
(\varphi)(y_{k+1},\ldots,y_{n})=\left\|Q_{i_{k}}(\varphi)
(\cdot,y_{k+1},\ldots,y_{n})\right\|_{L^{q_{k}'}(\rr)}^
{\widetilde{\xi}_{k}(z)-\widetilde{\eta}_{k}(z)}	
\end{equation*}
and, for any given $i_{n}\in \zz\cap (-m,m]$, $Q_{i_{n},
\widetilde{\xi}_{n}(z)-\widetilde{\eta}_{n}(z)}(\varphi)
= \|Q_{i_{n}}(\varphi)\|_{L^{q_{n}'}(\rr)}^{\widetilde{\xi}_
{n}(z)-\widetilde{\eta}_{n}(z)}$ are entire functions on
$z\in\cc$. By \eqref{hatgg}, we find that, for any given $k\in\{1,
\ldots,n \}$ and $(y_{k},\ldots,y_{n})\in I_{m}^{n-k+1}$,
there exists an $i_{k}\in\zz\cap (-m,m]$ such that, for any
$k\in\{1,\ldots,n-1 \}$ and $z\in\cc$,
\begin{equation}\label{hatzkC}
\widehat{\cg}_{z,k}(\varphi)(y_{k},\ldots,y_{n}):=
2^{i_{k}\alpha_{k}(z)-i_{k}\alpha_{k}\widetilde{\xi}_{k}(z)}
Q_{i_{k},\widetilde{\xi}_{k}(z)-\widetilde{\eta}_{k}(z)}
(\varphi)(y_{k+1},\ldots,y_{n})
\end{equation}
and
$$\widehat{\cg}_{z,n}(\varphi)(y_{n}):=
2^{i_{n}\alpha_{n}(z)-i_{n}\alpha_{n}\widetilde{\xi}_{n}(z)}
Q_{i_{n},\widetilde{\xi}_{n}(z)-\widetilde{\eta}_{n}(z)}
(\varphi).$$
Thus, for any given $k\in\{1,\ldots,n \}$
and $(y_{k},\ldots,y_{n})\in I_{m}^{n-k+1}$, $\widehat{\cg}_
{z,k}(\varphi)(y_{k},\ldots,y_{n})$ is an entire function on $z\in\cc$
and hence, for any given $j\in \zz\cap B(0,m)$ and $y
\in A_{m}$, $G_{z}^{(n)}(\phi)(j,y)$ is an entire function
on $z\in\cc$. Moreover, applying \eqref{epM}, we obtain, for
any $k\in \{0,\ldots,n-1 \}$, $(y_{k},\ldots,y_{n})
\in I_{m}^{n-k+1}$, and $z\in\cc$,
\begin{equation}\label{gzkC}
|\cg_{z,k}(\varphi)(y_{k+1},\ldots,y_{n})|\ls
\max\left\{ \epsilon^{\Re(\widetilde{\eta}_{k+1}(z)-
\widetilde{\xi}_{k}(z))}, M^{\Re(\widetilde{\eta}_
{k+1}(z)-\widetilde{\xi}_{k}(z))}\right\}.
\end{equation}
Using this, \eqref{etaxi}, and \eqref{etaxi1}, we conclude that, for any given
$k\in\{0,\ldots,n-1 \}$, there
exists a positive constant $C_{(k)}$ such that, for any
$z$ satisfying $\Re(z)\in [-1,2]$ and for any $(y_{k+1},\ldots,
y_{n})\in I_{m}^{n-k}$,
\begin{equation}\label{ggCk}
|\cg_{z,k}(\varphi)(y_{k+1},\ldots,y_{n})|\leq C_{(k)}.
\end{equation}
Similarly to the estimation of \eqref{gzkC}, we conclude
that, for any $k\in\{1,\ldots,n-1 \}$, $i_{k}\in\zz\cap
(-m,m]$, $(y_{k+1},\ldots,y_{n})\in I_{m}^{n-k}$,
and $z$ satisfying $\Re(z)\in [-1,2]$,
\begin{equation*}
|Q_{i_{k},\widetilde{\xi}_{k}(z)-\widetilde{\eta}_
{k}(z)}(\varphi)(y_{k+1},\ldots,y_{n})|\ls
\max\left\{\epsilon^{\Re(\widetilde{\xi}_{k}(z)-
\widetilde{\eta}_{k}(z))}, M^{\Re(\widetilde{
\xi}_{k}(z)-\widetilde{\eta}_{k}(z))} \right\}
\end{equation*}
and, for any $i_{n}\in\zz\cap (-m,m]$,
$$|Q_{i_{n},\widetilde{\xi}_{n}(z)-\widetilde{\eta}
_{n}(z)}(\varphi)|\ls \max\left\{\epsilon^{\Re(\widetilde
{\xi}_{n}(z)-\widetilde{\eta}_{n}(z))}, M^{\Re(
\widetilde{\xi}_{n}(z)-\widetilde{\eta}_{n}(z))}\right\}.$$
From this, \eqref{etaxi}, and \eqref{etaxi1}, we deduce
that, for any given $k\in\{1,\ldots,n \}$, there exists a
positive constant $\widetilde{C}_{(k)}$ such that, for
any given $k\in\{1,\ldots,n-1 \}$ and for any
$i_{k}\in\zz\cap (-m,m]$,
any $z$ satisfying $\Re(z)\in [-1,2]$, and any $(y_{k+1},
\ldots,y_{n})\in I_{m}^{n-k}$,
\begin{equation*}
\left|Q_{i_{k},\widetilde{\xi}_{k}(z)-\widetilde{
\eta}_{k}(z)}(\varphi)(y_{k+1},\ldots,
y_{n})\right|\ls \widetilde{C}_{(k)},
\end{equation*}
and $|Q_{i_{n},\widetilde{\xi}_{n}(z)-\widetilde{\eta}
_{n}(z)}(\varphi)|\ls \widetilde{C}_{(n)}$ for any
$z\in \{z\in\cc:\ \Re(z)\in [-1,2] \}$ and $i_{n}\in\zz
\cap (-m,m]$. Using this and \eqref{hatzkC}, we find
that, for any $k\in \{1,\ldots,n \}$, there exists a
positive constant $C_{(k)}'$ such that, for any $(y_{k},
\ldots,y_{n})\in I_{m}^{n-k+1}$ and $z\in \{z\in\cc:\
\Re(z)\in [-1,2] \}$,
\begin{equation*}
\left|\widehat{\cg}_{z,k}(\varphi)(y_{k},\ldots,y_{n})\right|\leq C_{(k)}'.
\end{equation*}
From this, \eqref{ggCk}, \eqref{pij}, and \ref{Gzphi},
it follows that there exists a positive constant $C$ such
that, for any $j\in\zz\cap B(0,m)$, $y\in A_{m}$, and
$z\in \{z\in\cc:\ \Re(z)\in [-1,2] \}$,
\begin{equation*}
\left|G_{z}^{(n)}(\phi)(j,y)\right|\leq C.
\end{equation*}
This finishes the proof of Lemma \ref{entire}.
\end{proof}

Now, we prove Theorem \ref{threeL}.

\begin{proof}[Proof of Theorem \ref{threeL}]
Let all the symbols be as in the present theorem.
We consider the following two cases on $\theta$.
If $\theta=0$ or $\theta=1$, then, by \eqref{point-0}
and \eqref{point-1}, we obtain \eqref{000} holds
true in this case.
If $\theta\in (0,1)$, then, for any given $\psi\in H(\rn)$
and $\phi:=\{\phi(j,\cdot) \}_{j\in\zz}\in H_{m}(\zz\times \rn)$ with
some $m\in\nn$, we assume that
$$\|\psi\|_{\nE}=\|\phi\|_{\dmE}=1.$$
To show \eqref{000}, by Corollary \ref{sppAm}, we conclude
that it suffices to prove that
\begin{equation}\label{bth}
\left|\int_{A_{m}}\sum_{j\in\zz\cap B(0,m)}T
(\psi)(j,y)\phi(j,y)\,dy\right|
\leq M_{0}^{1-\theta}M_{1}^{\theta}.
\end{equation}
Now, we claim that, for any $h\in\rr$,
\begin{align}
&\left\|F_{ih}^{(n)}(\psi)\right\|_{\nEz}=\|\psi\|_{\nE}^{\widetilde
{\lambda}_{n}(0)},\label{Fpsi-1}\\
&\left\|F_{1+ih}^{(n)}(\psi)\right\|_{\nEo}=\|\psi\|_{\nE}^{
\widetilde{\lambda}_{n}(1)},\label{Fpsi-2}\\
&\left\|G_{ih}^{(n)}(\phi) \right\|_{(\dot{E}^{-\vec{\alpha}^{(0)},
[\vec{p}^{(0)}]' }_{[\vec{q}^{(0)}]'}
(\rr^{n}),
l^{1}) }=\|\phi\|_{(\dot{E}^{-\vec{\alpha},
\vec{p}'}
_{\vec{q}'}(\rr^{n}),l^{1}) }^{\widetilde{\xi}
_{n}(0)},\label{Gphi-1}
\end{align}
and
\begin{equation}\label{Gphi-2}
\left\|G_{1+ih}^{(n)}(\phi)\right\|_{(\dot{E}^{-\vec{\alpha}^{(1)},
[\vec{p}^{(1)}]'}_{[\vec{q}^{(1)}]'}(\rr^
{n}),l^{1})}=\|\phi\|_{(\dot{E}^{-\vec{\alpha},
\vec{p}'}_{\vec{q}'}(\rr^{n}),l^{1})}^{
\widetilde{\xi}_{n}(1)},
\end{equation}
where, for any $z\in\cc$, $F_{z}^{(n)}(\psi)$ and $G_{z}^{(n)}
(\phi)$ are, respectively, as in \eqref{Fzpsi} and
\eqref{Gzphi}, $\widetilde{\lambda}_{n}(z)$ and
$\widetilde{\xi}_{n}(z)$ are, respectively, as in
\eqref{mulam1} and \eqref{etaxi1}.
Indeed, to prove the above claim, it suffices to show \eqref{Fpsi-1}
and \eqref{Gphi-1} because the proofs of both \eqref{Fpsi-2} and \eqref{Gphi-2}
are similar and we omit the details.

To prove \eqref{Fpsi-1}, we perform induction on $n$.
If $n:=1$, then, using \eqref{3.19x}, we conclude that,
for any $h\in\rr$,
\begin{equation}\label{1c}
\left\|F_{ih}^{(1)}(\psi)\right\|_{\dot{E}^{\beta_{1}^{(0)},s_{1}^{(0)}
}_{t_{1}^{(0)}}(\rr)}=\||\psi|^{\widetilde{\mu}
_{1}(ih)}\widehat{\cf}_{ih,1}(\psi)\|_{\dot{E}^{\beta_{1}
^{(0)},s_{1}^{(0)} }_{t_{1}^{(0)}}(\rr)}.
\end{equation}
Here and thereafter, for any $i\in\{1,\ldots,n \}$,
$\widetilde{\mu}_{i}(z)$ and $\widetilde{\lambda}_{i}
(z)$ are, respectively, as in \eqref{mulam} and \eqref{mulam1}.
For any $s_{1}^{(0)}\in[1,\infty]$, we consider the
following two cases on $t_{1}^{(0)}$. If $t_{1}^{(0)}
\in[1,\infty)$, then, from $\theta\in(0,1)$,
$1/t_{1}=(1-\theta)/t^{(0)}_{1}+\theta/t_{1}^{(1)}$, and
$t_{1}^{(1)}\in [1,\infty]$, it
follows that $t_{1}\in[1,\infty)$. Thus, $\widetilde
{\mu}_{1}(0)=\mu_{1}(0)/\mu_{1}$. By this,
\eqref{Fhatzk}, and \eqref{Pik},
we conclude that, for any $i_{1}\in\zz$ and $h\in\rr$,
\begin{align*}
&\left\|F_{ih}^{(1)}(\psi)\one_{R_{i_{1}}}\right\|_{L^{t_{1}^{(0)}}(\rr)}\\
&\quad=
\left\||\psi|^{\widetilde{\mu}_{1}(ih)}
\left[2^{-i_{1}\beta_{1}(ih)+i_{1}\beta_{1}\widetilde{\lambda}
_{1}(ih)}P_{i_{1},\widetilde{\lambda}_{1}(ih)-
\widetilde{\mu}_{1}(ih)}(\psi)\right]\one_{R_{i_{1}}}
\right\|_{L^{t_{1}^{(0)}}(\rr)}\\
&\quad=\left\||\psi|^{\widetilde{\mu}_{1}(0)}
\left[2^{-i_{1}\beta_{1}(0)+i_{1}\beta_{1}\widetilde{\lambda}_
{1}(0)}P_{i_{1},\widetilde{\lambda}_{1}(0)-\widetilde
{\mu}_{1}(0)}(\psi)\right]\one_{R_{i_{1}}} \right\|_{L^{t_{1}
^{(0)}}(\rr)}\\
&\quad=2^{-i_{1}\beta_{1}(0)+i_{1}\beta_{1}\widetilde{\lambda}
_{1}(0)}P_{i_{1},\widetilde{\lambda}_{1}(0)-
\widetilde{\mu}_{1}(0)}(\psi)\left\||\psi|^{\widetilde
{\mu}_{1}(0)}\one_{R_{i_{1}}} \right\|_{L^{t_{1}^{(0)}}
(\rr)}\\
&\quad=2^{-i_{1}\beta_{1}(0)+i_{1}\beta_{1}\widetilde{\lambda}
_{1}(0)}P_{i_{1},\widetilde{\lambda}_{1}(0)-
\widetilde{\mu}_{1}(0)}(\psi)\left\||\psi|\one_{R_{i_{1}}}
\right\|_{L^{t_{1}}(\rr)}^{\widetilde{\mu}_{1}(0)}\\
&\quad=2^{-i_{1}\beta_{1}(0)+i_{1}\beta_{1}\widetilde
{\lambda}_{1}(0)}P_{i_{1},\widetilde{\lambda}
_{1}(0)}(\psi).
\end{align*}
If $t_{1}^{(0)}=\infty$, then, in this case, we consider the following
two cases on $t_{1}$. If $t_{1}=\infty$,
then, using \eqref{mulam}, we obtain $\widetilde{\mu}_{1}(0)\equiv1$.
By this, we conclude that, for any $i_{1}\in\zz$ and $h\in\rr$,
\begin{align*}
\left\|F_{ih}^{(1)}(\psi)\one_{R_{i_{1}}}\right\|_{L^{\infty}(\rr)}&=2^{
-i_{1}\beta_{1}(0)+i_{1}\beta_{1}\widetilde{\lambda}
_{1}(0)}P_{i_{1},\widetilde{\lambda}_{1}(0)-
\widetilde{\mu}_{1}(0)}(\psi)\left\||\psi|^{\widetilde{\mu}
_{1}(0)}\one_{R_{i_{1}}} \right\|_{L^{\infty}(\rr)}\\
&=2^{-i_{1}\beta_{1}(0)+i_{1}\beta_{1}\widetilde{\lambda}
_{1}(0)}P_{i_{1},\widetilde{\lambda}_{1}(0)}(\psi).
\end{align*}
If $t_{1}\in[1,\infty)$, then $\widetilde{\mu}_{1}(0)\equiv0$.
From this, we deduce that
\begin{align*}
\left\|F_{ih}^{(1)}(\psi)\one_{R_{i_{1}}}\right\|_{L^{\infty}(\rr)}&=2^{-
i_{1}\beta_{1}(0)+i_{1}\beta_{1}\widetilde{\lambda}_
{1}(0)}P_{i_{1},\widetilde{\lambda}_{1}(0)}(\psi)\left\||
\psi|^{0}\one_{R_{i_{1}}} \right\|_{L^{\infty}(\rr)}\\
&=2^{-i_{1}\beta_{1}(0)+i_{1}\beta_{1}\widetilde{\lambda}
_{1}(0)}P_{i_{1},\widetilde{\lambda}_{1}(0)}(\psi).
\end{align*}
Thus, for any $s_{1}^{(0)}$,  $t_{1}^{(0)}\in[1,\infty]$,
we have
\begin{equation*}
\left\|F_{ih}^{(1)}(\psi)\one_{R_{i_{1}}}\right\|_{L^{t_{1}^{(0)}}(\rr)}=2
^{-i_{1}\beta_{1}(0)+i_{1}\beta_{1}\widetilde{\lambda}_{1}
(0)}P_{i_{1},\widetilde{\lambda}_{1}(0)}(\psi).
\end{equation*}
Now, we consider the following two cases on $s_{1}^{(0)}$.
If $s_{1}^{(0)}\in[1,\infty)$, since $\theta\in(0,1)$,
then, from $1/s_{1}=1/s_{1}^{(0)}+1/s_{1}^{(1)}$ and $s_{1}^{(1)}\in [1,\infty]$,
it follows that $s_{1}\in [1,\infty)$. Thus, we have $\widetilde{
\lambda}_{1}(0)=\lambda_{1}(0)/\lambda$. Using this and \eqref{Pik},
we obtain, for any $h\in\rr$,
\begin{align*}
\left\|F_{ih}^{(1)}(\psi)\right\|_{\dot{E}^{\beta_{1}^{(0)},s_{1}^{(0)}
}_{t_{1}^{(0)}}(\rr)}&=\left\{\sum_{i_{1}\in\zz}
2^{i_{1}
s_{1}^{(0)}\beta_{1}^{(0)}}\left[2^{-i_{1}\beta_{1}
(0)+i_{1}\beta_{1}\widetilde{\lambda}_{1}(0)}P_{i_{1}
,\widetilde{\lambda}_{1}(0)}(\psi) \right]^{s_{1}^
{(0)}}  \right\}^{1/s_{1}^{(0)}}\\
&=\left[\sum_{i_{1}\in\zz}2^{i_{1}\beta_{1}s_{1}}\left\||\psi
|\one_{R_{i_{1}}}\right\|^{s_{1}}_{L^{t_{1}}(\rr)} \right]^
{1/s_{1}^{(0)}}
=\|\psi\|_{\dot{E}_{t_{1}}^{\beta_{1},s_{1}}(\rr)}^{\widetilde
{\lambda}_{1}(0)}.
\end{align*}
If $s_{1}^{(0)}=\infty$, then, in this case, we consider the following
two cases on $s_{1}$. If $s_{1}=\infty$,
then $\widetilde{\lambda}_{1}(0)\equiv1$. From this and \eqref{Pik}, we
deduce that, for any $h\in\rr$,
\begin{align*}
\left\|F_{ih}^{(1)}(\psi)\right\|_{\dot{E}^{\beta_{1}^{(0)},\infty}_{t_{1}
^{(0)}}(\rr)}=\sup_{i_{1}\in\zz}\left\{2^{i_{1}
\beta_{1}^{(0)}}2^{-i_{1}\beta_{1}(0)+i_{1}\beta_{1}
\widetilde{\lambda}_{1}(0)} P_{i_{1},\widetilde{
\lambda}_{1}(0)}(\psi)\right\}=\|\psi\|_{\dot{E}
_{t_{1}}^{\beta_{1},s_{1}}(\rr)}.
\end{align*}
If $s_{1}\in[1,\infty)$, then
$\widetilde{\lambda}_{1}(0)\equiv0$. Applying this and
$\|\psi\|_{\dot{E}_{t_{1}}^{\beta_{1},s_{1}}(\rr)}=1$, we
find that, for any $h\in\rr$,
\begin{equation*}
\left\|F_{ih}^{(1)}(\psi)\right\|_{\dot{E}^{\beta_{1}^{(0)},\infty}_{t_{1}
^{(0)}}(\rr)}=\sup_{i_{1}\in\zz}P_{i_{1},0}(\psi)
=1=\|\psi\|_{\dot{E}_{t_{1}}^{\beta_{1},s_{1}}(\rr)}^{
\widetilde{\lambda}_{1}(0)}.
\end{equation*}
Thus, \eqref{Fpsi-1} holds true for $n:=1$. Assume that
\eqref{Fpsi-1}  holds true for a fixed $n\in\nn$, namely,
for any $h\in\rr$,
\begin{equation}\label{Fnp}
\left\|F_{ih}^{(n)}(\psi^{(n)})\right\|_{\dot{E}_{\vec{t}_{n}^{(0)}}^{\vec{
\beta}_{n}^{(0)},\vec{s}_{n}^{(0)}}(\rr^{n})}=
\|\psi^{(n)}\|_{\dot{E}_{\vec{t}_{n}}^{\vec{\beta}_{n},
\vec{s}_{n}}(\rr^{n})}^{\widetilde{\lambda}
_{n}(0)},
\end{equation}
where $\psi^{(n)}(x_{1},\ldots,x_{n}):=\psi(x_{1},\ldots,x_{n})$,
$\vec{\beta}_{n}^{(0)}:=(\beta_{1}^{(0)},\ldots,\beta_{n}^
{(0)}),$ $\vec{s}_{n}^{(0)}:=(s_{1}^{(0)},\ldots,s_{n}^{(0)}),$
$$\vec{t}_{n}^{(0)}:=(t_{1}^{(0)},\ldots,t_{n}^{(0)}),\
\vec{\beta}_{n}:=(\beta_{1},\ldots,\beta_{n}),$$
$\vec{s}_{n}:=(s_{1},\ldots,s_{n}),$
and $\vec{t}_{n}:=(t_{1},\ldots,t_{n})$.
Then we prove \eqref{Fpsi-1} for the case $n+1$. From
\eqref{Fzpsi}, Definition \ref{mhz}, and \eqref{Fnp}, we
deduce that, for any $h\in\rr$,
\begin{align*}
\left\|F_{ih}^{(n+1)}(\psi^{(n+1)})\right\|_{\dot{E}_{\vec{t}_{n+1}^{(0)}}^{
\vec{\beta}_{n+1}^{(0)},\vec{s}_{n+1}^{(0)}}(\rr^
{n+1})}&=\left\|\left\|F_{ih}^{(n+1)}(\psi^{(n+1)})\right\|_{\dot{E}_{\vec
{t}_{n}^{(0)}}^{\vec{\beta}_{n}^{(0)},\vec{s}_{n
}^{(0)}}(\rr^{n})} \right\|_{\dot{K}_{t_{n+1}^{(
0)}}^{\beta_{n+1}^{(0)},s_{n+1}^{(0)}}(\rr)}\\
&=\left\| \left\|\psi^{(n+1)}\right\|_{\dot{E}_{\vec{t}_{n}}^{\vec{\beta}
_{n},\vec{s}_{n}}(\rr^{n})}^{\widetilde
{\lambda}_{n}(0)}\cf_{ih,n}(\psi^{(n+1)})
\widehat{\cf}_{ih,n+1}(\psi^{(n+1)}) \right\|_{\dot{K}_{t_
{n+1}^{(0)}}^{\beta_{n+1}^{(0)},s_{n+1}^{(0)}}
(\rr)}\\
&=\left\| \left\|\psi^{(n+1)}\right\|_{\dot{E}_{\vec{t}_{n}}^{\vec{\beta}
_{n},\vec{s}_{n}}(\rr^{n})}^{\widetilde{\lambda}_
{n}(0)}\cf_{0,n}(\psi^{(n+1)}) \widehat{\cf}_
{0,n+1}(\psi^{(n+1)}) \right\|_{\dot{K}_{t_{n+1}^{(0)}}^{\beta_
{n+1}^{(0)},s_{n+1}^{(0)}}(\rr)}\\
&=\left\| \left\|\psi^{(n+1)}\right\|_{\dot{E}_{\vec{t}_{n}}^{\vec{\beta}
_{n},\vec{s}_{n}}(\rr^{n})}^{\widetilde{\mu}_{n+1}(0)}
\widehat{\cf}_{0,n+1}(\psi^{(n+1)}) \right\|_{\dot{K}_
{t_{n+1}^{(0)}}^{\beta_{n+1}^{(0)},s_{n+1}^{(0)}}(\rr)}.
\end{align*}
By this, similarly to the estimation of \eqref{1c}, we
conclude that, for any $h\in\rr$,
\begin{equation*}
\left\|F_{ih}^{(n+1)}(\psi^{(n+1)})\right\|_{\dot{E}_{\vec{t}_{n+1}^{(0)}}^{\vec
{\beta}_{n+1}^{(0)},\vec{s}_{n+1}^{(0)}}(\rr^{n+1})}=
\left\|\psi^{(n+1)}\right\|_{\dot{E}_{\vec{t}_{n+1}}^{\vec{\beta}_{n+1},
\vec{s}_{n+1}}(\rr^{n+1})}^{\widetilde{\lambda}_
{n+1}(0)}.
\end{equation*}
Thus, \eqref{Fpsi-1} holds true.

Next, we show \eqref{Gphi-1} holds true. Indeed, by
\eqref{Gzphi}, we find that, for any $h\in\rr$
and $y:=(y_{1},\ldots,y_{n})\in \rn$,
\begin{align}\label{3.49x}
\left\|\left\{G_{ih}^{(n)}(\phi)(j,y)\right\}_{j\in\zz} \right\|_{l^{1} }&=
\sum_{j\in\zz}|\phi(j,y)|\cg_{ih,0}(\varphi)(y)
\prod_{k=1}^{n-1}\cg_{ih,k}(\varphi)(y_{k+1},\ldots,y_{n})\\
&\quad\times\prod_{k=1}^{n}\widehat{\cg}_{ih,k}(\varphi)(y_{k},
\ldots,y_{n})\nonumber\\
&=|\varphi(y)|^{\widetilde{\eta}_{1}(ih)}\prod_{k=1}^
{n-1}\cg_{ih,k}(\varphi)(y_{k+1},\ldots,y_{n})\prod_
{k=1}^{n}\widehat{\cg}_{ih,k}(\varphi)(y_{k},\ldots,y_{n}).\nonumber
\end{align}
For any $j\in\{0,1\}$, let $\vec{\beta}
^{(j)}:=-\vec{\alpha}^{(j)},$ $\vec{s}^{(j)}:=[\vec{p}^
{(j)}]',$ and $\vec{t}^{(j)}:=[\vec{q}^{(j)}]'$,
Using this, \eqref{mulam}, \eqref{mulam1}, \eqref{etaxi},
and \eqref{etaxi1}, we easily obtain, for any $k\in\{1,\ldots,n \}$,
$\widetilde{\mu}_{k}(z)=\widetilde
{\eta}_{k}(z)$ and $\widetilde{\lambda}_{k}(z)=\widetilde{\xi}
_{k}(z)$.
From this, \eqref{Fzpsi}, and \eqref{3.49x}, we deduce that
\begin{align*}
\left|F_{ih}^{(n)}(\varphi)\right|&=|\varphi(y)|
^{\widetilde{\mu}_{1}(ih)}\prod_{k=1}^
{n-1}\cf_{ih,k}(\varphi)(y_{k+1},\ldots,y_{n})\prod_
{k=1}^{n}\widehat{\cf}_{ih,k}(\varphi)(y_{k},\ldots,y_{n})\\
&=
\left\|\left\{G_{ih}^{(n)}(\phi)(j,\cdot)\right\}_{j\in\zz}\right\|_{l^{1}}.
\end{align*}
Thus, similarly to the estimation of \eqref{Fpsi-1}, we
conclude that \eqref{Gphi-1} holds true. Thus, the above claim holds
true.

For any $z\in\cc$, let
\begin{equation}\label{aux}
\Phi(z):=\int_{\rr^{n}}\sum_{j\in\zz} T\left[F_{z}^{(n)}(\psi)\right](j,y)G_{z}
^{(n)}(\phi)(j,y)\,dy.
\end{equation}
From \eqref{Fpsi-1}, \eqref{Gphi-1}, the H\"{o}lder inequality, Lemma \ref{mhr},
Proposition \ref{simple-H}, \eqref{point-0}, \eqref{Fpsi-1}, and \eqref{Gphi-1},
it follows that, for any $h\in\rr$,
\begin{align}\label{C111}
|\Phi(ih)|&\leq \int_{\rr^{n}}\sum_{j\in\zz}|T\left[F_{ih}^{(n)}(\psi)\right]
(j,y)G_{ih}^{(n)}(\phi)(j,y)|\,dy\\\nonumber
&\leq \left\|T\left[F_{ih}^{(n)}(\psi)\right]\right\|_{(\dot{E}^{\vec{\alpha}^{(0)},\vec{p}
^{(0)} }_{\vec{q}^{(0)}}(\rr^{n}),l^{\infty})}\left\|G_{ih}^{(n)}
(\phi)\right\|_{(\dot{E}^{-\vec{\alpha}^{(0)},[\vec{p}^{(0)}]' }
_{[\vec{q}^{(0)}]'}(\rr^{n}),l^{1})}\\\nonumber
&\leq M_{0}\left\|F_{ih}^{(n)}(\psi)\right\|_{\nEz}\left\|G_{ih}^{(n)}(\phi)\right\|_{(\dot{E}^{
-\vec{\alpha}^{(0)},[\vec{p}^{(0)}]' }_{[\vec{q}^{(0)}]'}
(\rr^{n}),l^{1})}=M_{0}\nonumber
\end{align}
and, similarly to this estimation, we have,
for any $h\in\rr$,
\begin{equation}\label{C222}
|\Phi(1+ih)|\leq M_{1}.
\end{equation}
By Proposition \ref{simple-H}, \eqref{aux}, and the linearity of $T$, we find
that $\Phi(z)$ has the following expression: for any $z\in\cc$,
\begin{equation}\label{formP}
\Phi(z)=\sum_{u\in J}c_{u}(z)\int_{\rr^{n}}\sum_{j\in\zz}T(\one
_{E_{u}})(j,y)G_{z}^{(n)}(\phi)(j,y)\,dy,
\end{equation}
where $J$ is a finite set, the constant $c_{u}
(z)\neq 0$ for any $u\in J$ is as in \eqref{coff}, and $\{E_{u}\}_{u\in J}\subset
A_{k_{0}}$ for some $k_{0}\in\nn$ are parallelepipeds and
pairwise disjoint. From $\one_{E_{u}}\in H(\rn)$, \eqref{point-0},
Lemma \ref{entire}, $G_{z}^{(n)}(\phi)(j,\cdot)=0$ if $j\in\zz\cap
[B(0,m)]^{\com}$, the H\"{o}lder inequality, and Lemma \ref{mhr}, we deduce that,
for any $u\in J$ and $z\in \{z\in\cc:\ \Re(z)\in [-1,2] \}$,
\begin{align*}
&\left|\int_{\rr^{n}}\sum_{j\in\zz}T(\one_{E_{u}})(j,y)G_{z}^{(n)}
(\phi)(j,y)\,dy \right|\\\nonumber
&\quad\ls
\left\|T(\one_{E_{u}})\right\|_{\mEz}\left\|
\sum_{j\in\zz\cap B(0,m)}\one_{A_{m}}\right\|_{\dot{E}^{-\vec
{\alpha}^{(0)},[\vec{p}^{(0)}]' }_{[\vec{q}^{(0)}]'}
(\rr^{n})}\\\nonumber
&\quad\ls\|\one_{E_{u}}\|_{\nEz}<\infty.\nonumber
\end{align*}
By this, we find that each term of $\Phi(z)$ is bounded on $S:=
\{z\in\cc:\ \Re(z)\in [0,1] \}$. Then we claim that $\Phi(z)$ is
analytic on $S_{0}:=\{ z\in\cc:\ \Re(z)\in (0,1) \}$ and continuous
on $S$. Indeed, since $c(u)(z)\neq 0$, from \eqref{coff}, we deduce
that, for any $k\in\{1,\ldots,n-1 \}$ and $z\in\cc$,
$$c_{k,j_{k}}(z)=b_{k,j_{k}}^
{\widetilde{\mu}_{k+1}(z)-\widetilde{\lambda}_{k}(z)},$$
$$f_{k,l_{k}}^{(i_{k})}(z)=2^{-i_{k}\beta_{k}(z)+i_{k}\beta_{k}\widetilde{\lambda}_{k}(z) }\left[d_{k,l_{k}}^{(i_{k})}\right]^{\widetilde{\lambda}_{k}(z)-
\widetilde{\mu}_{k}(z)},$$
and
$$f_{n}^{(i_{n})}(z)=2^{-i_{n}\beta_{n}
(z)+i_{n}\beta_{n}\widetilde{\lambda}_{n}(z) }
\left[d_{n}^{(i_{n})}\right]^{\widetilde{\lambda}_{n}(z)-\widetilde{\mu}_{n}(z)},$$
where $b_{k,j_{k}}$,
$d_{k,l_{k}}^{(i_{k})}$, and $d_{n}^{(i_{n})}$ are as in
\eqref{3.31x}, \eqref{3.32x}, and \eqref{3.32y}. By this, we conclude that
$c_{u}(z)$ for any $u\in J$ is an entire function on $z\in\cc$
and hence it suffices to prove the analyticity of
$\int_{\rr^{n}}\sum_{j\in\zz}T(\one_{E_{u}})(j,y)G_{z}^{(n)}(\phi)(j,y)\,dy$
on $z\in S_{0}$. From Lemma \ref{entire}, it follows that,
for any given $j\in\zz\cap B(0,m)$ and $y\in \supp(G_{z}^{(n)}(\phi)(j,\cdot))
\subset A_{m}$, $G_{z}^{(n)}(\phi)(j,y)$ is an entire function on $z\in\cc$ and has
the following expansion in power series with center $z_{0}\in S_{0}$
(see, for instance, \cite[p. 207, Theorem 10.16]{Rudin}):
\begin{equation*}
G_{z}^{(n)}(\phi
)(j,y)=\sum_{k=0}^{\infty}a_{k}(j,y)(z-z_{0})^{k},
\end{equation*}
where
\begin{equation*}
a_{k}(j,y)=\f{1}{2\pi i}\int_{|z-z_{0}|=1}\f{G_{z}^{(n)}(\phi)(j,y)}
{(z-z_{0})^{k+1}}\,dz.
\end{equation*}
By this and Lemma \ref{entire}, we conclude that there exists a
positive constant $C$ such that, for any $k\in\zz_{+}$,
$j\in\zz\cap B(0,m)$, and $y\in A_{m}$,
\begin{equation}\label{akj}
|a_{k}(j,y)|\leq \f{C}{2\pi}\int_{|z-z_{0}|=1}\f{1}{|z-z_{0}|
^{k+1}}\,|dz|=C.
\end{equation}
Observe that, for any $j\in\zz\cap [B(0,m)]^{\com}$ and $y\in\rn$,
or for any $j\in\zz$ and $y\in A_{m}^{\com}$,
$a_{k}(j,y)=0$. Using this, \eqref{akj}, \eqref{point-0}, and
the H\"{o}lder inequality,
we conclude that, for any $z,\ z_{0}\in\cc$ satisfying that
$\Re(z_{0})\in(0,1)$ and $|z-z_{0}|<1$,
\begin{align*}
&\sum_{k=0}^{\infty}|z-z_{0}|^{k}\left|\int_{\rr^{n}}\sum_{j\in
\zz}a_{k}(j,y)T(\one_{E_{u}})(j,y)\,dy\right| \\
&\quad\leq
\sum_{k=0}^{\infty}|z-z_{0}|^{k}\int_{\rr^{n}}\sum_{j\in\zz}\left
|a_{k}(j,y)T(\one_{E_{u}})(j,y)\right|\,dy\\
&\quad\ls \sum_{k=0}^{\infty}|z-z_{0}|^{k}
\int_{A_{m}}\sum_{j\in\zz\cap B(0,m)}\left|T(\one_{E_{u}})(j,y)
\right|\,dy\\
&\quad \ls
\|\{T(\one_{E_{u}})(j,\cdot) \}_{j\in\zz}\|_{\dot{E}^{-\vec{
\alpha}^{(0)},[\vec{p}^{(0)}]' }_{[\vec{q}^{(0)}]'}
(\rr^{n})}\ls \|\one_{E_{u}}\|_{\nEz}<\infty.
\end{align*}
Thus, for any $z,\ z_{0}\in\cc$ satisfying that $\Re(z_{0})\in(0,1)$
and $|z-z_{0}|<1$,
\begin{equation}\label{unicov}
\int_{\rr^{n}}\sum_{j\in\zz}T(\one_{E_{u}})(j,y)G_{z}^{(n)}(\phi)(j,y)\,dy
=\sum_{k=0}^{\infty}(z-z_{0})^{k}\int_{\rr^{n}}\sum_{j\in\zz}a_{k}
(j,y)T(\one_{E_{u}})(j,y)\,dy
\end{equation}
is an absolutely and inner closed uniformly convergent series of
analytic functions on $\{z\in\cc:\ |z-z_{0}|<1 \}$, which,
combined with \cite[p.\,53,\ Theorem 5.2]{ss03} and the arbitrariness
of $z_{0}\in S_{0}$, \eqref{formP}, further implies that $\Phi(z)$
is analytic on $S_{0}$. Similarly, for any $z,\ z_{0}\in\cc$
satisfying that $\Re(z_{0})\in[0,1]$ and $|z-z_{0}|<1$,
\eqref{unicov} is also an absolutely and inner closed
uniformly convergent
series of continuous functions and hence $\Phi(z)$ is
continuous on $S$. Thus, the above last claim holds true.
By this last claim, \eqref{C111}, \eqref{C222}, the
boundedness of $\Phi(z)$ on $S$, Proposition \ref{ae=0},
and \cite[p.\,180,\ Lemma 1.4]{SW}, we conclude that
\begin{equation*}
|\Phi(\theta)|=\left|\int_{A_{m}}\sum_{j\in\zz\cap B(0,m)}
T(\psi)(j,y)\phi(j,y)\,dy \right|\leq M_{0}^{1-\theta}
M_{1}^{\theta},
\end{equation*}
which further implies that \eqref{bth} holds true. Thus, for any
$f\in H(\rn)$, \eqref{000} holds true.
This finishes the proof of Theorem \ref{threeL}.
\end{proof}

\section{Boundedness of Maximal Operators on Mixed-Norm Herz Spaces}\label{s4}

In this section, we obtain the boundedness of the Hardy--Littlewood
maximal operator on mixed-norm Herz spaces $\iihz(\rn)$; see Corollary
\ref{HL-2} below. As a corollary,
we establish the Fefferman--Stein
vector-valued maximal inequality on $\iihz(\rn)$.

To study the boundedness of $M_{n}$ on $\iihz(\rn)$, we need the
following conclusions on dense subsets of mixed-norm Herz spaces.
\begin{lemma}\label{Cc-E}
Let $\vec{p}$, $\vec{q}\in(0,\infty)^{n}$ and
$\vec{\alpha}\in\rn$. Then $C_{\mathrm{c}}(\rn)\cap \iihz(\rn)$ is dense
in $\iihz(\rn)$.
\end{lemma}
\begin{proof}
Let all the symbols be as in the present lemma. Without loss of
generality, we may assume that $f\in\iihz(\rn)$ is a non-negative measurable
function on $\rn$. Then, for any given $m\in\nn$, there exists
an increasing sequence of non-negative simple functions supported in $A_{m}$,
$\{f_{j} \}_{j\in\nn}$, which converges pointwise to
$f\one_{A_{m}}$ as $j\to\infty$ and $0\leq f_{j}\leq f\one_{A_{m}}$
for any $j\in\nn$. By this,
$|f_{j}-f\one_{A_{m}}|\leq 2|f|\in \iihz(\rn)$, and Lemma
\ref{donn}, we conclude that, for any given $\epsilon\in(0,\infty)$,
there exists a
simple function $g:=\sum_{k=1}^{N}\lambda_{k}\one_{E_{k}}\in
\{f_{j} \}_{j\in\nn}$ such that
\begin{equation*}
\|f\one_{A_{m}}-g\|_{\iihz(\rn)}<\epsilon,
\end{equation*}
where, for any $k\in\{1,\ldots,N \}$, $\lambda_{k}$ is a
positive constant and $E_{k}\subset A_{m}$ is a measurable set.
Moreover, for any $k\in\{1,\ldots,N \}$, via the
inner regularity of the Lebesgue measure, we
obtain a sequence of compact sets, $\{F_{k,j} \}_{j\in\nn}\subset
E_{k}$, such that $\one_{F_{k,j}}$ converges pointwise to $\one_
{E_{k}}$ as $j\to \infty$.
From this, $|\one_{F_{k,j}}-\one_{E_{k}}|\leq 2\one_{E_{k}}\ls
\one_{A_{m}}\in \iihz(\rn)$, and Lemma \ref{donn} again, we
deduce that there exists a
simple function $h:=\sum_{k=1}^{N} \lambda_{k}\one_{F_{k}}$
such that
\begin{equation*}
\|g-h\|_{\iihz(\rn)}\leq \epsilon,
\end{equation*}
where, for any $k\in\{1,\ldots,N \}$, $F_{k}\subset E_{k}$ is
a compact set, which, combined with the outer regularity of
the Lebesgue measure and Lemma \ref{donn}, further implies
that there exists a simple function
$u:=\sum_{k=1}^{N}\lambda_{k}\one_{U_{k}}$ such that
\begin{equation*}
\|u-h\|_{\iihz(\rn)}<\epsilon,
\end{equation*}
where, for any $k\in \{1,\ldots,N \}$, $U_{k}\supset F_{k}$
is a bounded open set. Then, applying the Urysohn lemma,
we obtain $f_{0}\in C_{\mathrm{c}}(\rn)$ satisfying $0\leq f_{0}-h\leq
u-h$, which, together with
Proposition \ref{E-lat}, further implies that
\begin{equation*}
\|f_{0}-h\|_{\iihz(\rn)}\leq \|u-h\|_{\iihz(\rn)}<\epsilon.
\end{equation*}
Furthermore,
$$\|f_{0}\|_{\iihz(\rn)}\ls \|f_{0}-h\|_{\iihz
(\rn)}+\|h\|_{\iihz(\rn)}<\infty$$
and hence $f_{0}\in C_{\mathrm{c}}
(\rn)\cap \iihz(\rn)$. By the above estimates, we conclude that
$$\|f_{0}-f\one_{A_{m}}\|_{\iihz(\rn)}\ls\epsilon.$$
Using Proposition \ref{ab-2},
we find that there exists an $N\in\nn$ such that, for any
$m>N$,
$$\|f-f\one_{A_{m}}\|_{\iihz(\rn)}<\epsilon.$$
From this and the above estimates, it follows that there exists
an $f_{0}\in C_{\mathrm{c}}(\rn)\cap\iihz(\rn)$ such that
$$\|f-f_{0}\|_{\iihz(\rn)}\ls \|f-f\one_{A_{m}}\|_{\iihz
(\rn)}+\|f\one_{A_{m}}-f_{0}\|_{\iihz(\rn)}\ls\epsilon.$$
This finishes the proof of Lemma \ref{Cc-E}.
\end{proof}
\begin{proposition}\label{E-dense}
Let $\vec{p}$, $\vec{q}\in(0,\infty)^{n}$
and $\vec{\alpha}\in\rn$. Then $H(\rn)$ is dense in
$\iihz(\rn)$.
\end{proposition}
\begin{proof}
Let all the symbols be as in the present proposition.
Let $f_{m}\in C_{\mathrm{c}}(\rn)$ be a non-negative measurable
function with $\supp(f_{m})\subset A_{m}$
for some $m\in\nn$. By this, we obtain $f_{m}\in C_{\mathrm{c}}
(\rn)\cap \iihz(\rn)$. Obviously, there exists a
sequence of $\{E_{j} \}_{j=1}^{N}$ satisfying that,
for any $j\in\{1,\ldots,N \}$, $E_{j}$ is as in \eqref{pppeds}
and $\{E_{j} \}_{j=1}^{N}$ are pairwise disjoint such that
\begin{equation*}
A_{m}=\bigcup_{j=1}^{N}E_{j}.
\end{equation*}
Then $\{E_{j} \}_{j=1}^{N}$ is called a \emph{partition} of $A_{m}$,
denoted by $P_{m}$.  Let
$$\lambda(P_{m}):=\max_{1\leq j\leq N}
\dist(E_{j}),$$
where
$$\dist(E_{j}):=\sup\{|x-y|:\ x,\
y\in E_{j} \}.$$
For any $j\in\{1,\ldots,N \}$, let
$$f_{m,j}:=\min\left\{f_{m}(x):\ x\in \overline{E}_{j}\right\}\geq 0.$$
Then, for any given partition $P_{m}$, let
\begin{equation*}
g(f_{m},P_{m}):=\sum_{j=1}^{N}f_{m,j}\one_{E_{j}}.
\end{equation*}
Thus, $g(f_{m},P_{m})\in H(\rn)$. By this and the uniformly continuity of $f_{m}$, we
conclude that, for any $\epsilon\in(0,\infty)$, there
exist a $\delta\in(0,\infty)$ and a partition $P_{m}$
satisfying $\lambda(P_{m})<\delta$ such that, for any
$x\in\rn$,
\begin{equation*}
|f_{m}(x)-g(f_{m},P_{m})(x)|<\epsilon.
\end{equation*}
Letting $\lambda(P_{m})\to 0^{+}$, we conclude that
$|f_{m}-g(f_{m},P_{m})|\to 0$ almost everywhere
on $\rn$. Observe that, for any partition $P_{m}$,
$0\leq g(f_{m},P_{m})\leq f_{m}$ and hence
$$|f_{m}-g(f_{m},
P_{m})|\leq 2|f_{m}|\in\iihz(\rn).$$
This, together with
Lemma \ref{donn}, further implies that
\begin{equation}\label{fgm}
\lim_{\lambda(P_{m})\to 0^{+}}\|f_{m}-g(f_{m},P_{m})\|_
{\iihz(\rn)}=0.
\end{equation}
Now, without loss of generality, we may assume that
$h\in\iihz(\rn)$ is a non-negative measurable function
and, for any $m\in\nn$, let $h_{m}:=h\one_{A_{m}}$.
By Proposition \ref{ab-2} and $h\in\iihz(\rn)$, we find
that, for any $\epsilon\in(0,\infty)$, there exists
an $N\in\nn$ such that, for any $m> N$,
$$\|h-h_{m}\|_{\iihz(\rn)}<\epsilon/3.$$
From Lemma \ref{Cc-E}, it follows that, for any given $m>N$,
there exists a non-negative function $f_{m}\in C_{\mathrm{c}}(\rn)
\cap \iihz(\rn)$ with
$\supp(f_{m})\subset A_{m}$ such that $\|h_{m}-f_{m}
\|_{\iihz(\rn)}<\epsilon/3$. Via this $f_{m}$ and
\eqref{fgm}, we conclude that there exists a function
$g_{m}\in H(\rn)$ such that $\|f_{m}-g_{m}\|_{\iihz(\rn)}
<\epsilon/3$. From above arguments, we deduce that,
for any $\epsilon\in(0,\infty)$, there exists an $N\in\nn$ such
that, for any $m>N$, we obtain a function $g_{m}
\in H(\rn)$ satisfying
\begin{align*}
\|h-g_{m}\|_{\iihz(\rn)}&\ls
\|h-h_{m}\|_{\iihz(\rn)}+
\|h_{m}-f_{m}\|_{\iihz(\rn)}+
\|f_{m}-g_{m}\|_{\iihz(\rn)}\\
&\ls \f{\epsilon}{3}+\f{\epsilon}{3}+\f{\epsilon}{3}\sim\epsilon.
\end{align*}
This finishes the proof of Proposition \ref{E-dense}.
\end{proof}
The following dual inequality of Stein type can be founded,
for instance, in \cite[p.\,111,\ Lemma 1]{FS}.
\begin{lemma}\label{stn}
For any $r\in (1,\infty)$, there exists a positive
constant $C$, depending only on $r$, such that, for any
$f\in \MM(\rn)$ and any non-negative measurable function $\varphi$ on
$\rn$,
\begin{equation}
\int_{\rn}[Mf(x)]^{r}\varphi(x)\,dx\leq
C\int_{\rn}|f(x)|^{r}M\varphi(x)\,dx.
\end{equation}
\end{lemma}
We have the following generalization of
\cite[p.\,419,\ Theorem 2]{bag},
which plays a key role in the present section.
\begin{theorem}\label{exine}
Let $n\in\nn$, $\vec{p}_{n}:=(p_{1},\ldots,p_{n})$, $\vec{q}_{n}:=(q_{1},\ldots,q_{n})\in(1,\infty)^{n}$, and
$\vec{\alpha}_{n}:=(\alpha_{1},\ldots,\alpha_{n})$ with $\alpha_{i}\in(-\f{1}{q_{i}},
1-\f{1}{q_{i}})$ for any $i\in\{1,\ldots,n\}$. Then
there exists a positive constant $C$
such that, for any $f\in \MM(\rn),$
\begin{equation}\label{ext}
\|M_{n}(f)\|_{\dot{E}_{\vec{q}_{n}}^{\vec{\alpha}_{n},\vec{p}
_{n}}(\rn)}\leq C \|f\|_{\dot{E}_{\vec{q}_{n}}^{\vec{
\alpha}_{n},\vec{p}_{n}}(\rn)},
\end{equation}
where $M_{n}$ is as in \eqref{Mk}.
\end{theorem}
\begin{proof}
Let all the symbols be as in the present theorem. We
perform induction on $n$.
If $n:=1$, then the desired
inequality becomes
\begin{equation*}
\|M_{1}(f)\|_{\ky(\rr)}\ls \|f\|_{\ky(\rr)},
\end{equation*}
which is obtained via the boundedness of the
Hardy--Littlewood maximal
operator on $\ky(\rr)$ (see,
for instance, \cite[p.\,488,\ Corollary 2.1]{ly96} or
\cite[p.\,131, Theorem 5.1.1 and Remark 5.1.3]{HZ}).
If $n:=2$,
then, we aim to show that, for any $\vec{p}_{2},$ $\vec{q}_{2}
\in(1,\infty)^{2}$ and $\alpha_{i}\in(-\f{1}{q_{i}},1-\f{1}
{q_{i}})$ with $i\in\{1,2 \}$,
\begin{equation}\label{m2-G}
\|M_{2}(f)\|_{\dot{E}_{\vec{q}_{2}}^{\vec{\alpha}_{2},\vec{p}_{2}}
(\rr^{2})}\ls \|f\|_{\dot{E}_{\vec{q}_{2}}^{\vec{\alpha}_{2},
\vec{p}_{2}}(\rr^{2})}.
\end{equation}
To this end, for any given $q_{1}\in[1,\infty]$ and $k_{1}\in\zz$,
and for any $x_{2}\in\rr$, let
$$g_{k_{1},q_{1}}(x_{2}):=\|f(\cdot,x_{2})\one_{R_{k_{1}}}
(\cdot)\|_{L^{q_{1}}(\rr)}$$
and
$$G_{k_{1},q_{1}}(x_{2}):=
\|M_{2}(f)(\cdot,x_{2})\one_{R_{k_{1}}}(\cdot)\|_{L^{q_{1}}
(\rr)}.$$
For any given $p_{1},$ $p_{2},$ $q_{2}\in
(1,\infty)$ and $\alpha_{2}\in(-\f{1}{q_{2}},1-\f{1}{q_{2}})$,
we consider the  following two cases on both $q_{1}$ and
$\alpha_{1}$.
If $q_{1}=\infty$ and $\alpha_{1}\in(0,1)$,
then, for any  $k_{1}\in\zz$ and $x_{2}\in\rr,$ and for almost every $x_{1}\in\rr$,
we have
\begin{equation*}
f(x_{1},x_{2})\one_{R_{k_{1}}}(x_{1})\leq
\|f(\cdot,x_{2})\one_{R_{k_{1}}}(\cdot)\|_{L^{\infty}(\rr)},
\end{equation*}
which further implies that
\begin{equation*}
M_{2}(f)(x_{1},x_{2})\one_{R_{k_{1}}}(x_{1})\ls
M(g_{k_{1},\infty})(x_{2}),
\end{equation*}
where $M$ is as in \eqref{Max}.
Thus, for any  $k_{1}\in\zz$ and $x_{2}\in\rr$,
\begin{equation}\label{kl-A}
G_{k_{1},\infty}(x_{2}):=\|M_{2}(f)(\cdot,x_{2})\one_{R_
{k_{1}}}(\cdot)\|_{L^{\infty}(\rr)}\ls
M(g_{k_{1},\infty})(x_{2}).
\end{equation}
By this, Definition \ref{chz}, Proposition \ref{E-lat},
and \cite[p.\,483,\ Corollary 4.5]{Ifs}, we conclude that
\begin{align*}
&\left\|\|M_{2}(f)\|_{\ky(\rr)} \right\|_{\ke(\rr)}\\
&\quad=
\left\|\left(\sum_{k_{1}\in\zz}2^{k_{1}p_{1}\alpha_{1}}
G_{k_{1},\infty}^{p_{1}}\right)^{\f{1}{p_{1}}}
\right\|_{\ke(\rr)}\\
&\quad\ls
\left\|\left\{\sum_{k_{1}\in\zz} [M(2^{k_{1}\alpha_{1}}
g_{k_{1},\infty})]^{p_{1}}\right\}^{1/p_{1}}
\right\|_{\ke(\rr)}\\
&\quad\ls
\left\|\left(\sum_{k_{1}\in\zz}2^{k_{1}p_{1}\alpha_{1}}
g_{k_{1},\infty}^{p_{1}}\right)^{1/p_{1}} \right\|
_{\ke(\rr)}\sim\left\|\|f\|_{\ky(\rr)} \right\|_{\ke(\rr)}.
\end{align*}
If $q_{1}\in (1,\rho_{2})$ and $\alpha_{1}\in (-\f{1}
{q_{1}},1-\f{1}{q_{1}})$, where $\rho_{2}:=\min\{p_{1},
p_{2},q_{2},(\alpha_{2}+\f{1}{q_{2}})^{-1}\}$, then,
from Lemma \ref{hzcox-2} and Definition \ref{cvx-2},
it follows that
\begin{align}\label{M2-A}
&\left\|\|M_{2}(f)\|_{\ky(\rr)} \right\|_{\ke(\rr)}\\\nonumber
&\quad=
\left\|\left(\sum_{k_{1}\in\zz} 2^{k_{1}p_{1}\alpha_{1}}
G_{k_{1},q_{1}} ^{p_{1}} \right)^{\f{q_{1}}{p_{1}}\cdot
\f{1}{q_{1}}} \right\|_{\ke(\rr)}\\\nonumber
&\quad=
\left\|\left\|\left\{2^{k_{1}q_{1}\alpha_{1}}G_{k_{1},q_{1}}
^{q_{1}}\right\}_{k_{1}\in\zz}\right\|_{l^{p_{1}/q_{1}}}
\right\|_{\dot{K}_{q_{2}/q_{1}}^{q_{1}\alpha_{2},p_{2}/q_{1}}
(\rr)}^{\f{1}{q_{1}}}.\nonumber
\end{align}
By this, $q_{1}\in (1,\rho_{2})$, and Lemma \ref{Eass-2}, we find that
\begin{align}\label{M3-A}
\left\|\|M_{2}(f)\|_{\ky(\rr)} \right\|_{\ke(\rr)}^{q_{1}}=
\sup_{\phi}\left\{
\int_{\rr}\sum_{k_{1}\in\zz}2^{k_{1}q_{1}\alpha_{1}}
[G_{k_{1},q_{1}}(x_{2})]^{q_{1}}|\phi_{k_{1}}
(x_{2})|\,dx_{2} \right\},
\end{align}
where the supremum is taken over all $\phi:=\{\phi_{k_{1}(\cdot)}
\}_{k_{1}\in\zz}\in (\dot{K}_{(q_{2}/q_{1})'}^{-q_{1}
\alpha_{2},(p_{2}/q_{1})'}(\rr),l^{(p_{1}/q_{1})'})$
satisfying
$\|\phi\|_{(\dot{K}_{(q_{2}/q_{1})'}^{-q_{1}\alpha_{2},
(p_{2}/q_{1})'}(\rr),l^{(p_{1}/q_{1})'})}=1$.
From this, the Tonelli theorem, the definition of
$G_{k_{1},q_{1}}$,
Lemmas \ref{stn} and \ref{mhr},
\cite[p.\,483,\ Corollary 4.5]{Ifs}, and $q_{1}\in (1,\rho_{2})$,
we deduce that
\begin{align*}
&\left\|\|M_{2}(f)\|_{\ky(\rr)} \right\|_{\ke(\rr)}^{q_{1}}\\
&\quad=\sup_{\phi}\left\{
\sum_{k_{1}\in\zz}2^{k_{1}q_{1}\alpha_{1}}\int_{\rr}\int
_{\rr}|M_{2}(f)(x_{1},x_{2})|^{q_{1}}|\phi_{k_{1}}(x_{2})|
\,dx_{2}\one_{R_{k_{1}}}(x_{1})\,dx_{1} \right\}\\
&\quad\ls
\sup_{\phi}\left\{
\sum_{k_{1}\in\zz}2^{k_{1}q_{1}\alpha_{1}}\int_{\rr}
\int_{\rr}|f(x_{1},x_{2})|^{q_{1}} M(\phi_{k_{1}})
(x_{2}) \,dx_{2}\one_{R_{k_{1}}}(x_{1})\,dx_{1}  \right\}\\
&\quad\sim
\sup_{\phi}\left\{
\int_{\rr}\sum_{k_{1}\in\zz}2^{k_{1}q_{1}\alpha_{1}}
\int_{\rr}|f(x_{1},x_{2})|^{q_{1}}\one_{R_{k_{1}}}
(x_{1})\,dx_{1}  M(\phi_{k_{1}})(x_{2})\,dx_{2}
\right\}\\
&\quad\ls
\left\|\left(\sum_{k_{1}\in\zz}2^{k_{1}p_{1}\alpha_{1}}
g_{k_{1},q_{1}}^{p_{1}}\right)^{q_{1}/p_{1}} \right\|
_{\dot{K}_{q_{2}/q_{1}}^{q_{1}\alpha_{2},
p_{2}/q_{1}}(\rr)}\\
&\quad\quad\times\sup_{\phi}\left\|
\left\{\sum_{k_{1}\in\zz}|M(\phi_{k_{1}})|^{(p_{1}/q_{1})'}
\right\}^{1/(p_{1}/q_{1})'} \right\|_{\dot{K}_
{(q_{2}/q_{1})'}^{-q_{1}\alpha_{2},(p_{2}/q_{1})'}
(\rr)}\\
&\quad\ls
\left\|\left(\sum_{k_{1}\in\zz}2^{k_{1}p_{1}\alpha_{1}}
g_{k_{1},q_{1}}^{p_{1}}\right)^{q_{1}/p_{1}} \right\|
_{\dot{K}_{q_{2}/q_{1}}^{q_{1}\alpha_{2},p_{2}/q_{1}}
(\rr)}\sim
\left\|\|f\|_{\ky(\rr)} \right\|^{q_{1}}_{\ke(\rr)}.
\end{align*}
Therefore, we have, for any given $p_{1},$
$p_{2},$ $q_{2}\in (1,\infty)$ and $\alpha_{2}\in (-\f{1}{q_{2}},
1-\f{1}{q_{2}})$, \eqref{m2-G} holds true for $q_{1}=\infty$ and
$\alpha_{1}\in(0,1)$, or for $q_{1}\in (1,\rho_{2})$ and $\alpha_{1}
\in (-\f{1}{q_{1}},1-\f{1}{q_{1}})$.
Then, we complete the proof of \eqref{m2-G} by an
interpolation procedure. Indeed, for any $g\in \MM(\rn)$, $i\in\zz$, and
$x_{1},$ $x_{2}\in\rr$, let
\begin{equation*}
\Gamma (g) (i,x_{1},x_{2}):=\f{1}{|B(x_{2},2^{i})|}\int_
{B(x_{2},2^{i})}g(x_{1},y_{2})\,dy_{2}.
\end{equation*}
Obviously, $\Gamma$ is a linear operator and, by \cite[p.\,421]{bag},
we find that
\begin{equation*}
\|\{\Gamma(g)(i,\cdot)\}_{i\in\zz}\|_{l^{\infty}}
\leq M_{2}(g),
\end{equation*}
while $M_{2}(g)\ls \|\{\Gamma(|g|)(i,\cdot)\}_{i\in\zz}\|
_{l^{\infty}}$. Thus, \eqref{m2-G} holds true for any
$f\in \dot{E}_{\vec{q}_{2}}^{\vec{\alpha}_{2},\vec{p}_{2}}(\rr^{2})$
if and only if
\begin{equation}\label{bpi-B}
\left\|\left\| \|\{\Gamma(f)(i,\cdot)\}_{i\in\zz}\|_{l^{\infty}}
\right\|_{\ky(\rr)} \right\|_{\ke(\rr)}\ls \left\|\|f\|_
{\ky(\rr)}\right\|_{\ke(\rr)}
\end{equation}
holds true for any $f\in \dot{E}_{\vec{q}_{2}}^{\vec{\alpha}_{2},
\vec{p}_{2}}(\rr^{2})$.
To prove \eqref{bpi-B} holds true for any $f\in H(\rr^{2})$, from
Theorem \ref{threeL}, it follows that we only need to show that,
for any given $q_{1}\in (1,\infty)$ and $\alpha_{1}\in
(-\f{1}{q_{1}},1-\f{1}{q_{1}})$, there exist a $\theta\in[0,1]$,
a $q_{1}^{(0)}=\infty$, an $\alpha_{1}^{(0)}\in(0,1)$,
a $q_{1}^{(1)}\in (1,\rho_{2})$, and an $\alpha_{1}^{(1)}\in
(-\f{1}{q_{1}^{(1)}},1-\f{1}{q_{1}^{(1)}})$ such that
\begin{equation}\label{q1}
\f{1}{q_{1}}=\f{1-\theta}{q_{1}^{(0)}}+\f{\theta}{q_{1}^{(1)}}
\end{equation}
and
\begin{equation}\label{q2}
\alpha_{1}=(1-\theta)\alpha_{1}^{(0)}+\theta \alpha_{1}^{(1)}.
\end{equation}
Indeed, we consider the following two cases on $q_{1}$.
If $q_{1}\in (1,\rho_{2})$, then let $\theta:=1$ and $q_{1}^{(1)}=q_{1}$.
In this case, \eqref{q1} holds true. If
$q_{1}\in [\rho_{2},\infty)$, then let $\theta=q_{1}^{(1)}/q_{1}$ for any
$q_{1}^{(1)}\in (1,\rho_{2})$.
In this case, \eqref{q1} holds true. Thus, for any given $q_{1}\in (1,\infty)$,
there exist a $\theta\in[0,1]$, a $q_{1}^{(0)}=\infty$,
and a $q_{1}^{(1)}\in(1,\rho_{2})$ such that \eqref{q1} holds true.
Moreover, for any given $q_{1}\in (1,\infty)$ and $\alpha_{1}\in
(-\f{1}{q_{1}},1-\f{1}{q_{1}})$, and aforementioned fixed $\theta$ and $q_{1}^{(1)}$,
there exist an $\alpha_{1}^{(0)}:=\alpha_{1}+\f{1}{q_{1}}\in(0,1)$ and an
$\alpha_{1}^{(1)}:=\alpha_{1}+\f{1}{q_{1}}-\f{1}{q_{1}^{(1)}}\in
(-\f{1}{q_{1}^{(1)}},1-\f{1}{q_{1}^{(1)}})$ such that
\begin{equation*}
(1-\theta)\alpha_{1}^{(0)}+\theta\alpha_{1}^{(1)}=(1-\theta)\left(
\alpha_{1}+\f{1}{q_{1}} \right)+\theta\left(\alpha_{1}+\f{1}{q_{1}}-
\f{1}{q_{1}^{(1)}} \right)=\alpha_{1}.
\end{equation*}
Thus, \eqref{q2} holds true. By this and Theorem \ref{threeL},
we conclude that, for any $q_{1}\in (1,\infty)$ and $\alpha_{1}
\in (-\f{1}{q_{1}},1-\f{1}{q_{1}})$, \eqref{bpi-B}
holds true for any $f\in H(\rr^{2})$.
Then, via Proposition \ref{E-dense}, we conclude that, for any
$f\in \dot{E}_{\vec{q}_{2}}^{\vec{\alpha}_{2},\vec
{p}_{2}}(\rr^{2})$, there exists
a sequence $\{f_{k} \}_{k\in\nn}\subset H(\rr^{2})$
such that $\|f-f_{k}\|_{\dot{E}_{\vec{q}_{2}}^{\vec{\alpha}_
{2},\vec{p}_{2}}(\rr^{2})}\to 0$ as
$k\to \infty$.
From this, Proposition \ref{bqb-2}(ii),
and Definition \ref{ballqB}(vi), we deduce that, as $k\to \infty$, for any
$A_{m}\subset \rr^{2}$ with $m\in\nn$,
\begin{equation*}
\int_{A_{m}}|f(x)-f_{k}(x)|\,dx
\leq  C_{(m)}\|f-f_{k}\|_{\dot{E}_{\vec{q}_{2}}^{\vec
{\alpha}_{2},\vec{p}_{2}}(\rr^{2})}\rightarrow 0,
\end{equation*}
where $A_{m}$ is as in Lemma \ref{Am}. Thus, for any $m\in\nn$,
$f_{k}\one_{A_{m}}\rightarrow f\one_{A_{m}}$ almost everywhere
as $k\to\infty$. By this, the
Fatou lemma, Proposition \ref{E-Fatou}, \eqref{bpi-B} with
$|f_{k}|\one_{A_{m}}\in H(\rr^{2})$, and $\|f-f_{k}\|_
{\dot{E}_{\vec{q}_{2}}^{\vec{\alpha}_{2},\vec{p}_{2}}
(\rr^{2})}\rightarrow 0$ as $k\to\infty$, we find that, for any $m\in\nn$,
\begin{align*}
&\|\Gamma(|f|\one_{A_{m}})\|_{(\dot{E}_{\vec{q}_{2}}^
{\vec{\alpha}_{2},\vec{p}_{2}}(\rr^{2}),l^{\infty})}\\
&\quad
\leq \left\|\varliminf_{k\to\infty}\Gamma(|f_{k}|
\one_{A_{m}}) \right\|_{(\dot{E}_{\vec{q}_{2}}^{
\vec{\alpha}_{2},\vec{p}_{2}}(\rr^{2}),
l^{\infty})}
\leq \left\|\varliminf_{k\to\infty}\|\Gamma(|f_{k}|
\one_{A_{m}})\|_{l^{\infty}} \right\|_{\dot{E}_
{\vec{q}_{2}}^{\vec{\alpha}_{2},\vec{p}_{2}}
(\rr^{2})}\\
&\quad\leq \varliminf_{k\to\infty} \left\|\Gamma(
|f_{k}|\one_{A_{m}}) \right\|_{(\dot{E}_{\vec{q}
_{2}}^{\vec{\alpha}_{2},\vec{p}_{2}}(
\rr^{2}),l^{\infty})}\ls
\lim_{k\to\infty} \||f_{k}|\one_{A_{m}}\|_{\dot{E}
_{\vec{q}_{2}}^{\vec{\alpha}_{2},\vec{p}_{2}}
(\rr^{2})}\\
&\quad\sim \||f|\one_{A_{m}}\|_{\dot{E}_{\vec{q}_{2}}^
{\vec{\alpha}_{2},\vec{p}_{2}}(\rr^{2})},
\end{align*}
which, together with the monotone convergence theorem and
Proposition \ref{ballqfs-3}, further implies that
\begin{align*}
&\|\Gamma(f)\|_{(\dot{E}_{\vec{q}_{2}}^{\vec{\alpha}
_{2},\vec{p}_{2}}(\rr^{2}),l^{\infty})}\\
&\quad\leq
\|\Gamma(|f|)\|_{(\dot{E}_{\vec{q}_{2}}^{\vec{\alpha}
_{2},\vec{p}_{2}}(\rr^{2}),l^{\infty})}=\left\|
\lim_{m\to\infty}\Gamma(|f|\one_{A_{m}})\right\|_{(\dot{E}
_{\vec{q}_{2}}^{\vec{\alpha}_{2},\vec{p}_{2}}
(\rr^{2}),l^{\infty})}\\
&\quad=\left\|\lim_{m\to\infty}\|
\Gamma(|f|\one_{A_{m}})\|_{l^{\infty}}\right\|_
{\dot{E}_{\vec{q}
_{2}}^{\vec{\alpha}_{2},\vec{p}_{2}}(\rr^{2})}
=\lim_{m\to\infty}\left\|\Gamma(|f|\one_{A_{m}})\right\|_
{(\dot{E}_{\vec{q}_{2}}^{\vec{\alpha}_{2},\vec{p}
_{2}}(\rr^{2}),l^{\infty})}\\
&\quad\ls \lim_{m\to\infty}\left\|
|f|\one_{A_{m}}\right\|_{\dot{E}_
{\vec{q}_{2}}^{\vec{\alpha}_{2},\vec{p}_{2}}(\rr^{2})}
\sim\|f\|_{\dot{E}_{\vec{q}_{2}}^{\vec{\alpha}_{2},
\vec{p}_{2}}(\rr^{2})}.
\end{align*}
Thus, \eqref{bpi-B} holds true for any $f\in \dot{E}_
{\vec{q}_{2}}^{\vec{\alpha}_{2},\vec{p}_{2}}(\rr^{2})$, and
hence \eqref{m2-G} holds true.

Now, assume that \eqref{ext} holds true for $n:=m,$
$m\in\nn$, and $m\geq 2$, namely, for any $\vec{p}_{m},$
$\vec{q}_{m}\in (1,\infty)^{m}$ and $\alpha_{i}\in
(-\f{1}{q_{i}},1-\f{1}{q_{i}})$ with $i\in\{1,\ldots,m \}$,
\begin{equation}\label{idas}
\|M_{m}(f)\|_{\dot{E}_{\vec{q}_{m}}^{\vec{\alpha}_{m},
\vec{p}_{m}}(\rr^{m})} \ls \|f\|_{\dot{E}_{\vec{q}_{m}}
^{\vec{\alpha}_{m},\vec{p}_{m}}(\rr^{m})}.
\end{equation}
To complete the proof of the present theorem, it still needs
to show that, for any $\vec{p}_{m+1}$, $\vec{q}
_{m+1}\in (1,\infty)^{m+1}$ and $\alpha_{i}\in(-\f{1}{q_{i}},
1-\f{1}{q_{i}})$ with $i\in\{1,\ldots,m+1 \}$,
\begin{equation}\label{m+1-G}
\|M_{m+1}(f)\|_{\dot{E}_{\vec{q}_{m+1}}^{\vec{\alpha}_{m+1},
\vec{p}_{m+1}}(\rr^{m+1})} \ls \|f\|_{\dot{E}_{\vec{q}_
{m+1}}^{\vec{\alpha}_{m+1},\vec{p}_{m+1}}(\rr^{m+1})}.
\end{equation}
To achieve it, for any given $k_{1}\in\zz$ and
$q_{1}\in[1,\infty],$ and for any $x_{2},\ldots,x_{m+1}\in
\rr$, let $$h_{k_{1},q_{1}}(x_{2},\ldots,x_{m+1}):=\|
f(\cdot,x_{2},\ldots,x_{m+1})\one_{R_{k_{1}}}(\cdot)\|_
{L^{q_{1}}(\rr)}$$
and
$$H_{k_{1},q_{1}}(x_{2},\ldots,x_{m+1}):=\|M_{m+1}(f)
(\cdot,x_{2},\ldots,x_{m+1})\one_{R_{k_{1}}}(\cdot)\|_
{L^{q_{1}}(\rr)}.$$
Let $\vec{\beta}_{m}:=(\alpha_{2},\ldots,\alpha_{m+1})$,
$\vec{s}_{m}:=(p_{2},\ldots,p_{m+1})$, and $\vec{t}_{m}:=
(q_{2},\ldots,q_{m+1})$. For any given $\vec{p}_{m+1}\in (1,
\infty)^{m+1}$, and $q_{j}\in (1,\infty)$ and $\alpha_{j}\in
(-\f{1}{q_{j}},1-\f{1}{q_{j}})$ with $j\in \{2,\ldots,m+1 \}$,
we consider the following two cases on both $q_{1}$ and $\alpha_{1}$.
If $q_{1}=\infty$ and $\alpha_{1}\in (0,1)$, then, similarly to
the estimation of \eqref{kl-A}, we have, for any $k_{1}\in\zz$ and
$x_{2},\ldots,x_{m+1}\in\rr$,
\begin{align}\label{mk-if}
H_{k_{1},\infty}(x_{2},\ldots,x_{m+1}):=&\,\|M_{m+1}(f)(\cdot,
x_{2},\ldots,x_{m+1})\one_{R_{k_{1}}}(\cdot)\|_{L^{\infty}
(\rr)}\\
\ls&\, M_{m}(h_{k_{1},q_{1}})(x_{2},\ldots,x_{m+1}).\nonumber
\end{align}
Let
\begin{equation*}
\widetilde{\rho}_{m+1}:=\min\left\{p_{2},\ldots,p_{m+1},q_{2},\ldots,q_{m+1},
\left(\alpha_{2}+\f{1}{q_{2}}\right)^{-1},\ldots,\left(\alpha_
{m+1}+\f{1}{q_{m+1}}\right)^{-1} \right\}.
\end{equation*}
Then, from Definition \ref{chz}, \eqref{mk-if},  \eqref{idas},
Proposition \ref{bqb-2}(ii), and Lemma \ref{ax-Fs} with $r:=p_{1}$,
$p\in(1,\widetilde{\rho}_{m+1})$, and
$X:=\dot{E}_{\vec{t}_{m}}^{\vec{\beta}_{m},\vec{s}_{m}}(\rr^{m})$,
it follows that
\begin{align*}
&\left\|M_{m+1}(f) \right\|_{\dot{E}_{\vec{q}_{m+1}}^{\vec{\alpha}
_{m+1},\vec{p}_{m+1}}(\rr^{m+1})}\\
&\quad=
\left\|\left(\sum_{k_{1}\in\zz}2^{k_{1}p_{1}\alpha_{1}}H_
{k_{1},\infty}^{p_{1}}\right)^{\f{1}{p_{1}}} \right\|_{\dot{E}_
{\vec{t}_{m}}^{\vec{\beta}_{m},\vec{s}_{m}}(\rr^{m})}\\
&\quad\leq
\left\|\left\{\sum_{k_{1}\in\zz} [M_{m}(2^{k_{1}\alpha_{1}}h_
{k_{1},\infty})]^{p_{1}}\right\}^{1/p_{1}} \right\|_{\dot{E}_
{\vec{t}_{m}}^{\vec{\beta}_{m},\vec{s}_{m}}(\rr^{m})}\\
&\quad\ls
\left\|\left(\sum_{k_{1}\in\zz}2^{k_{1}p_{1}\alpha_{1}}h_{k_{1},
\infty}^{p_{1}}\right)^{1/p_{1}} \right\|_{\dot{E}_{\vec{t}_
{m}}^{\vec{\beta}_{m},\vec{s}_{m}}(\rr^{m})}\sim\left\|f
\right\|_{\dot{E}_{\vec{q}_{m+1}}^{\vec{\alpha}_{m+1},\vec{p}
_{m+1}}(\rr^{m+1})}.
\end{align*}
If $q_{1}\in (1,\rho_{m+1})$ and $\alpha_{1}\in(-\f{1}{q_{1}},
1-\f{1}{q_{1}})$, where $\rho_{m+1}:=\min\{p_{1},\widetilde{\rho}_{m+1} \}$,
then, by Lemma \ref{Eass-2}, similarly to the estimations of
both \eqref{M2-A} and \eqref{M3-A}, we have
\begin{align*}
&\left\|M_{m+1}(f) \right\|_{\dot{E}_{\vec{q}_{m+1}}^{\vec{\alpha}_{m+1},
\vec{p}_{m+1}}(\rr^{m+1})}^{q_{1}}\\
&\quad= \sup_{\psi}\left\{
\int_{\rr^{m}}\sum_{k_{1}\in\zz}2^{k_{1}q_{1}\alpha_{1}}\int_{\rr}
\left[M_{m+1}(f)(x_{1},\ldots,x_{m+1})\right]^{q_{1}}\one_{R_{k_{1}}}
(x_{1})\,dx_{1}\right.\\
&\qquad\times |\psi_{k_{1}}(x_{2},\ldots,x_{m+1})|\,dx_{2}\cdots
\,dx_{m+1} \Bigg\},
\end{align*}
where the supremum is taken over all $\psi:=\{\psi_{k_{1}} \}_
{k_{1}\in\zz}\in (\dot{E}_{(\vec{t}_{m}/q_{1})'}^{-q_{1}\vec{
\beta}_{m},(\vec{s}_{m}/q_{1})'}(\rr^{m}),
l^{(p_{1}/q_{1})'})$
satisfying
$\|\psi\|_{(\dot{E}_{(\vec{t}_{m}/q_{1})'}^{-q_{1}\vec{\beta}
_{m},(\vec{s}_{m}/q_{1})'}(\rr^{m}),
l^{(p_{1}/q_{1})'})}=1$.
Let
\begin{equation*}
\eta_{m+1}:=\min\left\{\left(\f{p_{2}}{q_{1}}\right)',
\ldots,\left(\f{p_{m+1}}{q_{1}}\right)',\left(\f{q_{2}}{q_{1}}\right)',
\ldots,\left(\f{q_{m+1}}{q_{1}}\right)', \gamma_{m+1} \right\},
\end{equation*}
where
\begin{equation*}
\gamma_{m+1}:=\min\left\{\f{1}{1-q_{1}(\alpha_{2}+\f{1}{q_{2}})},
\ldots,\f{1}{1-q_{1}(\alpha_{m+1}+\f{1}{q_{m+1}})} \right\}.
\end{equation*}
Then, from \eqref{idas}, Lemma \ref{hzcox-2}, Proposition \ref{bqb-2}(ii),
and Corollary \ref{Eass},
we deduce that, for any $p\in(1,\eta_{m+1})$,
$[\dot{E}
_{(\vec{t}_{m}/q_{1})'}^{-q_{1}\vec{\beta}_{m},(\vec
{s}_{m}/q_{1})'}(\rr^{m})]^{1/p}$ is a ball Banach function space
and $M_{m}$ is bounded on $\{[\dot{E}
_{(\vec{t}_{m}/q_{1})'}^{-q_{1}\vec{\beta}_{m},(\vec
{s}_{m}/q_{1})'}(\rr^{m})]^{1/p}\}'$.
Using this, the Tonelli theorem, Proposition
\ref{bqb-2}(ii),
Lemmas \ref{stn} and \ref{mhr}, Lemma \ref{ax-Fs} with
$r:=(p_{1}/q_{1})'$,
$p\in(1,\eta_{m+1})$, and $X:=\dot{E}
_{(\vec{t}_{m}/q_{1})'}^{-q_{1}\vec{\beta}_{m},(\vec
{s}_{m}/q_{1})'}(\rr^{m})$, and the
assumption
that $q_{1}\in (1,\rho_{m+1})$, we conclude that
\begin{align*}
&\left\|M_{m+1}(f) \right\|_{\dot{E}_{\vec{q}_{m+1}}^{\vec
{\alpha}_{m+1},\vec{p}_{m+1}}(\rr^{m+1})}^{q_{1}}\\
&\quad= \sup_{\psi}\left\{\int_{\rr^{m-1}}
\sum_{k_{1}\in\zz}2^{k_{1}q_{1}\alpha_{1}}\int_{\rr}\int_{\rr}\left[
M_{m+1}(f)(x_{1},\ldots,x_{m+1})\right]^{q_{1}}\right.\\
&\quad\quad\times|\psi_{k_{1}}(x_{2},\ldots,x_{m+1})|\,dx_{m+1}
\one_{R_{k_{1}}}(x_{1})\,dx_{1}\,dx_{2}\cdots\,dx_{m} \Bigg\}\\
&\quad\ls\sup_{\psi}\left\{\int_{\rr^{m-1}}
\sum_{k_{1}\in\zz}2^{k_{1}q_{1}\alpha_{1}}\int_{\rr}\int_{\rr}
|f(x_{1},\ldots,x_{m+1})|^{q_{1}}\right.\\
&\quad\quad\times\left[M_{m}(\psi_{k_{1}})(x_{2},\ldots,x_{m+1})\right]
\,dx_{m+1}\one_{R_{k_{1}}}(x_{1})\,dx_{1}\,dx_{2}\cdots\,
dx_{m} \Bigg\}\\
&\quad\sim\sup_{\psi}\left\{\int_{\rr^{m}}
\sum_{k_{1}\in\zz}2^{k_{1}q_{1}\alpha_{1}}\int_{\rr}|f(x_{1},
\ldots,x_{m+1})|^{q_{1}}\one_{R_{k_{1}}}(x_{1})\,dx_{1}\right.\\
&\quad\quad\times\left[M_{m}(\psi_{k_{1}})(x_{2},\ldots,x_{m+1})\right]
\,dx_{2}\cdots\,dx_{m+1} \Bigg\}\\
&\quad\ls
\left\|\left\|\left\{2^{k_{1}q_{1}\alpha_{1}}h_{k_{1},q_{1}
}^{q_{1}}\right\}_{k_{1}\in\zz} \right\|_{l^{p_{1}/q_{1}}}
\right\|_{\dot{E}_{\vec{t}_{m}/q_{1}}^{q_{1}\vec{\beta}_{m},
\vec{s}_{m}/q_{1}}(\rr^{m})}\\
&\quad\quad\times\sup_{\psi}\left\|
\left[\sum_{k_{1}\in\zz}|M_{m}(\psi_{k_{1}})|^{(p_{1}/
q_{1})'} \right]^{1/(p_{1}/q_{1})'} \right\|_{\dot{E}
_{(\vec{t}_{m}/q_{1})'}^{-q_{1}\vec{\beta}_{m},(\vec
{s}_{m}/q_{1})'}(\rr^{m})}\\
&\quad\ls
\left\|\left\|\left\{2^{k_{1}q_{1}\alpha_{1}}h_{k_{1},q_
{1}}^{q_{1}}\right\}_{k_{1}\in\zz} \right\|_{l^{p_{1}
/q_{1}}} \right\|_{\dot{E}_{\vec{t}_{m}/q_{1}}^{
q_{1}\vec{\beta}_{m},\vec{s}_{m}/q_{1}}(\rr^{m})}
\sim
\left\|f \right\|^{q_{1}}_{\dot{E}_{\vec{q}_{m+1}}^{\vec{
\alpha}_{m+1},\vec{p}_{m+1}}(\rr^{m+1})}.
\end{align*}
Therefore, we have, for any given $\vec{p}_{m+1}
\in (1,\infty)^{m+1}$, $q_{j}\in (1,\infty)$, and $\alpha_
{j}\in (-\f{1}{q_{j}},1-\f{1}{q_{j}})$ with $j\in\{2,\ldots,
m+1 \}$, \eqref{m+1-G}
holds true for $q_{1}=\infty$ and $\alpha_{1}\in (0,1)$, or
for $q_{1}\in (1,\rho_{m+1})$ and $\alpha_{1}\in (-\f{1}{q_{1}},
1-\f{1}{q_{1}})$. Applying this and Theorem \ref{threeL},
similarly to the interpolation procedure of the case $n:=2$,
we conclude that \eqref{m+1-G} holds true for any $q_{1}
\in (1,\infty)$ and $\alpha_{1}\in (-\f{1}{q_{1}},1-\f{1}
{q_{1}})$.
This finishes the proof of Theorem \ref{exine}.
\end{proof}
As a consequence of Theorem \ref{exine}, we obtain the
boundedness of the iterated maximal operators on $\iihz(\rn)$
as follows.
\begin{theorem}\label{bound-2}
Let $\vec{p}
:=(p_{1},\ldots,p_{n})$, $\vec{q}:=
(q_{1},\ldots,q_{n})\in(0,\infty)^{n}$, $\vec{\alpha}
:=(\alpha_{1},\ldots,\alpha_{n})\in\rn$ satisfy that, for
any
$i\in\{1,\ldots,n\}$,
$\alpha_{i}\in(-\f{1}{q_{i}},1-\f{1}{q_{i}})$, and $t\in
(0,\min\{p_{-},q_{-}\})$, where $p_{-}$ and $q_{-}$ are as in
\eqref{p-}. Then there exists a positive constant $C$
such that, for any $f\in\MM(\rn)$,
\begin{equation}\label{bite}
\left\|\cm_{t}(f)\right\|_{\iihz(\rn)}\leq C\|f\|_{\iihz
(\rn)},
\end{equation}
where the positive constant $C$ is independent of $f$.
\end{theorem}
\begin{proof}
Let all the symbols be as in the present theorem.
Since
\begin{equation*}
\left\|\cm_{t}(f)\right\|_{\iihz(\rn)}=\left\|\left[
M_{n}\cdots M_{1}\left(|f|^{t}\right)\right]^{\f{1}{t}}\right\|_
{\iihz(\rn)}=
\left\|\left[ M_{n}\cdots M_{1}\left(|f|^{t}\right)\right]\right\|
^{\f{1}{t}}_{
\dot{E}^{t\vec{\alpha},\vec{p}/t}_{\vec{q}/t}(\rr^
{n})},
\end{equation*}
to complete the proof of the present theorem,
it suffices to show that \eqref{bite} holds true for
$t=1$, any given $\vec{p}$, $\vec{q}\in(1,\infty)^{n}$
and $\alpha_{i}\in(-\f{1}{q_{i}},1-\f{1}{q_{i}})$ with
$i\in \{1,\ldots,n \}$.
To this end, we rewrite \eqref{bite} as
\begin{equation}\label{teo}
\left\|\cm_{1}(f)\right\|_{\iihz(\rn)}=\left\| M_{n}\cdots
M_{1}(|f|)\right\|_{\iihz(\rn)}\ls\|f\|_{\iihz(\rn)}.
\end{equation}
Now, we prove that \eqref{teo} holds true for any given
$\vec{p},$ $\vec{q}\in(1,\infty)^{n}$ and $\alpha_{i}
\in (-\f{1}{q_{i}},1-\f{1}{q_{i}})$ with $i\in\{1,\ldots,n\}$
and we do this by induction on $n$. If $n=1$, then $f\in
\dot{K}^{\alpha_{1},p_{1}
}_{q_{1}}(\rr)$
and the desired inequality is obtained by the
boundedness of the centered
Hardy--Littlewood maximal operator on the classical
Herz space; see, for instance, \cite[p.\,488,\ Corollary 2.1]{ly96}
or \cite[p.131, Theorem
5.1.1 and Remark 5.1.3]{HZ}.

Suppose that \eqref{teo} holds true for $n-1$ with some
fixed $n\in\nn$, namely, for any given $\vec{p}_{n-1}:
=(p_{1}, \ldots, p_{n-1})$,
$\vec{q}_{n-1}:=(q_{1},\ldots,q_{n-1})\in(1,\infty)^
{n-1}$ and $\vec{\alpha}_{n-1}:=(
\alpha_{1},\ldots,\alpha_{n-1})\in\rr^{n-1}$ satisfying that,
for any
$i\in\{1,\ldots,n-1\}$, $\alpha_{i}\in(-\f{1}{q_{i}},
1-\f{1}{q_{i}})$, and for any
$f\in\MM(\rr^{n-1})$, we have
\begin{equation}\label{indc2}
\left\| M_{n-1}\cdots M_{1}(|f|)\right\|_{\nuiihz}\ls
\|f\|_{\nuiihz}.
\end{equation}
By Theorem \ref{exine} and \eqref{indc2}, we conclude that
\begin{align*}
&\left\|M_{n}M_{n-1}\cdots M_{1}(f)\right\|_{\iihz
(\rn)}\\
&\quad=\left\|M_{n}[M_{n-1}\cdots M_{1}(f)]\right
\|_{\iihz(\rn)}
\ls\left\|M_{n-1}\cdots M_{1}(f)\right\|_{\iihz(\rn)}\\
&\quad=\bigg\|\left\|M_{n-1}\cdots M_{1}(f)\right
\|_{\nuiihz}\bigg\|_{\kn}\\
&\quad\ls\left\|\left\|f\right\|_{\nuiihz}\right\|_{\kn}
=\|f\|_{\iihz(\rn)}.
\end{align*}
This finishes the proof of Theorem \ref{bound-2}.
\end{proof}
The following conclusion is a direct consequence of
both Theorem \ref{bound-2} and Remark \ref{max-cp}; the
details are omitted.
\begin{corollary}\label{HL-2}
Let $\vec{p},$ $\vec{q}:=(q_{1},\ldots,q_{n})\in(1,\infty)^{n}$ and
$\vec{\alpha}:=(\alpha_{1},\ldots,\alpha_{n})\in\rn$ with $\alpha_{i}\in(-\f{1}{
q_{i}},1-\f{1}{q_{i}})$ for any $i\in\{1,\ldots
,n\}$. Then there exists
a positive constant $C$ such that, for any
$f\in \MM(\rn)$,
\begin{equation*}
\|M (f)\|_{\iihz(\rn)}\leq C\|f\|_{\iihz(\rn)}.
\end{equation*}
\end{corollary}
\begin{remark}
It is worthy to mention that, in Theorem
\ref{bound-2}, if $p_{i}=q_{i}\in(0,\infty)$ for
any $i\in\{1,\ldots,n \}$ and $\vec{\alpha}:=\textbf{0}$,
then $\iihz(\rn)=\lq(\rn)$. In this case, Theorem
\ref{bound-2} is just \cite[Theorem 1.2]{mxM}.
Moreover, if $n:=1$, then $\iihz(\rr):= \dot{K}_{q}^{\alpha,p}(\rr)$.
In this case, from \cite[p.\,488,\ Corollary 2.1]{ly96}
or \cite[p.131, Theorem
5.1.1 and Remark 5.1.3]{HZ}, we deduce that the boundedness of
$M$ on $\dot{K}_{q}^{\alpha,p}(\rr)$ still holds true for $p:=\infty$.
However, if $n\geq 2$, by \cite[p.\,283,\ Remark 3.3]{hy21}, we find that the condition
$\vec{p}\in(1,\infty)^{n}$ is sharp in Corollary \ref{HL-2}.

\end{remark}
Now, we establish the Fefferman--Stein vector-valued
maximal inequality associated with the iterated maximal
operators on $\iihz(\rn)$.
\begin{theorem}\label{FSvv-2a}
Let $\vec{p}:=(p_{1},\ldots,p_{n})$,
$\vec{q}:=(q_{1},\ldots,q_{n})\in(0,\infty)^{n},$
$\vec{\alpha}:=(\alpha_{1},\ldots,\alpha_{n})\in \rn$
satisfy that, for any $i\in \{1,\ldots,n \}$, $\alpha_{i}
\in (-\f{1}{q_{i}},1-\f{1}{q_{i}})$, $r\in(0,\infty)$,
and $t\in (0,\min\{r,p_{-},q_{-}\})$, where $p_{-}$
and $q_{-}$ are as in \eqref{p-}. Then there exists a
positive constant $C$ such that, for any $\{f_{j}\}_
{j\in\zz}\subset  \iihz(\rn)$,
\begin{equation*}
\left\|\left\{\sum_{j\in\zz}\left[\cm_{t}(f_{j})\right]^{r}\right\}
^{\f{1}{r}} \right\|_{\iihz(\rn)}\leq C\left\|\left
(\sum_{j\in\zz}|f_{j}|^{r}\right)^{\f{1}{r}} \right
\|_{\iihz(\rn)}.
\end{equation*}
\end{theorem}
\begin{proof}
Let all the symbols be as in the present theorem.
Observe that
\begin{equation*}
\left\|\left\{\sum_{j\in\zz}\left[\cm_{t}(f_{j})\right]^{r}\right
\}^{\f{1}{r}} \right\|_{\iihz(\rn)} =
\left\|\left\{\sum_{j\in\zz}\left[M_{n}\cdots M_{1}(|f_{j}|
^{t})\right]^{\f{r}{t}}\right\}^{\f{t}{r}} \right\|_{\dot{E}
^{t\vec{\alpha},\vec{p}/t}_{\vec{q}/t}(\rr^{n})}
^{\f{1}{t}}.
\end{equation*}
By this, we conclude that it suffices to show that,
for $t=1$, any given $\vec{p}$, $\vec{q}
\in (1,\infty)$, $r\in (1,\infty)$, and $\alpha_{i}\in
(-\f{1}{q_{i}},1-\f{1}{q_{i}})$ with $i\in \{1,\ldots,
n \}$, and for any $\{f_{j} \}_{j\in\zz}\subset \iihz(\rn)$,
\begin{equation}\label{FS-G}
\left\|\left\{\sum_{j\in\zz}\left[M_{n}\cdots M_{1}(f_{j})
\right]^{r} \right\}^{\f{1}{r}}\right\|_{\iihz(\rn)}
\ls \left\|\left(\sum_{j\in\zz}|f_{j}|^{r}\right)^{
\f{1}{r}} \right\|_{\iihz(\rn)}.
\end{equation}
We prove this by induction. If $n:=1$, then the
desired inequality is obtained by
\cite[p.\,483,\ Corollary 4.5]{Ifs}.
Now, we assume that \eqref{FS-G} holds true for $n:=m\in\nn$,
namely, for any given $\vec{p},$
$\vec{q}\in (1,\infty)^{m},\ r\in (1,\infty),$ and
$\alpha_{i}\in (-\f{1}{q_{i}},1-\f{1}{q_{i}})$ with
$i\in \{1,\ldots,m \}$, and for any $\{f_{j} \}_{j\in\zz}
\subset \iihz(\rr^{m})$,
\begin{equation}\label{id-fsa}
\left\|\left\{\sum_{j\in\zz}\left[M_{m}\cdots M_{1}(f_{j})\right]
^{r} \right\}^{\f{1}{r}}\right\|_{\iihz(\rr^{m})}
\ls \left\|\left(\sum_{j\in\zz}|f_{j}|^{r}\right)^
{\f{1}{r}} \right\|_{\iihz(\rr^{m})}.
\end{equation}
To complete the proof of the present theorem, it
suffices to show that \eqref{FS-G} holds true for $n:=m+1$.
Let $$\rho:=\min\left\{p_{-},q_{-},\left(\alpha_{1}+\f{1}{q_{1}}\right)^{-1},
\ldots,\left(\alpha_{m+1}+\f{1}{q_{m+1}}\right)^{-1} \right\},$$
where $p_{-}:=\min\{p_{1},\ldots,p_{m+1} \}$ and $q_{-}
:=\min\{q_{1},\ldots,q_{m+1} \}$.
By Proposition \ref{bqb-2}(ii), Lemma \ref{hzcox-2}, Corollary \ref{Eass},
and Theorem \ref{exine}, we conclude that, for any
$\theta\in (1,\rho)$, $[\iihz(\rr^{m+1})]^{1/\theta}$ is a ball Banach
function space and $M_{m+1}$
is bounded on
$\{[\iihz(\rr^{m+1})]^{1/\theta}\}'$. From this,
Proposition \ref{bqb-2}(ii), Lemma \ref{ax-Fs},
Theorem \ref{exine}, and \eqref{id-fsa}, we deduce
that
\begin{align*}
&\left\|\left\{\sum_{j\in\zz}\left[M_{m+1}\cdots M_{1}(f_{j})\right]
^{r} \right\}^{\f{1}{r}}\right\|_{\iihz(\rr^{m+1})}\\
&\quad\ls\left\|\left\{\sum_{j\in\zz}\left[M_{m}\cdots M_{1}(f_{j})\right]^{r}
\right\}^{\f{1}{r}}\right\|_{\iihz(\rr^{m+1})}
\ls\left\|\left(\sum_{j\in\zz}|f_{j}|^{r}\right)^
{\f{1}{r}} \right\|_{\iihz(\rr^{m+1})}.
\end{align*}
Thus, \eqref{FS-G} holds true for $n:=m+1$, which then completes the proof
of Theorem \ref{FSvv-2a}.
\end{proof}	

As an immediate consequence of Theorem \ref{FSvv-2a}, we
have the following conclusion; we omit the details.

\begin{corollary}\label{FSvv-2}
Let $\vec{p},$ $\vec{q}:=(q_{1},\ldots,q_{n})\in(1,\infty)^{n},$ $
\vec{\alpha}:=(\alpha_{1},\ldots,\alpha_{n})\in\rn$ with
$\alpha_{i}\in(-\f{1}{q_{i}},1-\f{1}{q_{i}})$
for any $i\in\{1,\ldots,
n \}$, and $u\in(1,\infty)$. Then, for any
$\{f_{j} \}_{j\in\zz}\subset \MM(\rn)$,
\begin{equation*}
\left\|\left\{\sum_{j\in\zz}[M(f_{j})]^{u} \right\}
^{\f{1}{u}}\right\|
_{\iihz(\rn)}\ls\left\|\left(\sum_{j\in\zz}|f_{j}|
^{u} \right)^{\f{1}{u}}\right\|
_{\iihz(\rn)},
\end{equation*}
where the implicit positive constant is independent
of $\{f_{j} \}_{j\in\zz}$.
\end{corollary}

\section{Mixed-Norm Herz--Hardy Spaces}\label{s5}

In this section, we apply all the results obtained in
Sections
\ref{s2}, \ref{s3}, and \ref{s4} to the mixed-norm Herz--Hardy spaces $H\iihz(\rn)$
(see Definition \ref{hd-2} below) and establish
various real-variable characterizations of mixed-norm Herz--Hardy
spaces.

Recall that Sawano et\ al. \cite{SHYY} introduce the Hardy space associated
a ball quasi-Banach function space
via the maximal function of Peetre type. The maximal function
of Peetre type is defined as follows.
\begin{definition}
Let $b\in (0,\infty),$ $\Phi \in \SS(\rn)$,
and
$\rr_{+}^{n+1}:= \rn\times (0,\infty)$. For any $f\in \SS'(\rn)$, the \emph{maximal
function} $M^{**}_{b}(f,\Phi)$ \emph{of Peetre
type} is defined by setting, for any $x\in \rn$,
\begin{equation*}
M^{**}_{b}(f,\Phi)(x):=\sup_{(y,t)\in \rr_{+}^{n+1}}
\frac{|(\Phi_{t}\ast f)(x-y)|}{(1+t^{-1}|y|)^{b}}.
\end{equation*}	
\end{definition}
By \cite[p.\,19,\ Definition 2.22]{SHYY}, we introduce the following Hardy
space associated with the mixed-norm Herz space $\iihz(\rn)$.
\begin{definition}\label{hd-2}
Let $\vec{p}$, $\vec{q}:=(q_{1},\ldots,q_{n})\in(0,\infty)^{n}$
and $\vec{\alpha}:=(\alpha_{1},\ldots,\alpha_{n})$ with $\alpha_{i}\in(-\f{1}{q_{i}},\infty)$ for any
$i\in\{1,\ldots,n\}$.
Assume that $\Phi\in\SS(\rn)$ satisfies $\int_{\rn}
\Phi(x)\,dx\neq 0$
and $b\in(0,\infty)$ is sufficiently large. Then the
\emph{Hardy space} $H\iihz(\rn)$ associated with
$\iihz(\rn)$ (for short, the \emph{mixed-norm Herz--Hardy space})
is defined by setting
\begin{equation*}
H\iihz(\rn):=\left\{f\in\SS'(\rn):\ \|f\|_{H\iihz(\rn)}
:=\left\|M_{b}^{**}(f,\Phi)\right\|_{\iihz(\rn)}<\infty \right\}.
\end{equation*}
\end{definition}
For any given $\theta\in(0,\infty)$, the \emph{powered Hardy--Littlewood
maximal operator} $M^{(\theta)}$ is defined by setting, for any
$f\in \MM(\rn)$ and $x\in\rn$,
\begin{equation}\label{pmax}
M^{(\theta)}(f)(x):=\left[M(|f|^{\theta})(x)\right]^{\f{1}{\theta}}.
\end{equation}

The following technical lemma plays an important role in the study
of mixed-norm Herz--Hardy spaces.
\begin{lemma}\label{assp-E-1}
Let $\vec{p}:=(p_{1},\ldots,p_{n})$, $\vec{q}:=(q_{1},\ldots,q_{n})\in(0,\infty)^{n},$ $\vec{\alpha}:=(\alpha_{1},\ldots,\alpha_{n})
\in\rn$ with
$\alpha_{i}\in(-\f{1}{q_{i}},\infty)$ for any $i\in
\{1,\ldots,n\}$, $u\in(1,\infty)$, and
$\kappa:=\min\{1,\nu\},$ where $p_{-},
\ q_{-}$
are as in \eqref{p-} and
\begin{equation}\label{nuu}
\nu:=\min\left\{p_{-},q_{-},(\alpha_{1}+1/q_{1})^{-1},
\ldots,(\alpha_{n}+1/q_{n})^{-1} \right\}.
\end{equation}
Then
\begin{enumerate}
\item[$\mathrm{(i)}$] for any given $\theta\in(0,
\kappa)$, there exists a positive
constant $C$ such that, for any
$\{f_{j} \}_{j\in\nn}\subset \MM{(\rn)}$,
\begin{equation*}
\left\|\left\{\sum_{j\in \nn}[M(f_{j})]^{u} \right\}
^{\f{1}{u}}\right\|_{[\iihz(\rn)]^{1/\theta}}
\leq
C\left\|\left(\sum_{j\in \nn}|f_{j}|^{u} \right)
^{\f{1}{u}}\right\|_{[\iihz(\rn)]^{1/\theta}};
\end{equation*}
\item[$\mathrm{(ii)}$] there exist
an $r\in(0,\nu)$
and a $t\in(0,\infty)$ satisfying
$$t\in\left(\max\left\{p_{+},q_{+},(\alpha_{1}+1/q_{1} )^{-1}
,\ldots,(\alpha_{n}+
1/q_{n})^{-1}\right\},\infty\right),$$
with $p_+:=\max\{p_1,\ldots, p_n\}$ and
$q_+:=\max\{q_1,\ldots, q_n\}$,
such that $[\iihz(\rn)]^{1/r}$ is a ball Banach
function  space
and, for any $f\in([\iihz(\rn)]^{1/r})'$,
\begin{equation*}
\left\|M^{((t/r)')}(f) \right\|_{([\iihz(\rn)]^{1/r} )'}
\ls \|f\|_{([\iihz(\rn)]^{1/r} )'},
\end{equation*}
where the implicit positive constant is independent of $f$.
\end{enumerate}
\begin{proof}
Let all the symbols be as in the present lemma.
From Lemma \ref{hzcox-2} and Corollary \ref{FSvv-2},
it follows that Lemma \ref{assp-E-1}(i) holds true.
Applying Lemma \ref{hzcox-2}, Proposition
\ref{bqb-2}(ii), Corollary \ref{Eass}, \eqref{pmax},
Definition \ref{cvx-2}, and Corollary \ref{HL-2},
we conclude that the desired conclusions of Lemma
\ref{assp-E-1}(ii) are true. This finishes the proof
of Lemma \ref{assp-E-1}.
\end{proof}
\end{lemma}

\subsection{Littlewood--Paley Characterizations\label{s5.1}}

In this subsection, we establish various Littlewood--Paley
characterizations of $H\iihz(\rn)$, including its
characterizations by means of the Lusin
area function, the Littlewood--Paley $g$-function, and the
Littlewood--Paley
$g_{\lambda}^{*}$-function.

In what follows, for any $\varphi\in \SS(\rn)$, we define
its \emph{Fourier transform}  $\widehat{\varphi}$ by setting,
for any $\xi\in\rn$,
\begin{equation*}
\widehat{\varphi}(\xi):=\int_{\rn}\varphi(x)
e^{-2\pi ix\cdot \xi}\,dx.
\end{equation*}
Moreover, for any $f\in\SS'(\rn)$, $\widehat{f}$ is defined by
setting, for any $\varphi\in\SS(\rn),$ $\langle\widehat{f},
\varphi \rangle:=\langle f,\widehat{\varphi}\rangle$; also, for
any $f\in\SS'(\rn)$ and $\varphi\in\SS(\rn)$,
$f*\varphi$ is defined by setting, for any $\psi\in\SS(\rn)$,
$$\langle f*\varphi,\psi\rangle:=\langle f,\wt{\varphi}
*\psi\rangle,$$
where $\wt{\varphi}(x):=
\varphi(-x)$. In addition, $f\in\SS'(\rn)$ is said to
\emph{vanish weakly at infinity} if, for any $\phi\in\SS(\rn)$
and $\phi_{t}(x):=t^{-n}\phi(x/t)$ with $t\in(0,\infty)$
and $x\in\rn$, $f*\phi_{t}\rightarrow 0$ in $\SS'(\rn)$
as $t\rightarrow \infty$.

Now, we recall the definitions of various Littlewood--Paley
functions which were introduced by \cite{zywc} under some
weaker conditions (see \cite[Remark 4.3]{zywc}).
\begin{definition}
Let $\varphi\in\SS(\rn)$ satisfy $\widehat{\varphi}
(\textbf{0})=0$ and assume that, for any $\xi\in\rn
\setminus \{\textbf{0}\}$, there exists an $h\in(0,\infty)$
such that $\widehat{\varphi}(h\xi)\neq 0$. For any
distribution $f\in\SS'(\rn)$, the \emph{Lusin area function}
$S(f)$ and
the \emph{Littlewood--Paley} $g^{*}_{\lambda}$
\emph{-function} $g^{*}_{\lambda}(f)$ with any given
$\lambda\in(0,\infty)$ are defined,
respectively, by setting, for any $x\in\rn$,
\begin{equation*}
S(f)(x):=\left\{\int_{\Gamma(x)}|\varphi_{t}*f(y)|^{2}\,
\f{dy\,dt}{t^{n+1}} \right\}^{\f{1}{2}}
\end{equation*}
and
\begin{equation*}
g^{*}_{\lambda}(f)(x):=\left[\int_{0}^{\infty}\int_{\rn}
\left( \f{t}{t+|x-y|}\right)^{\lambda n}|\varphi_{t}*f(y)
|^{2}\,\f{dy\,dt}{t^{n+1}} \right]^{\f{1}{2}},
\end{equation*}
where, for any $x\in\rn$, $\Gamma(x):=\{(y,t)\in\rr_{+}^{n+1}:
\ |x-y|<t \}$ and, $\varphi_{t}(x):=t^{-n}\varphi(x/t)$
with $t\in(0,\infty)$ and $x\in\rn$.
\end{definition}
\begin{definition}
Let $\varphi\in\SS(\rn)$ satisfy $\widehat{\varphi}
(\textbf{0})=0$ and assume
that, for any $x\in\rn\setminus\{\textbf{0}\}$,
there exists a $j\in\zz$ such that $\widehat{\varphi}
(2^{j}x)\neq 0$. For any $f\in\SS'(\rn)$, the
\emph{Littlewood--Paley} $g$\emph{-function}
$g(f)$ is defined by setting, for any $x\in\rn$,
\begin{equation*}
g(f)(x):=\left[\int_{0}^{\infty}|f*\varphi_{t}(x)
|^{2}\,\f{dt}{t} \right]^{\f{1}{2}}.
\end{equation*}
\end{definition}

Now, we characterize mixed-norm Herz--Hardy spaces
$H\iihz(\rn)$, respectively, via the Lusin area function, the
Littlewood--Paley $g$-function, and the
Littlewood--Paley $g_{\lambda}^{*}$-function as follows.
\begin{theorem}\label{LP-2}
Let $\vec{p}:=(p_{1},\ldots,p_{n})$, $\vec{q}:=(q_{1},\ldots,q_{n})\in(0,
\infty)^{n},$ $\vec{\alpha}:=(\alpha_{1},\ldots,\alpha_{n})\in\rn$ with $\alpha_{i}
\in(-\f{1}{q_{i}},\infty)$
for any $i\in\{1,\ldots,n\}$, and $\lambda\in
(\max\{1,2/\nu \},\infty)$, where $\nu$ is as in \eqref{nuu}. Then the
following assertions are mutually equivalent.
\begin{enumerate}
\item[$\mathrm{(i)}$] $f\in H\iihz(\rn)$.
\item[$\mathrm{(ii)}$] $f\in\SS'(\rn)$, $f$
vanishes weakly at infinity,
and $\|S(f)\|_{\iihz(\rn)}<\infty$. In this
case, for any $f\in H\iihz(\rn)$,
\begin{equation*}
\|f\|_{H\iihz(\rn)}\sim \|S(f)\|_{\iihz(\rn)}.
\end{equation*}
\item[$\mathrm{(iii)}$] $f\in \SS'(\rn)$, $f$
vanishes weakly at infinity, and
$\|g(f)\|_{\iihz(\rn)}<\infty$. In this case,
for any $f\in H\iihz(\rn)$,
\begin{equation*}
\|f\|_{H\iihz(\rn)}\sim \|g(f)\|_{\iihz(\rn)}.
\end{equation*}
\item[$\mathrm{(iv)}$]  $f\in\SS'(\rn)$, $f$
vanishes weakly at infinity,
and $\|g_{\lambda}^{*}(f)\|_{\iihz(\rn)}<\infty$.
Furthermore, for
any $f\in H\iihz(\rn)$,
\begin{equation*}
\|f\|_{H\iihz(\rn)}\sim \|g_{\lambda}^{*}(f)
\|_{\iihz(\rn)}.
\end{equation*}
\end{enumerate}
Here all of the positive equivalence constants
are independent of $f$.
\end{theorem}
\begin{proof}
Let all the symbols be as in the present theorem.
By Proposition \ref{bqb-2}(i) and Lemma \ref{assp-E-1},
we conclude that all the assumptions
of \cite[Theorem 4.9]{zywc} with $X:=\iihz(\rn)$ are satisfied. Thus,
by \cite[Theorem 4.9]{zywc} with $X:=\iihz(\rn)$, we find that
the equivalence of both (i) and (ii) holds true.

From \eqref{pmax}, Definition
\ref{cvx-2}, Corollary \ref{FSvv-2}, Lemmas
\ref{assp-E-1}(ii) and
\ref{hzcox-2}, we deduce that all the assumptions
of \cite[Theorem 4.13]{zywc} with $X:=\iihz(\rn)$,
$s\in(0,\kappa)$ and $\theta\in(0,\min\{s,
\f{s^{2}}{2}\})$, where $\kappa$ is as in Lemma
\ref{assp-E-1}, are satisfied. Thus,
from this and \cite[Theorem 4.13]{zywc} with $X:=\iihz(\rn)$,
we infer that (i) is equivalent
to (iii).

Finally, by Lemmas \ref{assp-E-1} and \ref{hzcox-2}, and
Proposition \ref{bqb-2}(ii), it follows that all the assumptions of
\cite[Theorem 4.11]{zywc} with $X:=\iihz(\rn)$ and $\lambda\in
(\max\{1,2/\nu \},\infty)$ are satisfied, which implies that
(i) is equivalent to (iv). This finishes the proof
of Theorem \ref{LP-2}.
\end{proof}
\begin{remark}
In Theorem \ref{LP-2}, if $p_{i}=q_{i}\in (0,\infty)$
for any $i\in\{1,\ldots,n \}$, and $\vec{\alpha}:=
\textbf{0}$, then $\iihz(\rn)=\lq(\rn)$ and, in this case,
Theorem \ref{LP-2} is just \cite[p.\,25,\ Theorems
5.8 and 5.10]{zywc}; if $n:=1$,
then $\iihz(\rr)=\dot{K}_{q}^{\alpha,p}(\rr)$ and, in this case, to the best of our knowledge,
the conclusions of Theorem \ref{LP-2} are also new.
\end{remark}

\subsection{Boundedness of Calder\'{o}n--Zygmund Operators\label{s5.2}}

In this subsection, we obtain the boundedness of
Calder\'{o}n--Zygmund
operators on the mixed-norm Herz--Hardy spaces. Let us recall
the definition of convolutional
Calder\'{o}n--Zygmund operators (see, for instance,
\cite[Section 5.3.2]{LK}).

\begin{definition}
For any given $\delta\in(0,1)$, a \emph{convolutional
$\delta$-type
Calder\'{o}n--Zygmund operator} $T$ is a
linear bounded operator on
$L^{2}(\rn)$ with the kernel $K\in\SS'(\rn)$
coinciding with a locally
integrable function on $\rn\setminus \{\textbf{0}\}$
and satisfying:
\begin{enumerate}
\item[$\mathrm{(i)}$] there exists a positive constant $C$
such that, for any
$x,\ y\in\rn$ with $|x|>2|y|$,
\begin{equation*}
|K(x-y)-K(x)|\leq C\f{|y|^{\delta}}{|x|^{n+\delta}};
\end{equation*}
\item[$\mathrm{(ii)}$] for any $f\in L^{2}(\rn)$ and $x\in\rn,$
$T(f)(x)=\text{p.\,v.\,}K*f(x)$, where
p.\,v. denotes the \emph{Cauchy principal value}.
\end{enumerate}
\end{definition}
Now, we investigate the boundedness of convolutional
$\delta$-type Calder\'{o}n--Zygmund operators on
$H\iihz(\rn)$.
\begin{theorem}\label{CZ-bd-2}
Let $\delta\in(0,1),$ $\vec{p}:=(p_{1},\ldots,p_{n})$,
$\vec{q}:=(q_{1},\ldots,q_{n})\in(\f{n}{n+\delta},\infty)^{n},\
\vec{\alpha}:=(\alpha_{1},\ldots,\alpha_{n})\in\rn$ with
$\alpha_{i}\in(-\f{1}{q_{i}},1-\f{1}{q_{i}}+\f{\delta}{n})$,
and
$T$ be a convolutional $\delta$-type
Calder\'{o}n--Zygmund
operator. Then $T$ has a unique extension on
$H\iihz(\rn)$.
Moreover, there exists a positive constant $C$
such that,
for any $f\in H\iihz(\rn)$,
\begin{equation*}
\|T(f)\|_{H\iihz(\rn)}\leq C\|f\|_{H\iihz(\rn)}.
\end{equation*}
\end{theorem}
\begin{proof}
Let all the symbols be as in Theorem \ref{CZ-bd-2}.
From Proposition \ref{bqb-2}, Proposition \ref{ab-2}, and
Lemma \ref{assp-E-1}, it follows that all the assumptions of
\cite[p.\,26,\ Theorem 3.5]{rm-wyy} with $X:=\iihz(\rn)$ are satisfied,
which, together with \cite[p.\,26,\ Theorem 3.5]{rm-wyy} with $X:=\iihz(\rn)$,
further implies the desired conclusion. This finishes
the proof of Theorem \ref{CZ-bd-2}.
\end{proof}
\begin{remark}
In Theorem \ref{CZ-bd-2}, if $p_{i}=q_{i}\in (\f{n}
{n+\delta},\infty)$ for any $i\in\{1,\ldots,n \}$, and
$\vec{\alpha}:=\textbf{0}$, then $\iihz(\rn)=\lq(\rn)$ and,
in this case, Theorem \ref{CZ-bd-2} is just
\cite[p.\,3071,\ Theorem 8.4]{ahd} with $A:=d I_{n\times n}$
for some $d\in\rr$ with $d\in(1,\infty)$;
if $n:=1,\ p\in(\f{1}{1+\delta},\infty),$ $q\in(1,\infty),$
and $\alpha\in[1-\f{1}{q},1-\f{1}{q}+\delta)$, then
$\iihz(\rr)= \dot{K}_{q}^{\alpha,p}(\rr)$ and, in this case,
Theorem \ref{CZ-bd-2} coincides with
\cite[p.\,169,\ Theorem 6.2.3]{HZ} with $n:=1$ and $p\in (\f{1}{1+\delta},\infty)$ there.
\end{remark}
In order to establish the boundedness of
$\gamma$-order Calder\'{o}n--Zygmund operators on the
mixed-norm Herz--Hardy space, we first recall some concepts.
\begin{definition}\label{d5.11}
For any given $\gamma\in(0,\infty)$, a
\emph{$\gamma$-order Calder\'{o}n--Zygmund  operator} $T$
is a linear bounded operator on $L^{2}(\rn)$ and
there exists a kernel $K$ on $(\rn\times\rn)\setminus
\{(x,x):\ x\in\rn \}$ satisfying
\begin{enumerate}
\item[$\mathrm{(i)}$]  for any $\beta:=(\beta_{1},\ldots,\beta_{n})
\in\zz_{+}^{n}$ with $|\beta|\leq \lceil\gamma
\rceil-1$, there exists a positive
constant $C$ such that, for any $x,$ $y,$ $z\in\rn$
with
$|x-y|>2|y-z|$,
\begin{equation*}
|\pa_{x}^{\beta}K(x,y)-\pa_{x}^{\beta}K(x,z)|\leq C
\f{|y-z|^{\gamma-\lceil\gamma \rceil+1}}{|x-y|^{n+\gamma}},
\end{equation*}
here and thereafter, $\pa_{x}^{\beta}:=(\pa/\pa x_{1})^
{\beta_{1}}\ldots(\pa/\pa x_{n})^{\beta_{n}}$;
\item[$\mathrm{(ii)}$] for any $f\in L^{2}(\rn)$ with compact support and
$x\notin \supp(f)$,
\begin{equation*}
T f(x)=\int_{\supp(f)}K(x,y)f(y)\,dy.
\end{equation*}
\end{enumerate}	
\end{definition}
For any $m\in\nn$, an operator $T$ is said to have the
\emph{vanishing moments up to order} $m$ if, for any $f\in L^{2}(\rn)$
with compact support satisfying that, for any $\xi:=(\xi_{1},
\ldots,\xi_{n})\in\zz_{+}^{n}$ with $|\xi|\leq m,\
\int_{\rn}x^{\xi}f(x)\,dx=0$, it holds true that
$$\int_{\rn} x^{\xi}Tf(x)\,dx=0.$$

Now, we establish the boundedness of $\gamma$-order
Calde\'{o}n--Zygmund operators on $H\iihz(\rn)$ as follows.
\begin{theorem}\label{CZo-2}
Let $\gamma\in(0,\infty),$ $ \vec{p}:=(p_{1},\ldots,p_{n}),$
$\vec{q}:=(q_{1},\ldots,q_{n})\in(\f{n}{n+\gamma},2)^{n}$,
and $\vec{\alpha}:=(\alpha_{1},\ldots,\alpha_{n})\in\rn$
with $\alpha_{i}\in(\f{1}{2}-\f{1}{q_{i}},1-\f{1}{q_{i}}+\f{\gamma}{n})$
for any $i\in\{1,\ldots,n\}$. Assume that $T$ is a
$\gamma$-order Calder\'{o}n--Zygmund operator and has
the vanishing moments up to order $\lceil \gamma \rceil-1$.
Then $T$ has a unique extension on $H\iihz(\rn)$.
Moreover, there exists a positive
constant $C$ such that, for any $f\in H\iihz(\rn)$,
\begin{equation*}
\|T(f)\|_{H\iihz(\rn)}\leq C\|f\|_{H\iihz(\rn)}.
\end{equation*}
\end{theorem}
\begin{proof}
Let all the symbols be as in the present theorem.
By Propositions \ref{bqb-2}(i) and
\ref{ab-2}, $\vec{p},$
$\vec{q}\in(\f{n}{n+\gamma},2)^{n},$ $\alpha_{i}\in(\f{1}{2}-\f{1}
{q_{i}},1-\f{1}{q_{i}}+\f{\gamma}{n})$ for any $i\in\{1,\ldots,n\}$,
Lemma \ref{assp-E-1}, and
\cite[p.\,2040,\ Remark 6.5]{wk-1}, we find that
all the assumptions of \cite[Theorem 3.11]{rm-wyy} with $X:=\iihz(\rn)$
are satisfied, which, combined with \cite[Theorem 3.11]{rm-wyy} with
$X:=\iihz(\rn)$, implies
the desired conclusion. This finishes the
proof of Theorem \ref{CZo-2}.
\end{proof}
\begin{remark}
In Theorem \ref{CZo-2}, if $p_{i}=q_{i}\in (\f{n}{n+\delta},
\infty)$ for any $i\in\{1,\ldots,n \}$, and $\vec{\alpha}:=
\textbf{0}$, then $\iihz(\rn)=\lq(\rn)$ and, in this case,
Theorem \ref{CZo-2} is just \cite[p.\,2056,\ Theorem 6.8]{alhl}
with $\vec{a}:=(1,\ldots,1)$ and $\nu:=n$ there; if $n:=1,$
$\gamma\in(0,1],$ $p\in(\f{1}{1+\gamma},2),$ $q\in(1,2)$, and
$\alpha\in[1-\f{1}{q},1-\f{1}{q}+\gamma)$, then $\iihz(\rr)
=\dot{K}_{q}^{\alpha,p}(\rr)$ and, in this case, Theorem \ref{CZo-2}
coincides with \cite[Theorem 1]{cz-ga} with $p\in (\f{1}{1+\gamma},2)$,
$q\in (1,2)$, and $\oz_{1}=\oz_{2}:=1$.
\end{remark}
\begin{remark}
In Theorem \ref{CZo-2}, if the kernel $K$ on the second variable also satisfies
Definition \ref{d5.11}(i), then the conclusions in Theorem \ref{CZo-2} hold true for any
$\vec{p},$ $\vec{q}:=(q_{1},\ldots,q_{n})\in(\f{n}{n+\gamma},\infty)^{n}$
and $\vec{\alpha}:=(\alpha_{1},\ldots,\alpha_{n})\in\rn$
with $\alpha_{i}\in(-\f{1}{q_{i}},1-\f{1}{q_{i}}+\f{\gamma}{n})$
for any $i\in\{1,\ldots,n\}$.
\end{remark}
\subsection{Further Remarks}\label{s5.3}

Besides the Littlewood--Paley characterizations and the
boundedness of Calder\'{o}n--Zygmund operators, we can also
have other real-variable characterizations on
mixed-norm Herz--Hardy spaces; to limit the length of this
article, we omit the details by only mentioning them as in
the following remark.
\begin{remark}
\begin{enumerate}
\item[$\mathrm{(i)}$] From Proposition \ref{bqb-2}(i),
Lemma \ref{hzcox-2}, and Corollary \ref{HL-2}, it follows that the corresponding
conclusions of \cite[Theorem 3.1]{SHYY} with $X:=\iihz(\rn)$ hold true,
namely, we can obtain other maximal function characterizations of
mixed-norm Herz--Hardy spaces.
\item[$\mathrm{(ii)}$] From Proposition \ref{bqb-2}(i),
Lemma \ref{hzcox-2}, and Corollary \ref{HL-2},  we deduce that all the
assumptions
of \cite[Theorem 3.4]{SHYY} with $X:=\iihz(\rn)$ are
satisfied, and hence all the corresponding conclusions
of \cite[Theorem 3.4]{SHYY} with $X:=\iihz(\rn)$ also
hold true.
\item[$\mathrm{(iii)}$] Applying Proposition
\ref{bqb-2}(i) and
Lemma \ref{assp-E-1},
we conclude that all the assumptions of
\cite[Lemma 3.1]{ft-bqb}
(see also \cite[Theorems 3.6 and 3.7]{SHYY})
with $X:=\iihz(\rn)$ are satisfied,
and hence we have the atomic characterization
of $H\iihz(\rn)$; moreover, in this case, all the
assumptions of \cite[Theorem 1.10]{YYY} and
\cite[Theorem 3.9]{SHYY} also hold true for
$H\iihz(\rn)$ and hence
obtain the finite atomic and the molecular
characterizations of $H\iihz(\rn)$.
\item[$\mathrm{(iv)}$] By
Propositions \ref{bqb-2}(i) and \ref{ab-2}, and
Lemma \ref{assp-E-1}, we conclude that all
the assumptions of
\cite[Theorem 3.14]{Abs} with $X:=\iihz(\rn)$ are satisfied, and hence
all the corresponding conclusions of \cite[Theorem 3.14]{Abs}
with $X:=\iihz(\rn)$ also hold true.
\end{enumerate}
\end{remark}
\noindent\textbf{Acknowledgements}\quad
Yirui Zhao would
like to express
his deep thanks to Professor Long Huang for some helpful
suggestions
on this article,  Dr. Yinqing Li for several beneficial
conversations on some lemmas, and Professor Suqing Wu
for some useful advice on Proposition \ref{E-lat}.

\bigskip

\noindent  Yirui Zhao, Dachun Yang (Corresponding author) and
Yangyang Zhang

\medskip

\noindent  Laboratory of Mathematics and Complex Systems
(Ministry of Education of China),
School of Mathematical Sciences, Beijing Normal University,
Beijing 100875, People's Republic of China

\smallskip

\noindent {\it E-mails}:
\texttt{yiruizhao@mail.bnu.edu.cn} (Y. Zhao)

\noindent\phantom{{\it E-mails:}}
\texttt{dcyang@bnu.edu.cn} (D. Yang)

\noindent\phantom{{\it E-mails:}}
\texttt{yangyzhang@mail.bnu.edu.cn} (Y. Zhang)
\end{document}